\documentclass[a4paper]{article}
\usepackage{amsmath}
\usepackage{amsfonts}
\usepackage{amssymb}
\usepackage{fullpage}
\usepackage{hyperref}
\newtheorem{theorem}{Theorem}

\newtheorem{corollary}{Corollary}

\newtheorem{definition}{Definition}
\newtheorem{example}{Example}

\newtheorem{lemma}{Lemma}

\newtheorem{proposition}{Proposition}
\newtheorem{remark}{Remark}

\newenvironment{proof}{\noindent {Proof:}}{$\,\hfill \Box$\smallskip}
\numberwithin{equation}{section}
\begin{document}
\author{Goncalo Oliveira \\ Imperial College London}


\date{September, 2013}

\title{Monopoles on the Bryant-Salamon $G_2$-manifolds}
\maketitle

\begin{abstract}
$G_2$-Monopoles are solutions to gauge theoretical equations on noncompact $7$-manifolds of $G_2$ holonomy. We shall study this equation on the $3$ Bryant-Salamon manifolds. We construct examples of $G_2$-monopoles on two of these manifolds, namely the total space of the bundle of anti-self-dual two forms over the $\mathbb{S}^4$ and $\mathbb{CP}^2$. These are the first nontrivial examples of $G_2$-monopoles.\\
Associated with each monopole there is a parameter $m \in \mathbb{R}^+$, known as the mass of the monopole. We prove that under a symmetry assumption, for each given $m \in \mathbb{R}^+$ there is a unique monopole with mass $m$. We also find explicit irreducible $G_2$-instantons on $\Lambda^2_-(\mathbb{S}^4)$ and on $\Lambda^2_-(\mathbb{CP}^2)$.\\
The third Bryant-Salamon $G_2$-metric lives on the spinor bundle over the $3$-sphere. In this case we produce a vanishing theorem for monopoles.
\end{abstract}


\section{Introduction}

A $G_2$-manifold is a seven dimensional manifold $X^7$ equipped with a Riemannian metric whose holonomy lies in $G_2$. Equivalently this can be encoded in a $3$ form $\phi$, which determines the metric. The condition that the holonomy is in $G_2$ then amounts to $\phi$ being both closed and coclosed. It is standard to denote $\psi=\ast \phi$ and to refer to a $G_2$-manifold as the pair $(X^7, \phi)$.
Let $G$ a compact, semisimple Lie group with Lie algebra $\mathfrak{g}$ and $P \rightarrow X$ be a principal $G$-bundle. Denote the adjoint bundle $P \times_{(Ad, G)} \mathfrak{g}$ by $\mathfrak{g}_P$ and equip it with an $Ad$-invariant metric.

\begin{definition}
A pair $(A, \Phi)$ consisting on a connection $A$ on $P$ and Higgs Field $\Phi \in \Omega^0(X, \mathfrak{g}_P)$ is said to be a $G_2$-monopole if 
\begin{equation}\label{eq}
F_A \wedge \psi = \ast \nabla_A \Phi.
\end{equation}
\end{definition}

If $(A, \Phi)$ is a $G_2$-monopole with $\nabla_A \Phi =0$, then $F_A \wedge \psi =0$ and such connections are known as $G_2$-instantons (if $\Phi \neq 0$, the holonomy of $A$ must preserve $\Phi$, so $A$ is also reducible). In fact, an integration by parts (or a maximum principle) argument shows that in the compact case these are the unique solutions to equation \ref{eq}. So, in order to study solutions of equation \ref{eq} with $\nabla_A \Phi \neq 0$ one must either admit singularities or let $X$ be noncompact. The equation \ref{eq} is invariant under the action of the gauge group $\mathcal{G}$ and one is interested in the moduli space of irreducible monopoles on $P$
\begin{equation}\label{generalmoduli}
\mathcal{M}(X,P) = \lbrace (A, \Phi \neq 0 ) \ \vert \ \text{solving \ref{eq} and $A$ irreducible} \rbrace / \mathcal{G}.
\end{equation}

Donaldson and Segal in \cite{DS} suggested that these monopoles might be related to coassociative submanifolds of $X$. These are $4$ dimensional and $\psi$-calibrated submanifolds, in particular they are volume minimizing in their homology class and are the $G_2$ analogs of special Lagrangian submanifolds in the Calabi-Yau case, respectively. There are conjectural theories due to Dominic Joyce \cite{J}, which attempt to define an invariant of a $G_2$-manifold by counting rigid, compact coassociative submanifolds. In fact, it follows from McLean's work \cite{Mc} that compact coassociative manifold $M$ deforms in a smooth moduli space of dimension $b_2^-(M)$. Hence, these are rigid when $b_2^-(M)=0$ (e.g. $M = \mathbb{S}^4, \mathbb{CP}^2$) and one could hope to count these. An alternative way to define an enumerative invariant of $G_2$-manifolds goes by counting monopoles and the idea is that this may be related to a count of coassociative submanifolds. The general expectation is that under some asymptotic regime where the mass (i.e. the asymptotic value of $ \vert \Phi \vert$) gets very large, the monopoles concentrate along some coassociative cycles whose homology class is determined by the topological type of the bundle $P$. Such a concentration phenomena is expected to be modelled on $\mathbb{R}^3$ monopoles along the fibres of the normal bundle to the coassociatives. However, other then on $\mathbb{R}^7= \mathbb{R}^4 \times \mathbb{R}^3$ where dimensional reduction gives examples by lifting monopoles on $\mathbb{R}^3$, no examples of such monopoles are known to exist and is this question of existence which is addressed in this paper. There are also similar theories on noncompact Calabi-Yau's relating solutions to monopole equations to special Lagrangian cycles. The analytic properties of the monopole equations in both these cases are work for the PhD thesis of the author. See \cite{O} for examples of monopoles on a noncompact Calabi-Yau.\\
 
Some notation needs to be introduced in order to state the main theorem \ref{Theorem} of the paper. If $(M, g_M)$ is an Einstein, self-dual $4$ manifold ($M= \mathbb{S}^4, \mathbb{CP}^2$) with positive scalar curvature, Bryant and Salamon in \cite{BS}, constructed $G_2$-metrics on $\Lambda^2_- (M)$, i.e. the total space of the bundle of anti-self-dual $2$-forms on $M$. These examples have large symmetry groups, in each case there is a compact Lie group $K$ acting on $\Lambda^2_-(M)$ with cohomogeneity $1$. In such a situation there is a notion of $K$-homogeneous principal $G$-bundle, i.e. the $K$-action on $\Lambda^2_-(M)$ lifts to the total space $P$. Moreover, the Bryant-Salamon manifolds are asymptotically conical and so the $G_2$-structure is asymptotic to a conical one $\phi_C$ on the cone $((1, + \infty)_{r} \times \Sigma, g_C = dr^2 + r^2 g_{\Sigma})$ over a Nearly K\"ahler manifold $(\Sigma^6, g_{\Sigma})$, proposition \ref{G2conestr}. Let $\rho: \Lambda^2_-(M) \rightarrow \mathbb{R}$ be the distance to the zero section and $\hat{K} \subset \Lambda^2_-(M)$ a compact set such that there is a diffeomorphism $\varphi: (1, + \infty)_{r} \times \Sigma \rightarrow X \backslash \hat{K}$ with $r \circ \varphi = \rho \vert_{X \backslash \hat{K}}$ and
$$\vert  \nabla^j (\varphi^*\phi - \phi_C ) \vert_C = O(r^{\nu-j}),$$
for some $\nu <0$ and all $j \in \mathbb{N}_0$, where $\vert \cdot \vert_C, \nabla$ denote the norm and covariant derivative in the conical metric $g_C$. Then, we shall consider a $K$-homogeneous principal bundle $P$ such that there is an isomorphism $ \varphi^* P \vert_{X \backslash \hat{K}} \cong \pi^* P_{\infty}$, where $P_{\infty}$ is a bundle over $\Sigma$ and $\pi: (1, + \infty) \times \Sigma \rightarrow \Sigma$ the projection. Moreover, we shall suppose that $\mathfrak{g}_P$ is equipped with an $Ad$-invariant inner product $h$ which is modeled on an inner product $h_{\infty}$ in $\mathfrak{g}_{P_{\infty}}$. We will use a combination of this with $g_C$ in order to measure the growth rate of sections of $\Lambda^* \otimes \mathfrak{g}_P$. For example, if $a$ denotes a section of $\Lambda^* \otimes \mathfrak{g}_P$ we shall say it has rate $\delta \in \mathbb{R}$ with derivatives if on $X \backslash \hat{K}$, $\vert \nabla^j \varphi^*  a \vert = O(r^{\delta-j})$ for all $j \in \mathbb{N}_0$; where $\vert \cdot \vert$ denotes a combination the norm $g_C$ with $h_{\infty}$ and $\nabla$ the connection obtained by twisting the Levi Civita connection of $g_C$ with $\nabla_{\infty}$.

\begin{definition}\label{moduliS}
Let $P$ be a $K$-homogeneous principal $G$-bundle as above. A monopole $(A, \Phi)$ on $P$ is said to have finite mass if there is a connection $A_{\infty}$ on $P_{\infty}$ such that $\vert \nabla^j ( \varphi^* A - A_{\infty}) \vert_C = O(r^{-1-\epsilon-j})$, for some $\epsilon >0$, all $j \in \mathbb{N}$ and 
\begin{equation}\label{mass}
m(A,\Phi)= \lim_{\rho \rightarrow \infty} \vert \Phi \vert,
\end{equation}
is well defined and finite. In this case $m(A,\Phi) \in \mathbb{R}^+$ is the mass of the monopole.\\
Let $\mathcal{G}_{inv}$ denote the $K$-invariant gauge transformations, the moduli space of finite mass, invariant monopoles on $P \rightarrow \Lambda^2_-(M)$ is defined as
\begin{equation}\label{invariantmoduli}
\mathcal{M}_{inv}(\Lambda^2_-(M),P) = \lbrace \text{finite mass, $K$-invariant $(A, \Phi)$ solving \ref{eq} and $\nabla_A \Phi \neq 0$} \rbrace / \mathcal{G}_{inv}.
\end{equation}
\end{definition}

The monopole equations used here are inspired by the monopole equations in $3$ dimensions. In the Euclidean $\mathbb{R}^3$ and for structure group $SU(2)$, there is a unique mass $1$ and spherically symmetric solution; this is known as the BPS monopole \cite{BPS} and we shall denote it by $(A^{BPS}, \Phi^{BPS})$. Moreover, for structure group $\mathbb{S}^1$ there are no smooth solutions, but a singular one known as the Dirac monopole. It will also be the case for the $G_2$-Monopoles studied here that there are Abelian monopoles having singularities at the zero section. These will be constructed in sections \ref{section:DiracS4} and \ref{S1bundles} for $\Lambda^2_-(\mathbb{S}^4)$ and $ \Lambda^2_-( \mathbb{CP}^2)$ respectively, and will be called Dirac monopoles by analogy. Below the main result is stated and in remark \ref{rem:Theorem} an intuitive explanation of some technical statements is given.

\begin{theorem}\label{Theorem}
There are compact Lie groups $K= Spin(5)$ ($K=SU(3)$) acting with cohomogeneity $1$ on $\Lambda^2_-(M)$ for $M=\mathbb{S}^4$ (respectively $M= \mathbb{CP}^2$) and $K$-homogeneous principal $SU(2)$ (respectively $SO(3)$) bundles $P$, such that the moduli spaces $\mathcal{M}_{inv}(\Lambda^2_-(M),P)$ are non empty and the following hold:
\begin{enumerate}
\item For all $(A, \Phi) \in \mathcal{M}_{inv}$, $\Phi^{-1}(0)$ is the zero section, and the mass gives a bijection
$$m : \mathcal{M}_{inv}(\Lambda^2_-(M),P) \rightarrow \mathbb{R}^+.$$
\item Let $R >0$, and $\lbrace (A_{\lambda}, \Phi_{\lambda}) \rbrace_{\lambda \in [\Lambda , + \infty )} \in \mathcal{M}_{inv}(\Lambda^2_-(M),P)$ be a sequence of monopoles with mass $\lambda$ converging to $+\infty$. Then there is a sequence $\eta(\lambda, R)$ converging to $0$ as $\lambda \rightarrow + \infty$ such that for all $x \in M$
$$\exp_{\eta}^* (A_{\lambda}, \eta \Phi_{\lambda}) \vert_{\Lambda^2_-(M)_x}$$
converges uniformly to the BPS monopole $(A^{BPS}, \Phi^{BPS})$ in the ball of radius $R$ in $(\mathbb{R}^3,g_E)$. Here $\exp_{\eta}$ denotes the exponential map along the fibre $\Lambda^2_-(M)_x \cong \mathbb{R}^3$.
\item Let $\lbrace (A_{\lambda}, \Phi_{\lambda}) \rbrace_{\lambda \in [\Lambda, +\infty)} \subset \mathcal{M}_{inv}$ be the sequence above. Then the translated sequence
$$\left( A_{\lambda}, \Phi_{\lambda}- \lambda \frac{\Phi_{\lambda}}{\vert \Phi_{\lambda} \vert} \right),$$
converges uniformly with all derivatives to a reducible, singular monopole on $\Lambda^2_-(M)$ with zero mass and which is smooth on $\Lambda^2_-(M) \backslash M$.
\end{enumerate}
\end{theorem}

\begin{remark}\label{rem:Theorem}
At this point it is convenient to explain in a intuitive manner the meaning of each item in theorem \ref{Theorem}. The first item states that for each fixed mass $m \in \mathbb{R}^+$, there is a unique invariant monopole $(A, \Phi)$ on $P$ and that for this monopole $\Phi^{-1}(0)$ is the zero section $M$. This is a very promising result, indeed there is a unique compact coassociative submanifold on $\Lambda^2_-(M)$ and this is precisely the zero section $M$. Hence, in these examples a monopole count on $P$ agrees with a count of rigid, compact, coassociative submanifolds.\\
The remaining items investigate the large mass limit of finite mass monopoles. Combined these state that large mass monopoles concentrate on the coassociative zero section $M$, with one BPS monopole bubbling off along the transverse directions to $M$ and a reducible monopole left behind on $\Lambda^2_-(M) \backslash M$. More precisely, on a tubular neighborhood of $M$ a large mass monopole $(A, \Phi)$ is close to a family of BPS monopoles on the transverse directions to $M$ and outside such a neighborhood $(A, \Phi)$ is approximately reducible.
\end{remark}

In \cite{BS} one further AC $G_2$-metric is constructed, this lives in $\mathcal{S}(\mathbb{S}^3)$, the bundle of spinors over the $3$-sphere. In this example there are no compact coassociative submanifolds and so any coassociative count would vanish identically. It is then interesting to investigate the existence of monopoles on $\mathcal{S}(\mathbb{S}^3)$ with the Bryant-Salamon metric. At the end of section \ref{sec:MonAC} we shall prove

\begin{theorem}\label{th:SpinorBundle}
Let $P \rightarrow \mathcal{S}(\mathbb{S}^3)$ be a $SU(2)$-bundle. There are no finite mass $m \neq 0$, irreducible monopoles $(A, \Phi)$ on $P$, such that $\vert \nabla^j \left(\varphi^* A - \pi^* A_{\infty} \right) \vert = O(r^{-5- \epsilon-j})$, for some $\epsilon >0$ and all $j \in \mathbb{N}_0$.
\end{theorem}

Since there are no compact coassociative submanifolds in $\mathcal{S}(\mathbb{S}^3)$, this is a very suggestive result regarding the relation of these with monopoles. Some further results of the paper are

\begin{enumerate}
\item Explicit $G_2$-instantons in $SU(2)$ and $SO(3)$ bundles on $\Lambda^2_-(\mathbb{S}^4)$ and $\Lambda^2_-(\mathbb{CP}^2)$ respectively are found. The statements are theorems \ref{insttheorem} and \ref{insttheorem2}, which give explicit formulas for the $G_2$-instantons. As it is the case for the $G_2$-monopoles of theorem \ref{Theorem}, these are asymptotic to Hermitian Yang Mills connections for the nearly K\"ahler structures on $\mathbb{CP}^3$ and $\mathbb{F}_3$ (the manifold of maximal flags in $\mathbb{C}^3$) respectively, whose formulas are explicitly given.
\item There is also an explicit family of irreducible $G_2$-instantons with structure group $SU(3)$ on $\Lambda^2_-(\mathbb{CP}^2)$. These are given in theorem \ref{SU(3)instantons}.
\item Abelian monopoles (Dirac monopoles) are defined and some explicitly examples found.
\end{enumerate}

The strategy to the proof of theorem \ref{Theorem} is to use the cohomogeneity-$1$ action of $K$ to reduce the monopole equations to ODE's. In the body of the paper examples of bundles $P \rightarrow M$ admitting a lift of such action are constructed and the monopole equations for invariant data $(A,\Phi)$ on these reduced to ODE's. These ODE's end up being the same as those arising for symmetric monopoles on $\mathbb{R}^3$ equipped with a specific spherically symmetric metric. The existence of solutions giving rise to smooth, irreducible monopoles with non-Abelian gauge group follows from the work in the Appendix \ref{AppendixA}, which then feeds into the main theorem \ref{Theorem}.

\subsection{Acknowledgments}

I would like to thank my supervisor Simon Donaldson and express my mathematical gratitude to him. I also want to thank Andrew Dancer, Mark Haskins and an anonymous Referee for very helpful comments on an earlier version of this paper. To the members of the geometry group in Imperial College I am grateful for many helpful conversations. This work is part of my work towards a PhD thesis in Imperial College London. This is financially supported by the FCT doctoral grant with reference SFRH / BD / 68756 / 2010 and I thank FCT this financial support.

\section{Preliminaries on Monopoles and Invariant Connections}

\subsection{YMH Energy and Identities}

Let $(X, \phi)$ be a $G_2$-manifold and $P \rightarrow X$ a principal $G$-bundle, in the following we shall consider pairs $(A, \Phi)$ of a connection and Higgs field not necessarily satisfying the monopole equations 

\begin{definition}
Let $U \subset X$ be open, when finite define the Energy $E$ and the Intermediate energy $E^I$ of a pair $(A, \Phi)$ on $U$ by
\begin{equation}
E(U) = \frac{1}{2}\int_U \vert \nabla_A \Phi \vert^2 + \vert F_A \vert^2 \ \ , \ \ E^I(U)= \frac{1}{2}\int_U \vert \nabla_A \Phi \vert^2 + \vert F_A \wedge \psi \vert^2 .
\end{equation}
\end{definition}

Note that the Intermediate Energy is always smaller or equal than the Energy. In fact in a $G_2$-manifold the $2$-forms split into irreducible $G_2$-representations as $\Lambda^2 = \Lambda^2_7 \oplus \Lambda^2_{14}$, where the subscripts indicate the dimension of the respective representations. This splitting is orthogonal and one denotes the respective projections by $\pi_7$ and $\pi_{14}$. In fact $F \wedge \psi = \pi_7(F) \wedge \psi$ for all $F \in \Omega^2(X, \mathbb{R})$ and so the Intermediate Energy replaces the $L^2$ norm of the curvature by the $L^2$ norm of the component of the curvature in $\Lambda^2_7$.

\begin{proposition}\label{EnergyIdentityProp}
Let $U \subset X $ be open and $(A,\Phi)$ a pair with finite Intermediate Energy on $U$, then
\begin{eqnarray}\label{G2idIntermediate}
E^I(U) & = &  \int_{\partial U} \langle \Phi, F_A \rangle \wedge \psi + \frac{1}{2} \Vert F_A \wedge \psi - \ast \nabla_A \Phi \Vert^2_{L^2(U)}.
\end{eqnarray}
Moreover, if the energy on $U$ is also finite, then
\begin{eqnarray}\label{G2id}
E(U) & = &  -\frac{1}{2} \int_U \langle F_A \wedge F_A \rangle \wedge \phi + \int_{\partial U} \langle \Phi, F_A \rangle \wedge \psi + \frac{1}{2} \Vert F_A \wedge \psi - \ast \nabla_A \Phi \Vert^2_{L^2(U)},
\end{eqnarray}
where $\langle F_A \wedge F_A \rangle$ combines the wedge of the $2$-form components of $F_A \in \Omega^2(X, \mathfrak{g}_P)$ with the inner product of the $\mathfrak{g}_P$ components.
\end{proposition}
\begin{proof}
It follows from linear algebra that for any $2$-form $\beta$, $\ast ( \ast ( \beta \wedge \psi ) \wedge \psi ) = 3 \pi_7(\beta)$ and $ \ast ( \beta \wedge \phi ) = 2\pi_7( \beta ) - \pi_{14} ( \beta )$. This gives rise to the following point-wise identities for any $2$-form $\beta$
\begin{eqnarray}\nonumber
\vert \beta \vert^2 dvol & = &  \left( \vert \pi_{14} (\beta) \vert^2 - 2\vert \pi_7(\beta) \vert^2 \right) dvol + 3 \vert \pi_7(\beta) \vert^2 dvol \\ \nonumber
& = &  - \beta  \wedge \ast ( \beta \wedge \phi )  +  \beta \wedge \ast ( \ast ( \beta \wedge \psi ) \wedge \psi )  \\ \nonumber
& = &  -\beta \wedge \beta \wedge \phi + \beta  \wedge \psi  \wedge \ast ( \beta \wedge \psi ) .
\end{eqnarray}
Using this for $\beta$ the $2$-form components of the curvature $F_A$ gives
\begin{eqnarray}\nonumber
E_U & = & \frac{1}{2}\Vert F_A \Vert^2_{L^2(U)} + \frac{1}{2} \Vert \nabla_A \Phi \Vert^2_{L^2(U)} \\ \label{EnergyHalfWay}
& = & - \frac{1}{2} \int_U \langle F_A \wedge F_A \rangle \wedge \phi + \frac{1}{2}\Vert F_A \wedge \psi \Vert^2_{L^2(U)} + \frac{1}{2} \Vert \nabla_A \Phi \Vert^2_{L^2(U)}.
\end{eqnarray}
To evaluate $\frac{1}{2}\Vert F_A \wedge \psi \Vert^2_{L^2(U)} + \frac{1}{2} \Vert \nabla_A \Phi \Vert^2_{L^2(U)}$ we proceed by computing
\begin{eqnarray}\label{IntermideateEnergy}
\Vert F_A \wedge \psi - \ast \nabla_A \Phi \Vert^2_{L^2(U)} & = &\Vert F_A \wedge \psi \Vert^2_{L^2(U)} + \Vert \nabla_A \Phi \Vert^2_{L^2(U)} - 2 \langle F_A \wedge \psi , \ast \nabla_A \Phi \rangle_{L^2(U)} .
\end{eqnarray}
The last term is given by the integral
\begin{eqnarray}\nonumber
\langle F_A \wedge \Theta , \ast \nabla_A \Phi \rangle_{L^2(U)}  & = &\int_U \langle F_A \wedge \psi \wedge \ast^2 \nabla_A \Phi \rangle = \int_U  \langle F_A \wedge \psi \wedge \nabla_A \Phi \rangle \\ \nonumber
& = & \int_U  \langle \nabla_A \Phi \wedge F_A\rangle \wedge \psi = \int_U  d \langle \Phi , F_A\rangle \wedge \psi \\ \nonumber
& = & \int_{\partial U} \langle \Phi , F_A\rangle \wedge \psi,
\end{eqnarray}
where we used the Bianchi identity $d_A F_A =0$, the fact that $\psi$ is closed and Stoke's theorem. Passing this term to the other side in equation \ref{IntermideateEnergy}, gives
\begin{equation}
E^I(U)= \frac{1}{2}\Vert F_A \wedge \psi \Vert^2_{L^2(U)} + \frac{1}{2} \Vert \nabla_A \Phi \Vert^2_{L^2(U)} = \int_{\partial U} \langle \Phi, F_A \rangle \wedge \psi + \Vert F_A \wedge \psi - \ast \nabla_A \Phi \Vert^2_{L^2(U)}.
\end{equation}
And replacing this back in equation \ref{EnergyHalfWay} gives
\begin{eqnarray}\nonumber
E(U) & = & - \frac{1}{2} \int_U \langle F_A \wedge F_A \rangle \wedge \phi + \int_{\partial U} \langle \Phi, F_A \rangle \wedge \psi + \frac{1}{2} \Vert F_A \wedge \psi - \ast \nabla_A \Phi \Vert^2_{L^2(U)}.
\end{eqnarray}
\end{proof}

In proposition \ref{prop:ACEnergy}, in the next section we shall give an application of this result by rewriting formula \ref{G2idIntermediate} for the intermediate energy of a pair $(A, \Phi)$ in the asymptotically conical (AC) case. It will be the case that the first term in formula \ref{G2idIntermediate} will give rise to a quantity that is fixed in terms of the topology of the bundle $P$ and the $G_2$-structure $\phi$. Since the second term is always greater or equal to zero with equality if and only if $(A, \Phi)$ is a monopole, it will follow that in the AC case monopoles minimize the intermediate energy. See proposition \ref{prop:ACEnergy} for the precise result and remark \ref{rem:MonClasses} below it.\\
We now give an application of proposition \ref{EnergyIdentityProp} to the case when $X$ is a compact manifold.

\begin{corollary}\label{NoMonopoles}
Let $X$ be compact and $(A, \Phi)$ a smooth monopole with $\Phi \neq 0$, then $A$ is a reducible $G_2$-instanton with energy
\begin{eqnarray}\label{G2id2}
E & = &  -\frac{1}{2} \int_X \langle F_A \wedge F_A \rangle \wedge \phi .
\end{eqnarray}
\end{corollary}
\begin{proof}
The energy identity \ref{G2idIntermediate} from proposition \ref{EnergyIdentityProp} gives that $\nabla_A \Phi =0$ and $F_A \wedge \psi=0$. Then $A$ is a $G_2$-instanton and since $\nabla_A \Phi=0$, the holonomy of $A$ must preserve $\Phi \neq 0$. Since $G$ is assumed to be semisimple, it follows that the holonomy is a proper subgroup of $G$ and so $A$ is reducible. The computation of the energy is reduced to the first term in equation \ref{G2id}.
\end{proof}

\begin{remark}
If $(X, \phi)$ is a compact $G_2$-manifold and $A$ is a $G_2$-instanton on a $G$-bundle over $X$, then $\pi_7(F_A)=0$ . Hence $F_A = \pi_{14}(F_A)$ and so $F_A \wedge \phi = -\ast F_A$. From this is easy to see that the quantity in formula \ref{G2id2} must always be positive.
\end{remark}

As a result of corollary \ref{NoMonopoles}, one concludes that the theory of smooth irreducible monopoles with $\Phi \neq 0$ can be reduced to the noncompact case. Hence, one must restrict to the case where $X$ is a complete, noncompact manifold with $G_2$ holonomy. For some kinds of asymptotic behavior one can adapt the definition of finite mass monopole in definition \ref{moduliS}. The Bryant-Salamon manifolds are asymptotically conical (AC) and in this case the definition of finite mass monopoles adapts with no change.

\subsection{Monopoles on AC $G_2$-manifolds}\label{sec:MonAC}

We start by describing the geometric structures on the Riemannian $6$-dimensional manifolds $(\Sigma, g_{\Sigma})$ that arise as the links of Riemannian $G_2$ cones.

\begin{definition}\label{su3str}
Let $\Sigma^6$ be a $6$ dimensional manifold, then the forms $(\omega, \Omega_1) \in \Omega^2 \oplus \Omega^3 (\Sigma, \mathbb{R})$, determine an $SU(3)$ structure on $\Sigma$ if:
\begin{itemize}
\item The $GL(6, \mathbb{R})$ orbit of $\Omega_1$ is open, with stabilizer a covering of $SL(3, \mathbb{C})$;
\item The following compatibility relations hold
\begin{eqnarray}\label{su3structure}
\omega \wedge \Omega_1 =  \omega \wedge \Omega_2 = 0 \ , \  \frac{\omega^3}{3!} = \frac{1}{4} \Omega_1 \wedge  \Omega_2.
\end{eqnarray}
where $\Omega_2=J\Omega_1$ and $J$ denotes the almost complex structure determined by $\Omega_1$
\item and $h(\cdot, \cdot )= \omega ( \cdot, J \cdot)$ determines on $\Sigma$ a Riemannian metric, i.e. $h$ is positive definite. 
\end{itemize}
\end{definition}

\begin{proposition}\label{G2conestr}
The Riemannian cone $(C(\Sigma) ,g_C = dr^2 + r^2 g_{\Sigma})$, with the $G_2$ structure
\begin{eqnarray}
\phi_C = r^2 dr \wedge \omega + r^3 \Omega_1 \ , \ \psi_C = r^4 \frac{\omega^2}{2} - r^3 dr \wedge \Omega_2,
\end{eqnarray}
has holonomy in $G_2$ if and only if $(\Sigma^6, g_{\Sigma})$ is nearly K\"ahler, i.e. the forms $(\omega, \Omega_1 , \Omega_2)$ satisfy
\begin{equation}\label{nearlyKahler}
 d \Omega_2 = -2 \omega^2 \ \ , \ \ d\omega = 3 \Omega_1.
\end{equation}
\end{proposition}
\begin{proof}
The conical metric $g_C$ has holonomy contained in $G_2$ if and only if $d \phi = d\psi =0$. Since
\begin{eqnarray} \nonumber
d \phi_C & = & r^2 dr \wedge \left( 3 \Omega_1 - d \omega \right) + r^3 d\Omega_1 \\ \nonumber
d \psi_C & = & r^4 d \left( \frac{\omega^2}{2} \right) + r^3 dr \wedge \left( d\Omega_2 + 2 \omega^2 \right),
\end{eqnarray}
one concludes that this holds if and only if $(\Sigma, g_{\Sigma})$ is nearly K\"ahler, i.e. the equations \ref{nearlyKahler} for the forms $(\omega, \Omega_1, \Omega_2)$ hold.
\end{proof}

\begin{definition}\label{AsymptoticallyConical}
A $G_2$-manifold $(X,g)$ is Asymptotically Conical (AC) with rate $\nu <0$ if there is a compact set $K \subset X$, a compact nearly K\"ahler $6$-manifold $(\Sigma, g_{\Sigma})$ and a diffeomorphism $\varphi : (1 , \infty) \times \Sigma \rightarrow X \backslash K$, such that on $(1 , \infty) \times \Sigma$, the metric $g_C = dr^2 + r^2 g_{\Sigma}$ and its Levi Civita connection $\nabla$ satisfy
$$\vert \nabla^j \left( \varphi^* g - g_C \right) \vert_{C} = O (r^{\nu-j}),$$
for all $j \in \mathbb{N}_0$. Moreover, we shall also suppose that $\vert \nabla^j \left( \varphi^* \phi - \phi_C \right) \vert_{C} = O (r^{\nu-j})$. A radius function will be any positive function $\rho : X \rightarrow \mathbb{R}^{+}$, such that in $X \backslash K$, $\rho = r \circ \varphi^{-1}$.
\end{definition}

\begin{example}\label{ex:G2manifolds}
There are only $3$ known examples of complete, AC and irreducible $G_2$-manifolds, these are known as the Bryant-Salamon manifolds \cite{BS}, see also \cite{Gibbons90}. The reference \cite{Lotay2012} studies the moduli spaces of AC $G_2$-manifolds, with a fixed asymptotic cone $C(\Sigma)$, and these examples are shown to be rigid. We shall examine these examples in more detail later.
\begin{enumerate}
\item Let $\Lambda^2_-(M)$ be the total space of the bundle of anti-self-dual $2$-forms over $(M^4,g)$, where $(M^4, g)$ denotes either the round $\mathbb{S}^4$ or the Fubini-Study $\mathbb{CP}^2$. Then, $\Lambda^2_-(M)$ admits a complete AC $G_2$ metric with rate $\nu=-4$ asymptotic to the cone over $\mathbb{CP}^3$ or $\mathbb{F}_3$ (the manifold of flags in $\mathbb{C}^3$), for $M=\mathbb{S}^4$ or $\mathbb{CP}^2$ respectively.
\item $\mathcal{S}(\mathbb{S}^3)$, the spinor bundle over the round $\mathbb{S}^3$ admits a complete AC $G_2$ metric with rate $\nu=-3$ which is asymptotic to the cone over $\mathbb{S}^3 \times \mathbb{S}^3$.
\end{enumerate}
The links of the cones which these are asymptotic to are (apart from $\mathbb{S}^6$) the only known examples of nearly K\"ahler manifolds. In fact all three are homogeneous: $\mathbb{CP}^3= Sp(2)/ U(1) \times SU(2)$, $\mathbb{F}_3= SU(3)/ T^2$ and $\mathbb{S}^3 \times \mathbb{S}^3= SU(2) \times SU(2)$.
\end{example}

We shall now consider finite mass monopoles as in definition \ref{moduliS} on $(X,g)$ an AC $G_2$-manifold with asymptotic cone $(C(\Sigma),g_{C} = dr^2 + r^2 g_{\Sigma})$. By proposition \ref{G2conestr}, $(\Sigma^6, g_{\Sigma})$ comes equipped with a nearly K\"ahler structure $(\omega, \Omega_1)$. Moreover, recall from the discussion preceding definition \ref{moduliS} that we assume the bundle $P$ is modeled on a bundle $P_{\infty}$ over $\Sigma$, i.e. $\varphi^* P \vert_{X \backslash K} \cong \pi^* P_{\infty}$ where $\pi: (1, + \infty) \times \Sigma \rightarrow \Sigma$ denotes the projection on the second component. Also recall from the same discussion, that in order to measure the growth rate of sections of $\Lambda^* \otimes \mathfrak{g}_P$ we shall be using a combination of $g_C$ with $h_{\infty}$, an $Ad$-invariant inner product on $\mathfrak{g}_{P_{\infty}}$.

\begin{proposition}\label{prop:AsymptoticMonAC}
Let $(A, \Phi)$ be a finite mass, irreducible monopole on $P$, then $\nabla_A \Phi \in L^2$ and there is $\Phi_{\infty} \in \Omega^0(X, \mathfrak{g}_{P_{\infty}})$, such that $\nabla_{A_{\infty}} \Phi_{\infty} =0$ and $\vert \nabla^j \left( \varphi^* \Phi - \pi^* \Phi_{\infty} \right) \vert = O(r^{-5-j})$ for all $j \in \mathbb{N}_0$. In particular, the intermediate energy $E^I(X)$ is finite.\\
Moreover, if we further suppose $[\varphi^* A - \pi^*A_{\infty}, \pi^* \Phi_{\infty} ] = O(r^{-6-\epsilon})$ for some $\epsilon>0$, then $A_{\infty}$ is an Hermitian-Yang-Mills (HYM) connection for the nearly K\"ahler strucre $(\omega, \Omega_1)$ on $\Sigma$, i.e. $F_A \wedge \omega^2 = F_A \wedge \Omega_2=0$.
\end{proposition}
\begin{proof}
This is proved in the author's PhD thesis \cite{Oliveira2014}. More concretely, that $\nabla_A \Phi \in L^2$ and the existence of $\Phi_{\infty}$ with the stated properties are proved in proposition $1.4.4$. To see that the intermediate energy is finite, notice that the monopole equation $\ast \nabla_A \Phi = F_A \wedge \psi$ implies that $E^I(X)= \Vert \nabla_A \Phi \Vert_{L^2} < + \infty$ as $\nabla_A \Phi \in L^2$. The second item showing that $A_{\infty}$ is HYM is proved in proposition $1.4.6$ of the same reference.
\end{proof}

\begin{remark}\label{rem:AsymptoticMonAC}
Notice that if the connection $A_{\infty}$ is not flat, the energy is infinte. This follows from the fact that $\pi^* F_{\infty}$ is a homogeneous $2$-form on the cone and so $\vert F_A \vert = O(\rho^{-2})$. Proposition \ref{prop:AsymptoticMonAC}, shows that even though, for a finite mass, irreducible monopole the intermediate energy is finite, i.e. $F_A \not\in L^2$, but $\pi_7(F_A) \in L^2$.
\end{remark}

To conclude this preliminary version on monopoles on AC $G_2$-manifolds we give an application of the Intermediate Energy identity in proposition \ref{EnergyIdentityProp} to such a situation. Suppose that $P$ is an $SU(2)$-bundle and $(A, \Phi)$ is a finite mass $m \in \mathbb{R}^+$ pair on $P$, not necessarily satisfying the monopole equations, but still satisfying all the conclusions from proposition \ref{prop:AsymptoticMonAC}. This means that we shall be assuming, the intermediate energy is finite and that there are $A_{\infty}, \Phi_{\infty}$ as before such that $A_{\infty}$ is HYM and $\nabla_{A_{\infty}} \Phi_{\infty} =0$.

\begin{proposition}\label{prop:ACEnergy}
Let $P$ be a $SU(2)$-bundle and $(A , \Phi \neq 0)$ a pair (not necessarily a monopole) as above and such that $\vert \nabla^j \left(\varphi^* A - \pi^* A_{\infty} \right) \vert = O(r^{-5- \epsilon-j})$, for some $\epsilon >0$ and all $j \in \mathbb{N}_0$. Then $A_{\infty}$ is reducible to an HYM connection on a $\mathbb{S}^1$-subbundle $Q_{\infty} \subset P_{\infty}$, and
\begin{eqnarray}\label{eq:G2id}
E^I & = & -2 \pi m \langle c_1(L) \cup [i^*\psi ], [\Sigma] \rangle + \frac{1}{2} \Vert F_A \wedge \psi - \ast \nabla_A \Phi \Vert^2_{L^2},
\end{eqnarray}
where $[i^* \psi] \in H^4(\Sigma, \mathbb{R})$ denotes the restriction of $[\psi]$ to any cross section $\varphi^{-1}(\lbrace r \rbrace \times \Sigma)$ over the conical end $X \backslash K$ and $L$ denotes the complex line bundle associated with $Q_{\infty}$ with respect to the standard representation.\\
If one further supposes that $c_1(L) \cup [i^*\psi ]=0$ or $(X,g)$ has rate $\nu < -4$, then there are no such finite mass, irreducible monopoles on $P$.
\end{proposition}
\begin{proof}
Since $\nabla_{\infty} \Phi_{\infty}=0$ and $\vert \Phi_{\infty} \vert =m >0$ the holonomy of the connection $A_{\infty}$ must preserve $\Phi_{\infty}$ and so be contained in $\mathbb{S}^1$. Hence $A_{\infty}$ is reducible to a connection on $Q_{\infty}$, a $\mathbb{S}^1$-subbundle of $P_{\infty}$.\\
To prove formula \ref{eq:G2id}, recall that by hypothesis the intermediate energy is finite $E^I_X= \lim_{r \rightarrow \infty} E^I_{B_r}$, so using formula \ref{G2idIntermediate}
$$E^I_X= \lim_{r \rightarrow \infty} \int_{\partial \overline{B}_r} \langle \Phi, F_A \rangle \wedge i_{r}^* \psi + \lim_{r \rightarrow + \infty} \frac{1}{2}\Vert F_A \wedge \psi - \ast \nabla_A \Phi \Vert^2_{L^2(B_r)} .$$
Now notice that $\frac{1}{2}\Vert F_A \wedge \psi - \ast \nabla_A \Phi \Vert^2_{L^2(B_r)}$ is a bounded, monotone and increasing sequence, and so converges to $\frac{1}{2}\Vert F_A \wedge \psi - \ast \nabla_A \Phi \Vert^2_{L^2(X)}$. We conclude that the sequence $ \int_{\partial \overline{B}_r} \langle \Phi, F_A \rangle \wedge i_{r}^* \psi$ converges as well. We proceed by abusing notation to write $\Phi= \Phi_{\infty} + \phi$, $\nabla_{A}= \nabla_{\infty}+a$ and $\psi= \psi_{C} + \eta$ where $\phi, a , \eta$ respectively have rates $-5,-5-\epsilon$ and $\nu$ with all derivatives. Since $A_{\infty}$ is nearly K\"ahler it is trivial to check that $F_{\infty} \wedge \psi_C =0$, so
\begin{equation}\label{eq:intermediateExpation}
\langle \Phi , F_{A} \rangle \wedge \psi = \langle \Phi_{\infty}, F_{\infty} \rangle \wedge \eta + O(\rho^{-6-\epsilon}),
\end{equation}
for some $\epsilon >0$. As $6 + \epsilon >6$ and $Vol(\partial \overline{B}_r) = O(r^6)$, the second term above can be disregarded and $\lim_{r \rightarrow \infty} \int_{\partial \overline{B}_r} \langle \Phi, F_A \rangle \wedge i_{r}^* \psi= \lim_{r \rightarrow \infty} \int_{\partial \overline{B}_r} \langle \Phi_{\infty}, F_{\infty} \rangle \wedge i_{r}^* \psi$.\\
It follows from $SU(2)$ representation theory that $\mathfrak{g}_{P_{\infty}} \otimes \mathbb{C}= \underline{\mathbb{C}} \oplus L_{\alpha} \oplus L_{-\alpha}$, for a complex line bundle $L_{\alpha}$ and $c_1(L_{\alpha}) = - c_1(L_{-\alpha})$, i.e. $L_{-\alpha} \cong L_{\alpha}^*$. Alternatively, one constructs the bundle $E = P_{\infty} \times_{SU(2)} \mathbb{C}^2$ associated with the standard representation. This splits into eigenspaces for $\Phi_{\infty}$ as $E= L \oplus L^*$, where $L^2 \cong L_{\alpha}$ and since $\nabla_{A_{\infty}} \Phi_{\infty} =0$ the connection $\nabla_{A_{\infty}}$ is reducible to a connection on $L$. In the end, one obtains
\begin{equation}
\Phi_{\infty} = \begin{pmatrix}
i m & 0 \\
0  & -i m
\end{pmatrix} \ , \ 
F_{A_{\infty}} = \begin{pmatrix}
F_L & 0 \\
0  & -F_L
\end{pmatrix},
\end{equation}
with $F_L \in -2 \pi i c_1(L) \in H^2(\Sigma, 2 \pi i \mathbb{Z})$ and $\vert \Phi_{\infty} \vert = m$, so
\begin{eqnarray}\nonumber
\Vert \nabla_A \Phi \Vert^2_{L^2} & = & \lim_{r \rightarrow \infty}  \int_{\partial B_r} \langle \Phi, F_A\rangle \wedge \psi  =  2i m \lim_{r \rightarrow \infty}  \int_{\partial B_r} F_L \wedge \psi \\ \nonumber
& = & 4\pi m \langle c_1(L) \cup [i^* \psi], \left[ \Sigma \right]  \rangle.
\end{eqnarray}
Moreover, if one supposes $c_1(L) \cup [i^*\psi] =0$ or $\nu < -4$ the first term in equation \ref{eq:intermediateExpation} is also $O(\rho^{-6-\epsilon'})$ for some $\epsilon' >0$ and also vanishes in the limit above. Hence, in both of those cases $E^I_X =0$, which implies $\nabla_A \Phi =0$ and so $A$ would have to be reducible.
\end{proof}

\begin{remark}\label{rem:MonClasses}
Notice that equation \ref{eq:G2id} writes the intermediate energy of $(A,\Phi)$ as a sum a topological term (or better said geometrical) and a term which is always greater or equal to zero, with equality if and only if $(A, \Phi)$ is a monopole. Hence, keeping the class $c_1(L) \cup [i^* \psi]$ fixed, monopoles minimize the intermediate energy. The class $[i^* \psi]$ is determined by the $G_2$-structure and $c_1(L) \in H^2(\Sigma, \mathbb{R})$ by the asymptotic behavior of the monopole. In fact, the class $c_1(L)$ is what we call in \cite{Oliveira2014} a monopole class and this is the $G_2$ analog of what in $3$ dimensions is known as the monopole charge.
\end{remark}

We shall now give an immediate application of these energy formulas by proving theorem \ref{th:SpinorBundle}.\\

\begin{proof}(of theorem \ref{th:SpinorBundle})
Notice that $\mathcal{S}(\mathbb{S}^3)$ is asymptotic to a cone over $\mathbb{S}^3 \times \mathbb{S}^3$ which has no second cohomoly and so the line bundle $L$ is topologically trivial (i.e. there are no nontrivial monopole classes in the language of remark \ref{rem:MonClasses}). We conclude that for any monopole $(A, \Phi)$ on $P$, the class $[i^*\psi] \cup c_1(L)$ vanishes and appealing to formula \ref{eq:G2id} in proposition \ref{prop:ACEnergy}, the intermediate energy vanishes. Hence $\nabla_A \Phi=0$ and since $\Phi \neq 0$ (as $m \neq 0$) would imply $A$ is reducible which is a contradiction.\\
Alternatively, to show that $c_1(L) \cup [i^* \psi]=0$ we could have argued as follows: $\mathcal{S}(\mathbb{S}^3)$ retracts onto $\mathbb{S}^3$, hence $H^4(\mathcal{S}(\mathbb{S}^3), \mathbb{R})=0$ and so $[\psi]=0$.
\end{proof}

In fact, very much the same proof as that of proposition \ref{prop:ACEnergy} can be used to compute the intermediate energy of a finite mass irreducible monopole on any $G$-bundle $P$, where $G$ is a compact semisimple Lie group. This is done in proposition $1.4.9$ of \cite{Oliveira2014}, where under the same hypothesis on the decay of the connection it is proven that
$$E_I= \langle [ \langle F_{\infty}, \Phi_{\infty} \rangle] \cup [i^* \psi], [\Sigma] \rangle.$$
Let $P$ be a $G$-bundle over $\mathcal{S}(\mathcal{S}^3)$, this formula is used in proposition $4.1.10$ of the same reference to prove the following result. There are no finite mass, irreducible monopoles $(A, \Phi \neq 0)$ on $P$, such that $\vert \nabla^j \left(\varphi^* A - \pi^* A_{\infty} \right) \vert = O(r^{-5- \epsilon-j})$, for some $\epsilon >0$ and all $j \in \mathbb{N}_0$.\\
One must also remark that there are $G_2$-instantons on $\mathcal{S}(\mathbb{S}^3)$ and these have been recently been constructed by Andrew Clarke in \cite{Clarke14}.

\begin{remark}\label{rem:SpinorBundle}
Theorem \ref{th:SpinorBundle} is a good motivation for the conjectural relation between monopoles and compact coassociative submanifolds.
\end{remark}

\subsection{Homogeneous Bundles and Invariant Connections}\label{HomogeneousSection}

This section contains standard material on bundles over homogeneous spaces, namely Wang's theorem classifying invariant connections on these, the main reference is \cite{KN}.\\
Let $K$ be a connected Lie group, $H \subset K$ a closed subgroup, then $K$ acts transitively on the homogeneous space $X=K/H$ with isotropy $H$. Denote by $\mathfrak{h} \subset \mathfrak{k}$ the Lie algebras of $H$ and $K$ respectively and suppose there is a complement $\mathfrak{m}$ to $\mathfrak{h}$ in $\mathfrak{k}$, i.e. $\mathfrak{k}= \mathfrak{h} \oplus \mathfrak{m}$ such that $Ad_h (\mathfrak{m}) \subseteq \mathfrak{m}$ for all $h \in H$. It is a standard result that there is a one-to-one correspondence between $K$-invariant metrics on $X$ and metrics on $\mathfrak{m}$ invariant under the adjoint $H$-action. We shall use this to construct the relevant $G_2$-metrics in sections \ref{BSmetricS} and \ref{MetricP2}.\\

Let $\pi: P \rightarrow X$ be a principal $G$-bundle. As usual $K$ acts on the left on $X$ and $G$ on the right on $P$. The bundle $P$ is said to be Homogeneous if there is a lift of the left action of $K$ on $X$ to the total space of $P$ which commutes with the right $G$-action on $P$. Suppose such a lift is given and let $H$ be the isotropy subgroup at $x \in X$, then it acts on the fibre $\pi^{-1}(x)$. As this action commutes with the transitive right $G$-action and gives rise to the \textbf{isotropy homomorphism} $\lambda : H \rightarrow G$, which can be used to reconstruct the bundle $P$ via $P = K \times_{(H, \lambda)} G$.\\
Let $(V, \eta)$ be a $G$-representation, where $V$ is a vector space and $\eta: G \rightarrow GL(V)$, construct the associated bundle $E= P \times_{G,\eta}V$ with fibre $V$. The lift of the $K$-action to $P$ naturally gives a $K$-action on $E$ and there is an isomorphism of homogeneous bundles
\begin{equation}\label{bundleisomorphism}
E \cong K \times_{H,\eta \circ \lambda} V.
\end{equation}
A section $s_E \in \Gamma(E)$ is said to be an \textbf{invariant section} under the $K$ action on $E$ if once regarded as an $H$-equivariant map $s_E : K \rightarrow V$ it is constant. Hence, $\eta \circ \lambda: H \rightarrow GL(V)$ can be used to decompose $V$ into irreducibles and the $H$-equivariant condition restricts $s_E$ to take values in the trivial components of $V$. A slight modification of the above paragraph in order to obtain invariant section of more general bundles can be stated. In particular, gauge transformations can be regarded as sections of the bundle $c(P) = P \times_{c, G} G$, where $c(g_1)g_2 = g_1 g_2 g_1^{-1}$ is the action by conjugation. Under the isomorphism $c(P) \cong K \times_{c \circ \lambda, H} G$ and so $K$-invariant Gauge transformations correspond to those $g \in \Omega^0(K, G)$ which are contants in the subgroup of $G$ centralized by $\lambda(H)$.\\
One turns now to the definition of invariant connections on the principal bundle $P = K \times_{H, \lambda} G$. These are given by a left invariant connection $1$-form $A \in \Omega^1(K, \mathfrak{g})$ and classified by Wang's theorem. The reductive decomposition $\mathfrak{k} = \mathfrak{m} \oplus \mathfrak{h}$ equips the bundle $K \rightarrow X=K/H$ with a $K$-invariant connection whose horizontal spaces are the left translates of $\mathfrak{m}$. This is known as the canonical invariant connection and it's connection $1$-form is left invariant and $A = \pi_{\mathfrak{h}}$ at the identity, where $\pi_{\mathfrak{h}}$ is the projection $\mathfrak{m}\oplus \mathfrak{h} \rightarrow \mathfrak{h}$. One can now state Wang's theorem (see \cite{KN} volume II., theorem 11.5).

\begin{theorem}\label{Wang}
(Wang \cite{Wang1958}) Let $P = K \times_{H, \lambda} G$ be a principal homogeneous $G$-bundle. Then $K$-invariant connections $A$ on $P$ are in one to one correspondence with morphisms of $H$ representations
\begin{equation}
\Lambda : (\mathfrak{m}, Ad) \rightarrow (\mathfrak{g}, Ad \circ \lambda).
\end{equation}
\end{theorem}

The upshot is that the left invariant $1$-form $A$ at the identity $1 \in K$ is given by $d\lambda \oplus \Lambda:\mathfrak{k}= \mathfrak{h} \oplus \mathfrak{m} \rightarrow \mathfrak{g}$. Moreover, the $H$-equivariant condition implies that the component $\Lambda: \mathfrak{m} \rightarrow \mathfrak{g}$ is a morphism of $H$ representations. In this case the canonical invariant connection is given by taking $\Lambda=0$, so that $A = d\lambda \circ \pi_{\mathfrak{h}}$.

\section{Monopoles on $\Lambda^2_-(\mathbb{S}^4)$}

Let $\mathbb{S}^4 \subset \mathbb{R}^5$ be the round sphere. Split $Spin(4) = SU_1(2) \times SU_2(2)$ and identify each $SU(2)$ with the unit quaternions. Let $F$ be the orthogonal frame bundle of $\mathbb{S}^4$, then given $\eta_x \in F_x$ the fiber over $x \in \mathbb{S}^4$, $\eta_x: T_x \mathbb{S}^4 \rightarrow \mathbb{H}$ is an isometry with the quaternions. $Spin(4)$ acts with a $\mathbb{Z}_2$ kernel on $F$ in the following way: given $(p,q) \in Spin(4)$ it acts via $(p,q) \eta_x =p \eta_x  \overline{q}$. $\Lambda^2_-(\mathbb{S}^4)$ is the bundle associated with $F$ via the action of $Spin(4)$ on the imaginary quaternions $(p,q) \Omega =q \Omega  \overline{q}$, for $(p,q) \in Spin(4)$ and $\Omega$ an imaginary quaternion.\\
Its isometry group of $\mathbb{S}^4$ equipped with the round metric is $SO(5)$, whose universal cover is $K=Spin(5)$. We lift the action of $SO(5)$ on $\mathbb{S}^4$ to $Spin(5)$. Then this action further lifts to the total space of $\Lambda^2_-(\mathbb{S}^4)$ as $A \cdot \Omega_x = (A^{-1})^* \Omega_x$, for $A \in Spin(5)$ and $\Omega_x \in \Lambda^2_-(\mathbb{S}^4)$ a nonzero anti-self-dual $2$-form on the tangent space to the point $x \in \mathbb{S}^4$. We now compute the isotropy $H$ of this action at $\Omega_x$ and we first observe that $H$ is contained in the stabilizer of the action on $\mathbb{S}^4$. Let $SU_1(2) \times SU_2(2) \cong Spin(4) \subset Spin(5)$ be the isotropy of the action on $x \in \mathbb{S}^4$. Then from the discussion on the previous paragraph we see that this acts on the fibre $\Lambda^2_-(T^*_x\mathbb{S}^4)$ by rotations via $SU_2(2) \hookrightarrow SU_1(2) \times SU_2(2)$. We conclude that away from the zero section, the isotropy of the $Spin(5)$ action is $H=SU_1(2) \times U_2(1)$. Moreover, since the action is isometric, so the principal orbits
$$Spin(5)/SU_1(2) \times U_2(1) \cong \mathbb{CP}^3,$$
are the level sets of the norm function $r = \vert \cdot \vert$ in $\Lambda^2_-(\mathbb{S}^4)$ induced by the round metric on $\mathbb{S}^4$. Let $\mathfrak{h}$ denote the Lie algebra of $H=SU_1(2) \times U_2(1)$, given a reductive decomposition
$$\mathfrak{spin}(5) = \mathfrak{h} \oplus \mathfrak{m},$$
the bundle $Spin(5)$ over $\mathbb{CP}^3$ gets equipped with a connection whose horizontal space is $\mathfrak{m}$. To give such a splitting write $\mathfrak{spin}(5) = \mathfrak{m}_1 \oplus \mathfrak{su}_1(2) \oplus \mathfrak{su}_2(2)$ and introduce a basis for the dual $\mathfrak{spin}(5)^*$ such that
\begin{equation}
\mathfrak{m}_{1}^* = \langle e^1, e^2, e^3, e^4 \rangle \ , \ \mathfrak{su}^*_1(2) =  \langle \eta^1, \eta^2, \eta^3 \rangle \  , \  \mathfrak{su}^*_2(2) = \langle   \omega^1 , \omega^2, \omega^3 \rangle,
\end{equation}
and the $\eta^i$, $\omega^i$ form standard dual basis for $\mathfrak{su}(2)$. Using the notation $e^{12}= e^1 \wedge e^2$, define the $2$-forms
\begin{eqnarray}\label{asdbasis}
\Omega_1  =  e^{12} - e^{34}  \ , \ & \Omega_2 = e^{13} - e^{42} &  , \ \Omega_3 =e^{14} - e^{23} \\ \nonumber
 \overline{\Omega}_1  =  e^{12} + e^{34}   \ , & \overline{\Omega}_2 = e^{13} + e^{42}  &  , \ \overline{\Omega}_3 =e^{14} + e^{23} ,
\end{eqnarray}
The Maurer Cartan relations encode the Lie algebra structure
\begin{eqnarray}\label{MC}
d \omega^1 = -2 \omega^{23} + \frac{1}{2}\Omega_1 \ , \ &  d \omega^2 = -2 \omega^{31} + \frac{1}{2} \Omega_2 & \ , \  d \omega^3 = -2 \omega^{12} + \frac{1}{2}\Omega_3 , \\ \nonumber
d \eta^1 = -2 \eta^{23} - \frac{1}{2} \overline{\Omega}_1 \ , \ &  d \eta^2 = -2 \eta^{31} - \frac{1}{2} \overline{\Omega}_2 & \ , \  d \eta^3 = -2 \eta^{12} - \frac{1}{2} \overline{\Omega}_3.
\end{eqnarray}
The ones involving the $de$'s are less important for what follows, but need to be used to compute
\begin{equation}
d \Omega_i = \epsilon_{ijk} \left( 2 \Omega_j \wedge \omega^k - 2\Omega_k \wedge \omega^j \right),
\end{equation}
for $i \in \lbrace 1,2,3 \rbrace$. Take the reductive decomposition $\mathfrak{spin}(5) = \mathfrak{h} \oplus \mathfrak{m}$, with
\begin{eqnarray}\label{irreds}
\mathfrak{m}^* & = &\mathfrak{m}_1 \oplus \mathfrak{m}_2 = \mathfrak{m}_1 \oplus \mathbb{R} \langle \omega^2 , \omega^3 \rangle \\
\mathfrak{h}^* & = & \mathfrak{su}_1(2) \oplus \mathbb{R} \langle \omega^1 \rangle.
\end{eqnarray}
The sphere bundle of $\Lambda^2_-$ is the twistor fibration $\pi: \mathbb{CP}^3 \rightarrow \mathbb{S}^4$ and at each point $p \in \mathbb{CP}^3$ there are non-canonical identifications $\mathfrak{m} \cong T_p\mathbb{CP}^3$ and $\mathfrak{m}_1 \cong T_{\pi(p)} \mathbb{S}^4$. The $2$-forms $\Omega_i$ give a basis for the anti-self-dual $2$-forms at $\pi(p)$, while the $\overline{\Omega}_i$ for the self-dual ones.

\subsection{Bryant and Salamon's $G_2$ Metric}\label{BSmetricS}

As seen above $\Lambda^2_-(\mathbb{S}^4) \backslash \mathbb{S}^4 \cong \mathbb{CP}^3 \times \mathbb{R}^+$, where each $\mathbb{CP}^3$ is a principal orbit of the $K=Spin(5)$ action. One may write the metric on $\Lambda^2_-(\mathbb{S}^4) \backslash \mathbb{S}^4$ from a family of $Spin(5)$ invariant metrics on $\mathbb{CP}^3$ and by letting the coordinate $\rho \in \mathbb{R}^+$ be the length through a geodesic intersecting the principal orbits orthogonally. As remarked at the beginning of section \ref{HomogeneousSection}, a $Spin(5)$ invariant metric on $\mathbb{CP}^3$ is determined by the splitting of $\mathfrak{m}$ into $\mathfrak{h}$ irreducible pieces. In the current situation we shall
\begin{equation}
\tilde{g}= d\rho \otimes d\rho + a^2(\rho) \left( \omega^2 \otimes \omega^2 + \omega^3 \otimes \omega^3 \right) + b^2(\rho) \left( \sum_{i=1}^4 e^i \otimes e^i \right),
\end{equation}
where $a,b$ are suitable real valued functions, to be be chosen so that $\tilde{g}$ is complete and has $G_2$ holonomy. The associative and coassociate calibrations are respectively
\begin{eqnarray}
\phi & = & d\rho \wedge \left( a^2 \omega^{23} + b^2 \Omega_1 \right) + ab^2 \left( \omega^3 \wedge \Omega_2 - \omega^2 \wedge \Omega_3 \right) \\
\psi & = & b^4 e^{1234} -a^2 b^2 \omega^{23} \wedge \Omega_1 - ab^2 d\rho \wedge \left( \omega^2 \wedge \Omega_2 + \omega^3 \wedge \Omega_3 \right).
\end{eqnarray}
The condition that the holonomy be in $G_2$ is equivalent to the closedness of these both. Using $d\Omega_1 = 4 d\omega^{23}$ and $d \left( \omega^2 \wedge \Omega_2 + \omega^3 \wedge \Omega_3 \right) = -2 e^{1234} + 4 \omega^{23} \wedge \Omega_1$, this reduces to the following set of ODE's
\begin{eqnarray}
\frac{d}{d\rho} (ab^2) = \frac{a^2}{2} + 2b^2 \ , \ \frac{d}{d\rho} (b^2) = a \ , \ \frac{d}{d\rho} (a^2 b^2) = 4ab^2.
\end{eqnarray}
These are solved by setting $a(s) = 2s f(s^2)$ and $b(s)=g(s^2)$, where the functions $f,g$ and the coordinate $s$ are given by
\begin{equation}\label{t}
\rho (s) = \int_0^s f(t^2) dt \ \ , \ \ f(s^2) = (1+ s^2)^{-\frac{1}{4}}  \ \ , \ \ g(s^2)=\sqrt{2}(1+ s^2)^{\frac{1}{4}}.
\end{equation}
These will be referred as $f,g$ but this should be understood as $f(s^2),g(s^2)$. The notation here is to be matched with the original reference \cite{BS}, which shows that the $G_2$-metric so obtained has full $G_2$-holonomy. For future reference, rewrite the $G_2$-structure in terms of $f,g$
\begin{eqnarray}\label{InvMetric}
\tilde{g} & = &  f^2 ds \otimes ds + 4 s^2 f^2 \left( \omega^2 \otimes \omega^2 + \omega^3 \otimes \omega^3 \right) + g^2 ( e^a \otimes e^a ) \\
\psi & = & g^4  e^{1234} -4s^2 f^2 g^2 \omega^{23} \wedge \Omega_1 - 2sf^2 g^2 ds \wedge \left( \omega^2 \wedge \Omega_2 + \omega^3 \wedge \Omega_3 \right).
\end{eqnarray}
For large $s$, $\rho(s) \sim 2\sqrt{s}$, $a(\rho) \sim \rho$ and $b(\rho) \sim \frac{\rho}{\sqrt{2}}$, so that the $G_2$-structure is asymptotic to the conical one over the nearly K\"ahler $\mathbb{CP}^3$
\begin{eqnarray}
\tilde{g}_C & = &  d\rho \otimes d\rho + \rho^2 \left( \omega^2 \otimes \omega^2 + \omega^3 \otimes \omega^3 \right) + \rho^2 \left( \sum_{i=1}^4 \frac{e^i}{\sqrt{2}} \otimes \frac{e^i}{\sqrt{2}} \right) \\ 
\phi_C & = &  \rho^2 d\rho \wedge \left( \omega^{23} + \frac{\Omega_1}{2} \right) + \frac{\rho^3}{2} \left( \omega^3 \wedge \Omega_2 - \omega^2 \wedge \Omega_3 \right) \\ 
\psi_C & = & \rho^4 \left( \frac{e^{1234}}{4} - \omega^{23} \wedge \frac{ \Omega_1 }{2}  \right) - \frac{\rho^3}{2} d\rho \wedge \left( \omega^2 \wedge \Omega_2 + \omega^3 \wedge \Omega_3 \right).
\end{eqnarray}

\subsection{$G_2$-Monopoles}

In this section we shall construct $G_2$-monopoles with gauge group $SU(2)$, the strategy pursued is as follows. First, we shall construct homogeneous bundles
$$P_{\lambda}= Spin(5) \times_{(\lambda, SU_1(2) \times U_2(1) )} SU(2),$$
by classifying the possible isotropy homomorphisms $\lambda: SU_1(2) \times U_2(1) \rightarrow SU(2)$. The second step is to compute invariant connections and Higgs fields with the help of Wang's theorem \ref{Wang}. Then, we compute the monopole equations for the invariant Higgs fields and connections reducing these to ODE's on $\rho \in \mathbb{R}^+_0$. Finally, to construct and study solutions to these ODE's we rely on the results from the appendix \ref{AppendixA}.\\

We start with the classification of isotropy homomorphisms. Notice that conjugate isotropy homomorphisms give rise to equivalent homogeneous bundles. Hence it is enough to consider istropy homomorphisms up to conjugation

\begin{proposition}
Up to conjugation, group homomorphisms $\lambda: SU_1(2) \times U_2(1) \rightarrow SU(2)$ are either trivial, or $\lambda(g , e^{i \theta})= g$, or $\lambda(g , e^{i \theta})= diag( e^{il\theta}, e^{-il \theta})$, for some $l \in \mathbb{Z}$ and $(g , e^{i \theta}) \in SU_1(2) \times U_2(1)$. 
\end{proposition}
\begin{proof}
Let $\lambda: SU_1(2) \times U_2(1) \rightarrow SU(2)$ be a group homomorphism. Both $\lambda \vert_{SU_1(2) \times \lbrace 1 \rbrace} : SU(2) \rightarrow SU(2)$ and $\lambda \vert_{ \lbrace 1 \rbrace \times U_2(1)}:  U(1) \rightarrow SU(2)$ are group homomorphisms. The first of these are automorphisms of $SU(2)$, since these are all inner, up to conjugation $\lambda \vert_{ SU_1(2) \times \lbrace 1 \rbrace}$ is either trivial or the identity. Regarding $\lambda \vert_{ \lbrace 1 \rbrace \times U_2(1)}:  U(1) \rightarrow SU(2)$, it must map $U(1)$ to an Abelian subgroup of $SU(2)$, this is either the identity or the maximal torus, which is isomorphic to $U(1)$. Hence we are reduced to classify automorphisms of $U(1)$, which are parametrized by their degree $l \in \mathbb{Z}$. The conclusion is that up to conjugacy $\lambda \vert_{ \lbrace 1 \rbrace \times SU_2(1)}$ is either trivial or given by $\lambda (1, e^{i\theta})= diag( e^{il\theta}, e^{-il \theta})$, for some $l \in \mathbb{Z}$.\\
Moreover, since any $(g,1) \in SU_1(2) \times \lbrace 1 \rbrace$ commutes with all $(1, e^{i\theta}) \in \lbrace 1 \rbrace \times U_2(1)$ and $\lambda$ is a homomorphism, elements in the image of $\lambda \vert_{SU_1(2) \times \lbrace 1 \rbrace}$ must commute with elements in the image of $\lambda \vert_{ \lbrace 1 \rbrace \times U_2(1)}$. Hence, we are reduced to the following cases:
\begin{itemize}
\item Both $\lambda \vert_{SU_1(2) \times \lbrace 1 \rbrace}, \lambda \vert_{ \lbrace 1 \rbrace \times U_2(1)}$ are trivial, in which case $\lambda$ is trivial.
\item If $\lambda \vert_{SU_1(2) \times \lbrace 1 \rbrace}$ is nontrivial, then it is conjugate to the identity and so its image is the whole $SU(2)$. Hence the image of $\lambda \vert_{ \lbrace 1 \rbrace \times U_2(1)}$ must be in the center of $SU(2)$ which is trivial. So in this case given $(g , e^{i \theta}) \in SU_1(2) \times U_2(1)$ we have $\lambda(g, e^{i \theta})= g \in SU(2)$.
\item If $\lambda \vert_{ \lbrace 1 \rbrace \times U_2(1)}$ has degree $l \neq 0$ and so nontrivial, then its image is the whole maximal torus in $SU(2)$. Hence, $\lambda \vert_{SU_1(2) \times \lbrace 1 \rbrace}$ must have its image in the maximal torus and so can not be conjugate to the identity. We are left with the case where $\lambda \vert_{SU_1(2) \times \lbrace 1 \rbrace}$ is trivial and we conclude that $\lambda(g, e^{i \theta})= diag( e^{il\theta}, e^{-il \theta}) \in SU(2)$.
\end{itemize}
\end{proof}

\begin{remark}
In this section we shall be investigating the bundles which are obtained via the isotropy homomorphisms given by $\lambda(g , e^{i \theta})= diag( e^{il\theta}, e^{-il \theta})$, for some $l \in \mathbb{Z}$ and $(g , e^{i \theta}) \in SU_1(2) \times U_2(1)$. We shall show that for $l=1$ there are finite mass, irreducible monopoles. Later in section \ref{sec:G2Inst} we shall also be able to construct irreducible $G_2$-instantons in this bundle.\\
The bundles obtained via the isotropy homomorphism $\lambda(g , e^{i \theta})= g$ will be examined in the second part of section \ref{sec:G2Inst}, where $G_2$-instantons will also be constructed.
\end{remark}

Let $l \in \mathbb{Z}$ be an integer and $\lambda_l: SU_1(2) \times U_2(1) \rightarrow SU(2)$ be the group homomorphism $\lambda_l(g, \theta) = diag \left( e^{il \theta} , e^{-il \theta} \right)$. These determine the family of homogeneous $SU(2)$ bundles $P_l= Spin(5) \times_{\lambda_l, SU_1(2) \times U_2(1)} SU(2)$, with $l \in \mathbb{Z}$.

\begin{lemma}\label{invariantconnections}
\begin{enumerate}
\item For each $l \in \mathbb{Z}$, the canonical invariant connection is given by $A_c = l\omega^1 \otimes T_1$, where $T_1,T_2,T_3$ is a standard basis for $\mathfrak{su}(2)$.
\item Let $A \in \Omega^1(Spin(5), \mathfrak{su}(2))$ be an invariant connection on $P_l$, then $A= A_c + (A- A_c)$ and $A-A_c =0$ for $l \neq 1$. For $l=1$ this is (up to an invariant gauge transformation)
\begin{eqnarray}\label{Lambda}
A - A_c & = & a \left( T_2 \otimes \omega^2  +   T_3 \otimes \omega^3 \right) ,
\end{eqnarray}
with $a \in \mathbb{R}$.
\item Let $\Phi$ be an invariant Higgs field of $P_1$, i.e. a section of the adjoint bundle $\mathfrak{g}_{P_1}$ invariant with respect to the $Spin(5)$ action, then $\Phi = \phi T_1$, for some constant $\phi \in \mathbb{R}$.
\end{enumerate}
\end{lemma}
\begin{proof}
\begin{enumerate}
\item The proof of the first assertion amounts to compute the derivative of the isotropy homomorphism $\lambda_l$, this is $d \lambda = l \omega^1 \otimes T_1$.
\item The second assertion is an application of Wang's theorem \ref{Wang}. Invariant connections correspond to morphisms of $SU_1(2) \times U_2(1)$ representations 
$$\Lambda_l: \left( \mathfrak{m} , Ad \right) \rightarrow \left( \mathfrak{su}(2), Ad \circ \lambda_l \right).$$
Decompose these into irreducible factors $\mathfrak{m}=\mathfrak{m}_1 \oplus \mathfrak{m}_2$ and $\mathfrak{su}(2)= \mathbb{R} \langle T_1 \rangle \oplus  \mathbb{R} \langle T_2 , T_3 \rangle$, where on $\mathbb{R} \langle T_1 \rangle$ the representation is trivial and $(\mathbb{R} \langle T_2, T_3 \rangle , Ad \circ \lambda_l) \cong (\mathfrak{m}_2, Ad)$ as representations, if and only if $l=1$. Then, Schur's lemma gives $\Lambda \vert_{\mathfrak{m}_1}=0$, while $\Lambda \vert_{\mathfrak{m}_2}$ vanishes for $l \neq 1$ and is an isomorphism for $l=1$. Invariant gauge transformations $g$ are constants in the subgroup of $SU(2)$ centralized by $\lambda_l(SU_1(2) \times U_2(1)) = U(1)$, the maximal torus in $SU(2)$. This is obviously its own centralizer and so $g$ must lie in the maximal torus which acts on $\mathbb{R} \langle T_2 , T_3 \rangle$ by rotations. So up to such a rotation $\Lambda_1$ can be picked to look like \ref{Lambda}.
\item To prove the third item, recall from section \ref{HomogeneousSection} that $Ad(P) = P \times_{(SU(2), Ad)} \mathfrak{su}(2)$ which is $Spin(5) \times_{(SU_1(2) \times U_2(1), Ad \circ \lambda)} \mathfrak{su}(2)$ and $\Phi$ must be constant with values in a trivial component of $(\mathfrak{su}(2), Ad \circ \lambda)$ as an $SU_1(2) \times U_2(1)$ representation.
\end{enumerate}
\end{proof}

\begin{remark}
The bundles $P_l$ are reducible to $\mathbb{S}^1$-bundles associated with the degree $l$ homomorphism of $\mathbb{S}^1$. Moreover, the canonical invariant connection is also reducible and induced from the canonical invariant connection on this bundle.
\end{remark}

Pullback the bundle $P_1$ to $\Lambda^2_-\mathbb{S}^4 \backslash \mathbb{S}^4 \cong \mathbb{R}^+ \times Spin(5)/ (SU_1(2) \times U_1(1))$ and equip it with an invariant connection $A \in \Omega^1( \mathbb{R}^+_s \times Spin(5), \mathfrak{su}(2))$ in radial gauge, i.e $A(\partial_s)=0$. Hence an invariant pair $(A, \Phi)$ must be a $1$ parameter family of pairs as in lemma \ref{invariantconnections}, i.e. $(A, \Phi)$ is now determined by $(a,\phi)$ which are real valued functions of $s \in \mathbb{R}^+$, being constant along each principal orbit of $Spin(5)$. From now on the dot $\cdot$ denotes differentiation with respect to $s$.

\begin{lemma}\label{LemmaCurvature}
\begin{enumerate}
\item The curvature of the connection $A= A_c + (A- A_c)$ is $F_A= F^H + F^V$, where $F^H$ and $F^V$ are respectively given by
\begin{eqnarray}\label{curvaturaagora}
F^H & = & \frac{1}{2}T_1 \otimes \Omega_1 +  \frac{a}{2} \left( T_2 \otimes \Omega_2 + T_3 \otimes \Omega_3 \right) \\
F^V & = &  -2 \left( 1- a^2 \right) T_1 \otimes \omega^{23} + \dot{a} \left( T_2 \otimes ds \wedge \omega^2 + T_3 \otimes ds \wedge \omega^3 \right).
\end{eqnarray}
\item The covariant derivative of the invariant Higgs field $\Phi = \phi T_1$ is given by
\begin{eqnarray}\nonumber
\nabla_A \Phi = \dot{\phi}T_1 \otimes ds + 2 \phi a \left( T_2 \otimes \omega^3  -  T_3 \otimes \omega^2 \right)
\end{eqnarray}
\end{enumerate}
\end{lemma}
\begin{proof}
The curvature of the invariant connection is computed as $F_A = F_{c} + d_{A_c} \left( A - A_c \right) + \frac{1}{2} \left[ \left( A - A_c \right) \wedge \left( A - A_c \right) \right]$. Making use of the Maurer Cartan relations \ref{MC}, these terms are
\begin{eqnarray}
F_c & = &  \left( -2 \omega^{23} + \frac{1}{2} \Omega_1 \right) \otimes T_1. \\ \nonumber
d_{A_c} \left( A - A_c \right) & = & ds \wedge \frac{d}{ds} \left( A - A_c \right) + d(A-A_c) + \left[ A_c \wedge (A-A_c) \right] \\ \nonumber\nonumber
& = & \dot{a} T_2 \otimes ds \wedge \omega^2 + \dot{a} T_3 \otimes ds \wedge \omega^3 + a T_2 \otimes \left( - 2\omega^{31} + \frac{1}{2} \Omega_2 \right) \\ \nonumber
& & +  \ a T_3 \otimes \left( -2 \omega^{12} + \frac{1}{2} \Omega_3 \right) + a\underbrace{[T_1,T_2 ]}_{=2T_3} \otimes \omega^{12} + a\underbrace{[T_1,T_3 ]}_{-2T_2} \otimes \omega^{13} \\  \nonumber
& = & \dot{a} \otimes ds \wedge \omega^2 + \dot{a} \otimes ds \wedge \omega^3 +  \frac{a}{2} \otimes \Omega_2 +\frac{a}{2} \otimes \Omega_3   \\ \nonumber
\frac{1}{2} \left[ ( A-A_c)^2 \right] & = & a^2 \left[ T_2, T_3  \right] \otimes \omega^{23} = 2 a^2 T_1 \otimes \omega^{23}. 
\end{eqnarray}
Summing all of these one can write $F_A= F^H + F^V$, where each of these is as in the statement. To compute is the covariant derivative of the Higgs field, write $\nabla_A \Phi= \dot{\phi}T_1 \otimes ds + \phi \nabla_{A_c} T_1 + \phi \left[ (A- A_c) ,T_1\right]$. The Bianchi identity for $A_c$ shows that $\nabla_{A_c} T_1=0$ and so
\begin{eqnarray}\nonumber
\nabla_A \Phi = \dot{\phi}T_1 \otimes ds + 2 \phi a \left( T_2 \otimes \omega^3  - T_3 \otimes \omega^2 \right).
\end{eqnarray}
\end{proof}

\begin{proposition}\label{ODEtheorem}
Let $h^2(\rho)=2s^2(\rho)f^{-2}(s^2(\rho))$ and consider the following set of ODE's for $(b,\phi)$
\begin{eqnarray}\label{ODEinvarianteS}
\frac{d\phi}{d\rho} & = & \frac{1}{2h^2}\left( b^2 -1 \right) \\ \label{ODEinvarianteS2}
\frac{db}{d\rho} & = & 2 \phi b. 
\end{eqnarray}
Then, the moduli space of invariant Monopoles $\mathcal{M}_{inv}(\Lambda^2_-(\mathbb{S}^4),P)$ can be identified with those pairs $(b=f^{-2}a, \phi)$ which solve \ref{ODEinvarianteS} and \label{ODEinvarianteS2} with $b(0)=1, \dot{b}(0)=0$ and $\lim_{\rho \rightarrow + \infty} f^2(s^2) b =0$.
\end{proposition}
\begin{proof}
To compute the monopole equation $F_A \wedge \psi=\ast \nabla_A \Phi$ one uses lemma \ref{LemmaCurvature}. The left hand side is
\begin{eqnarray} \nonumber
F_A \wedge \psi & = & F^V \wedge g^4 e^{1234} +  F^H \wedge \left( -4s^2 f^2g^2 \omega^{23} \wedge \Omega_1  + 2sf^2 g^2 ds \wedge \left( \omega^2 \wedge \Omega_2 + \omega^3 \wedge \Omega_3 \right) \right)\\ \nonumber
& = & \left(  -2g^4 \left( 1- a^2 \right) T_1 \otimes \omega^{23}+ g^4\dot{a} \left(  T_2 \otimes ds \wedge \omega^{2} +  T_3 \otimes ds \wedge \omega^3 \right) \right) \wedge e^{1234}  \\ \nonumber
& &  + \left( 4 s^2 g^2 f^2 \ T_1 \otimes \omega^{23} +  2s g^2 f^2 a \left( T_2 \otimes ds \wedge \omega^{2}+  T_3 \otimes ds \wedge \omega^3 \right) \right) \wedge e^{1234} \\ \nonumber
& = &  g^4 \left( 2\left( 2s^2 \frac{f^2}{g^2} - \left( 1- a^2 \right) \right) T_1 \otimes \omega^{23} \right)  \wedge e^{1234} \\ \nonumber
& & + g^4 \left( 2s \frac{f^2}{g^2} a +\dot{a} \right) \left( T_2 \otimes ds \wedge \omega^2  + T_3 \otimes ds \wedge \omega^3  \right)   \wedge e^{1234}. 
\end{eqnarray}
While for the right hand side of the equation, i.e. $\ast \nabla_A \Phi$, it is given by
\begin{eqnarray}
\ast \nabla_A \Phi = f g^4 \left( 4s \dot{\phi} T_1 \otimes \omega^{23} + 2 \phi a \left( T_2 \otimes ds \wedge \omega^2 + T_3 ds \wedge \omega^3 \right) \right) \wedge e^{1234}.
\end{eqnarray}
Hence the monopole equation reduces to the following set of ODE's
\begin{eqnarray}
\frac{d\phi}{ds} =-\frac{1}{2s^2f} \left( 1- a^2 \right) +  \frac{f}{g^2} \ , \ \frac{da}{ds} = - 2s\frac{ f^2}{g^2}a + 2f \phi a. 
\end{eqnarray}
Which in terms of $ \rho(s)= \int_0^s dl f(l^2) = \int_0^s dl \left( 1 + l^2 \right)^{-\frac{1}{4}}$ are
\begin{eqnarray}\label{monO2}
\frac{d\phi}{d\rho} = \frac{1}{2s^2f^2} \left( 1- a^2 \right) +  \frac{1}{g^2} \ , \ \frac{da}{d\rho} = - 2s\frac{ f}{g^2} a+ 2 \phi a.
\end{eqnarray}
Define $b(\rho)= f^{-2}(s^2(\rho)) a(\rho)$ as in the statement, then the second ODE in \ref{monO2} is equivalent to $\frac{db}{d\rho} = 2 \phi b$. What is left to show is that substituting $a$ by $b$ in the first equation in \ref{monO2} gives rise to the remaining equation for $\phi$. Notice that $\frac{1}{g^2} - \frac{1}{2s^2f^2} = -\frac{1}{2s^2} f^2$, and factor this term out in the following way
\begin{eqnarray}\nonumber
\frac{d\phi}{d\rho} & = &  \frac{1}{g^2}  -\frac{1}{2s^2f^2} \left( 1- a^2 \right) = -\frac{1}{2s^2} f^2 - \frac{1}{2s^2f^2} a^2 = -\frac{1}{2s^2} f^2 \left( 1 - f^{-4} a^2 \right). \nonumber
\end{eqnarray}
Replacing the term $ f^{-4} a^2$ by $b^2$ gives the equation in the statement. Then $\mathcal{M}_{inv}$, it is identified with the solutions to the ODE's that give rise to a connection and Higgs field extending over the zero section. This is the same as requiring the curvature to be bounded at $\rho = 0$, which from formula \ref{curvaturaagora} holds if and only if $a(0)=1$ and $\dot{a}(0)=0$. The ODE's imply that if these two hold then also $\phi(0)=0$ and $\dot{\phi}(0)$ is finite and so the Higgs field extends as well. So the conditions $a(0)=1$ and $\dot{a}(0)=0$ are necessary and sufficient to guarantee the monopole extends over the zero section. Moreover, a finite mass monopole must have its connection asymptotic to some $A_{\infty}$. This must be invariant and reducible and so $A_{\infty} = A_c$, in fact this is HYM for the nearly K\"ahler structure on $\mathbb{CP}^3$, as expected from proposition \ref{prop:AsymptoticMonAC}. Moreover, the condition that $A-A_c$ must decay at rate strictly faster than $-1$ is equivalent to $\lim_{\rho \rightarrow + \infty} \rho \vert A - A_c \vert=0$. Since $\vert \omega^2 \vert = \vert \omega^3 \vert= \frac{1}{2sf} \sim \rho^{-1}$ this condition is equivalent to $\lim_{\rho \rightarrow + \infty} a =0$, i.e. $\lim_{\rho \rightarrow + \infty} f^2(s^2) b =0$.
\end{proof}

\begin{remark}\label{remarkM1}
In the previous proof the field $a$ was rescaled to the field $b$. This makes the ODE's more familiar. In fact, the ODE's  \ref{ODEinvarianteS} and \label{ODEinvarianteS2} with $b(0)=1, \dot{b}(0)=0$ are precisely the same as the ones obtained for invariant monopoles on $\mathbb{R}^3$ with a spherically symmetric metric $g= dr^2 + h^2(r) g_{\mathbb{S}^2}$. Also, we remark that the condition $\lim_{\rho \rightarrow + \infty} f^2(s^2) b =0$ implies that $b$ has at most polynomial growth. These equations are analyzed in detail in the Appendix \ref{AppendixA}.
\end{remark}

\subsubsection{Reducible Monopoles}\label{section:DiracS4}

There are solutions to the equations in proposition \ref{ODEtheorem} by setting $b=a=0$ and letting $\phi$ solve $\frac{d\phi}{d\rho}= - \frac{1}{2h^2}$. The connection $A=A_c$ is the canonical invariant connection which is reducible and $\Phi$ is unbounded and has singularities at the zero section. This is the $G_2$ analogue of the Dirac monopole in the Euclidean $\mathbb{R}^3$.

\begin{remark}\label{rem:BCS4}
The canonical invariant connection $A_c$ is pulled back from an HYM connection on a complex line bundle $L$ over $\mathbb{CP}^3$ with $c_1(L)=[-2 \omega^{23} + \frac{1}{2} \Omega_1] \in H^2(\mathbb{CP}^3, \mathbb{Z})$ a monopole class as defined in section $4.1.3$ of \cite{Oliveira2014}.
\end{remark}

\begin{definition}\label{def:Dirac}
Let $(X, \phi)$ be a noncompact $G_2$-manifold and $P \subset X$ a coassociative submanifold. A Dirac monopole is an Abelian monopole on a line bundle defined on the complement of $P$. Moreover, $P$ will also be called the singular set of the Dirac monopole.
\end{definition}

Define the Green's function $G$, to be the unique function on $\Lambda^2_- (\mathbb{S}^4) \backslash \mathbb{S}^4$, such that $dG = \frac{1}{2h^2} d \rho$ and $ \lim_{\rho \rightarrow \infty}G(\rho)=0$. One can check it is harmonic on $\Lambda^2_- (\mathbb{S}^4) \backslash \mathbb{S}^4$; using $\ast \Delta= d \ast d$
\begin{eqnarray}\nonumber
\ast \Delta G & = &  \ast d \left( \frac{\partial G}{\partial \rho} \ast d\rho \right) = \ast d \left( 4s^2 \frac{\partial G}{\partial \rho} f ^2 g^4 \omega^{23} \wedge e^{1234} \right) =0, \nonumber
\end{eqnarray}
since $\frac{\partial G}{\partial \rho}=\frac{1}{2h^2} = \frac{f^2}{2s^2}$ and $g^2=2f^{-2}$ and so the quantity inside the parenthesis is constant on $\Lambda^2_- (\mathbb{S}^4) \backslash \mathbb{S}^4$. The upshot of this section is

\begin{proposition}
The solution to the monopole equations $(A^D, \Phi^D_m)=(A_c , G-m)$, is a mass $m$ Dirac monopole on $\Lambda^2_- (\mathbb{S}^4)$ with the zero section as its singular set.
\end{proposition}

\subsubsection{Irreducible Monopoles}

The general strategy to solve the ODE's to which proposition \ref{ODEtheorem} reduced the initial problem is via remark \ref{remarkM1} and the work in the Appendix \ref{AppendixA}. This gives an existence theorem for monopoles on $\Lambda^2_-(\mathbb{S}^4)$ parametrized by their mass and modeled on transverse BPS monopoles on a small neighborhood of the zero section the $\mathbb{R}^3$ fibers. The precise statement is given and proved below

\begin{theorem}\label{hconst}
The moduli space $\mathcal{M}_{inv}(\Lambda^2_-(\mathbb{S}^4), P_1)$ is not empty and the following hold
\begin{enumerate}
\item For all $(A, \Phi) \in \mathcal{M}_{inv}$, $\Phi^{-1}(0)= \mathbb{S}^4$ is the zero section, and the mass gives a bijection
$$m : \mathcal{M}_{inv} \rightarrow \mathbb{R}^+.$$
Let $(A_{m}, \Phi_{m}) \in \mathcal{M}_{inv}$ be a monopole with mass $m(A_{m}, \Phi_{m}) = m \in \mathbb{R}^+$, then there is a gauge in which
$$(A_{m} , \Phi_{m} ) = \left( A_c + f^2 b_{m}\left( T_2 \otimes \omega^2 + T_3 \otimes \omega^3 \right), \phi_{m} T_3  \right),$$
for suitable invariant functions $b_{m}, \phi_{m}$, satisfying $b_{m}(0)=1$ and $ \phi_{m}(0)=0$. The curvature of the connection $A_{m}$ in this gauge is
\begin{eqnarray}\nonumber
F_{A_m} & = &  \frac{f^2 b_{m}}{2} \left( T_2 \otimes \Omega_2 + T_3 \otimes \Omega_3 \right) + \frac{d}{d\rho} \left( f^2 b_{m} \right) \left( T_2 \otimes d\rho \wedge \omega^2 + T_3 \otimes d\rho \wedge \omega^3 \right) \\
& & + \left( 2 \left(f^4 b_{m}^2 -1 \right)  + \frac{\Omega_1}{2} \right) T_1 \otimes \omega^{23} 
\end{eqnarray}
\item Let $R >0$, and $\lbrace (A_{\lambda}, \Phi_{\lambda}) \rbrace_{\lambda \in [\Lambda , + \infty )} \in \mathcal{M}_{inv}(\Lambda^2_-(M),P)$ converging to $0$ as $\lambda \rightarrow + \infty$ a sequence of monopoles with mass $\lambda$ converging to $+\infty$. Then there is a sequence $\eta(\lambda, R)$ converging to $0$ as $\lambda \rightarrow + \infty$ such that for all $x \in M$
$$\exp_{\eta}^* (A_{\lambda}, \eta \Phi_{\lambda}) \vert_{\Lambda^2_-(M)_x}$$
converges uniformly to the BPS monopole $(A^{BPS}, \Phi^{BPS})$ in the ball of radius $R$ in $(\mathbb{R}^3,g_E)$. Here $\exp_{\eta}$ denotes the exponential map along the fibre $\Lambda^2_-(M)_x \cong \mathbb{R}^3$.
\item Let $\lbrace (A_{\lambda}, \Phi_{\lambda}) \rbrace_{\lambda \in [\Lambda, +\infty)}$ be the sequence above. Then the translated sequence
$$\left(A_{\lambda}, \Phi_{\lambda}- \lambda\frac{\Phi_{\lambda}}{\vert \Phi_{\lambda} \vert} \right),$$
converges uniformly with all derivatives on $(\Lambda^2_-(\mathbb{S}^4) \backslash \mathbb{S}^4, g )$, to the reducible mass $0$ Dirac monopole $(A^D, \Phi^D_0)$.
\end{enumerate}
\end{theorem}

\begin{proof}
One needs to find the solutions of the ODE's in proposition \ref{ODEtheorem} giving rise to Monopoles extending over the zero section, i.e. such that $b(0)=1$ and $ \dot{b}(0)=0$. This together with the ODE's implies that $\phi(0)=0$ and $\frac{d \phi}{d \rho}$ is bounded. Theorem \ref{teofinal} in the Appendix \ref{AppendixA}, gives the solutions $(b_m, \phi_m)$ to the ODE's which are unique after fixing $\lim_{\rho \rightarrow \infty} = - \frac{m}{2} \in \mathbb{R}^-$. From these solutions one obtains $a_m = f^2 b_m$ and $\phi_m$ which give rise to the monopole
$$(A_{m} , \Phi_{m} ) = \left( A_c + f^2 b_{m}\left( T_2 \otimes \omega^2 + T_3 \otimes \omega^3 \right), \phi_{m} T_3  \right)$$
The fact that the mass function is well defined and a bijection is a direct consequence from theorem \ref{teofinal} in Appendix \ref{AppendixA}, which basically asserts the previously claimed uniqueness of the solutions to the ODE's.\\
The results in the last two items refers to the bubbling behavior, which can be proven by using the corresponding one for monopoles in $\mathbb{R}^3$ and stated in Theorem \ref{teofinal}. Those results are proved in propositions \ref{convergencetoBPS} and \ref{bubbling} and based on the estimates provided by lemma \ref{estimatesclosetozero}. One must note that the result one wants to prove does not follow immediately from those ones. The reason is the following: The results from theorem \ref{teofinal} are for a family of monopoles 
$$\left( \tilde{A}_{\lambda}, \tilde{\Phi}_{\lambda} \right)= \left( A_c +  b_{\lambda}\left( T_2 \otimes \omega^2 + T_3 \otimes \omega^3 \right), \phi_{\lambda} T_3  \right).$$
on $\mathbb{R}^3 \cong (\Lambda^2_-)_x$ equipped with the metric $ g\vert_{(\Lambda^2_-)_x}$. These need to be re-proven for a monopole on $\Lambda^2_-(\mathbb{S}^4)$ restricted to a fiber, which differs from $(\tilde{A}_{\lambda}, \tilde{\Phi}_{\lambda})$ by rescaling the fields as
$$(A_{\lambda} , \Phi_{\lambda} )\vert_{(\Lambda^2_-)_x} = \left( A_c + f^2 b_{\lambda}\left( T_2 \otimes \omega^2 + T_3 \otimes \omega^3 \right), \phi_{\lambda} T_3  \right).$$
Let $\exp_{\eta} = s_{\eta}$, since $\tilde{\Phi}_{\lambda} = \Phi_{\lambda}$ one just needs to check that for all $\epsilon >0$ there is $\lambda$ and $\eta(R,\epsilon, \lambda)$ making $\Vert s_{\eta}^* A_{\lambda} -  A^{BPS} \Vert_{C^{0}( B_{R})}  \leq \epsilon$. Proceed as follows
\begin{eqnarray}\nonumber
\Vert s_{\eta}^* A_{\lambda} -  A^{BPS} \Vert_{C^{0}( B_{R})} & \leq &  \Vert s_{\eta}^*\tilde{A}_{\lambda} -  A^{BPS} \Vert_{C^{0}( B_{R})} +  \Vert  s_{\eta}^*\tilde{A}_{\lambda} - s_{\eta}^*A_{\lambda} \Vert_{C^{0}( B_{R})},
\end{eqnarray}
as already remarked the first of these can be made arbitrarily small due to proposition \ref{convergencetoBPS}. So there is $\eta'$ which makes the first term less than $\frac{\epsilon}{2}$. Regarding the second term $\Vert  s_{\eta}^*\tilde{A}_{\lambda} - s_{\eta}^*A_{\lambda} \Vert_{C^{0}( B_{R})} = \Vert  \tilde{A}_{\lambda} - A_{\lambda} \Vert_{C^{0}( B_{\eta R})}$ and so
\begin{eqnarray}\nonumber
\Vert  s_{\eta}^*\tilde{A}_{\lambda} - s_{\eta}^*A_{\lambda} \Vert_{C^{0}( B_{R})} =  \sup_{s \leq \eta R}  \Big\vert\left( b_{\lambda} ( 1- f^2 ) \right) \vert \omega_2 \vert_{g_E}  \Big\vert \leq  \sup_{s \leq \eta R}  \Big\vert \frac{1}{s} \left( b_{\lambda} (1- f^2) \right) \Big\vert \leq \frac{ \eta R}{2} ,\nonumber
\end{eqnarray}
where in the last line one uses the fact that $f=(1+s^2)^{\frac{1}{4}}$. We conclude that the estimate $\Vert s_{\eta}^* A_{\lambda} -  A^{BPS} \Vert_{C^{0}( B_{R})}  \leq \epsilon$ follows from making $\eta$ equal to the minimum of $\eta'$ and $\frac{\epsilon}{R}$. The last item in the statement needs no further check and follows directly from proposition \ref{bubbling} in Appendix \ref{AppendixA}.
\end{proof}

\begin{remark}\label{rem:as}
\begin{itemize}
\item It is straightforward to check that the connection of these monopoles converges to the canonical invariant connection $A_c$, which recall from remark \ref{rem:BCS4} is the pullback of an HYM connection on a line bundle $L \rightarrow \mathbb{CP}^3$, see remark \ref{rem:BCS4}.
\item The energy of these monopoles is not finite (as they are asymptotic to a nonflat connection on $\mathbb{CP}^3$). However, the Intermediate energy is indeed finite and the formula \ref{G2idIntermediate} in proposition \ref{EnergyIdentityProp} can be used to compute
\begin{equation}
E^I (A_m, \Phi_m) = \lim_{\rho \rightarrow \infty} 2\phi_m(\rho) \int_{\mathbb{CP}^3} 2 \omega_{23} \wedge 4 e_{1234} = 4 \pi m \langle [\mathbb{CP}^3], c_1(L) \cup  [i^* \psi] \rangle .
\end{equation}
Moreover, recall that inside the cohomology ring of $\Lambda^2_-(\mathbb{CP}^2)$ the class $[\psi]$ is dual to the zero section $\mathbb{S}^4$. As for $m \in \mathbb{R}^+$, it denotes the mass of the monopole $(A_m, \Phi_m)$.
\end{itemize}
\end{remark}

\subsection{$G_2$-instantons}\label{sec:G2Inst}

There is one further solution to the ODE's in proposition \ref{ODEtheorem} obtained by setting $\phi=0$ and $b=1$, which gives $a=f^2$. This is not contained in $\mathcal{M}_{inv}$, since $\Phi=0$ in this case, in fact this solution gives rise to an irreducible $G_2$-instanton, the solution is explicit and shall be stated below. 

\begin{theorem}\label{insttheorem}
The connection on $SU(2)$ bundle $P \rightarrow \Lambda^2_-(\mathbb{S}^4)$ given by $A = A_c + f^2(s^2) \left( T_2 \otimes \omega^2 + T_3 \otimes \omega^3 \right)$ is an irreducible $G_2$-instanton. Its curvature is given by
\begin{eqnarray}\nonumber
F_A & = &  \left( \frac{\Omega_1}{2}  - \frac{2s^2}{1 + s^2}\omega^{23} \right) \otimes T_1  + \frac{1}{2\sqrt{1+s^2}}\left( \Omega_2 \otimes T_2 + \Omega_3 \otimes T_3 \right) \\ \nonumber                      
& &  - \frac{s}{1+s^2} \left( ds \wedge \omega^2 \otimes T_2+ ds \wedge \omega^3 \otimes T_3\right).
\end{eqnarray}
\end{theorem}

\begin{remark}
As for the monopoles from the last section and these $G_2$-instantons also converge to the canonical invariant connection, see remarks \ref{rem:BCS4} and \ref{rem:as}. However, the convergence is much slower in the case of the instantons.
\end{remark}

Next one considers the $Spin$ bundle over $\mathbb{S}^4$, it may be equipped with a self-dual connection. Lifting this to $\Lambda^2_-(\mathbb{S}^4)$ also gives rise to a $G_2$-instanton. 

\begin{proposition}
The Spin connection $\theta$ on $\mathbb{S}^4$, when lifted to $\Lambda^2_-(\mathbb{S}^4)$ is a $G_2$-instanton with curvature
$$F_{\theta}= - \frac{1}{2} \overline{\Omega}_1  \otimes T_1  + \frac{1}{2} \overline{\Omega}_2 \otimes T_2 - \frac{1}{2} \overline{\Omega}_3 \otimes T_3.$$
\end{proposition}
\begin{proof}
The lift of the positive $Spin$ bundle, denoted by $Q$ is constructed by choosing the isotropy homomorphism $\lambda: SU_1(2) \times U_2(1) \rightarrow SU(2)$, given by $\lambda(g,e^{i\theta})=g$, for $(g,e^{i\theta}) \in SU_1(2) \times U_2(1)$. The canonical invariant connection $\theta \in \Omega^{1} ( Spin(5) ,\mathfrak{su}(2))$ is given by extending the projection on $\mathfrak{su}_1(2)$ as a left invariant $1$-form. Let $T_1,T_2,T_3$ denote a basis for $\mathfrak{su}(2)$ such that $[T_i,T_j] = 2 \epsilon_{ijk}T_k$. Then $\theta= \eta^1 \otimes T_1 + \eta^2 \otimes T_2 + \eta^3 \otimes T_3$. Using the Maurer Cartan relations \ref{MC} to compute the curvature $F_{\theta} = d\theta + \frac{1}{2} [ \theta \wedge \theta ]$, gives
\begin{eqnarray}\nonumber
F_{\theta} & = & 2 \eta^{23} \otimes T_1 + 2 \eta^{31} \otimes T_2 + 2 \eta^{12} \otimes T_3 \\ \nonumber
& & - \left( 2 \eta^{23} + \frac{1}{2} \overline{\Omega}_1 \right) \otimes T_1 - \left( 2 \eta^{31} + \frac{1}{2} \overline{\Omega}_2 \right) \otimes T_2 - \left( 2 \eta^{12} + \frac{1}{2} \overline{\Omega}_3 \right) \otimes T_3. 
\end{eqnarray}
\end{proof}

In fact one can check that $\theta$ is the unique invariant connection on $Q$ and $\Phi=0$ the unique invariant Higgs field. The first of these claims follows from an application of Wang's theorem \ref{Wang}, which identifies other invariant connections with morphisms of reps $\Lambda: \left( \mathfrak{m} , Ad \right) \rightarrow \left( \mathfrak{su}(2), Ad \circ \lambda \right)$. The left hand side splits into irreducibles as $\mathfrak{m} = \mathfrak{m}_1 \oplus \mathfrak{m}_2$, where $\mathfrak{m}_1$ is irreducible and $\mathfrak{m}_2$ is trivial. Since the right hand side is irreducible not isomorphic to $\mathfrak{m}_1$ (they have different dimensions), Schur's lemma gives $\Lambda=0$ as the only possibility.\\
Regarding invariant Higgs Fields $\Phi$, these must be constant for each $\rho$ and have values in the trivial component of the representation $\left( \mathfrak{su}(2), Ad \circ \lambda \right)$. Since this representation is irreducible and nontrivial, there are no nonzero invariant Higgs fields.

\section{Monopoles on $\Lambda^2_-(\mathbb{CP}^2)$}

The unit tangent bundle in $\Lambda^2_- (\mathbb{CP}^2)$, i.e. the twistor space of $\mathbb{CP}^2$, is the manifold of flags in $\mathbb{C}^3$. One may write
$$ \mathbb{ F }_3 = \lbrace (x, \xi ) \in \mathbb{CP}^2 \times (\mathbb{CP}^2)^* \ \vert \ \xi(x) =0 \rbrace ,$$
i.e. $x$ is a line in the hyperplane $\xi$. Then, there are three natural projections to $\mathbb{CP}^2$, given by $\pi_1(x,\xi) = x$, $\pi_2(x,\xi) = \xi \cap x^{\perp} $ and $\pi_3 (x, \xi) = \xi^{\perp}$, where $x^{\perp}, \xi^{\perp}$ denote the duals using the standard Hermitian product in $\mathbb{C}^3$. The fibrations $\pi_1$ and $\pi_3$ are holomorphic while $\pi_2$ is the twistor fibration.\\
The standard action of $SU(3)$ on $\mathbb{C}^3$ descends to a transitive action on $\mathbb{F}_3$ with isotropy the maximal torus $T^2 \subset SU(3)$, i.e.
$$\mathbb{F}_3 = SU(3) / T^2.$$
Moreover, $SU(3)$ also acts on the different $\mathbb{CP}^2$'s making the respective projections equivariant. The isotropy of this action on each $\mathbb{CP}^2$ is a different subgroup $H \cong S (U(1) \times U(2))$ of $SU(3)$, and are all conjugate by $\sigma$ an element of order $3$ in the Weyl group of $SU(3)$, i.e. $\pi_1 \circ \sigma^2 = \pi_2 \circ \sigma = \pi_3$. (Recall that the Weyl group is the residual action on $SU(3)/T^2$, descending from the action of $SU(3)$ on itself by conjugation.)
The standard Hermitian structure gives an isomorphism $\mathbb{CP}^2 \times (\mathbb{CP}^2)^* \cong \mathbb{CP}^2 \times \mathbb{CP}^2$. Let $\left( [x_1,x_2,x_3] , [\xi_1,\xi_2,\xi_3] \right) \in \mathbb{CP}^2 \times \mathbb{CP}^2$ be homogeneous coordinates, then $\mathbb{F}_3 \subset \mathbb{CP}^2 \times \mathbb{CP}^2$ is given by the points such that $x_1 \xi_1 + x_2 \xi_2 + x_3 \xi_3 =0$. At the point $(x,\xi) = ([1,0,0],[0,1,0])$, the isotropy is a fixed $T^2$ subgroup of $SU(3)$ given by
\begin{equation}
T^2=\label{maximaltori}
\left\lbrace i(e^{i\alpha_1}, e^{i\alpha_2}) = 
\begin{pmatrix}
e^{i\alpha_1} & 0 & 0 \\
0 & e^{i\alpha_2} & 0 \\
0 & 0 & e^{-i(\alpha_1 + \alpha_2) }
\end{pmatrix} \ , \ (\alpha_1, \alpha_2) \in [0,2\pi]^2 \right\rbrace,
\end{equation}
and this identification will be used throughout. Identify $\mathfrak{su}(3)$ with the anti-Hermitian matrices. Denote by $C_{ij}$ the matrix with all entries vanishing but $\pm 1$ on the $(i,j)$ and $(j,i)$ positions respectively, and let $D_{ij}$ the matrix with all entries vanishing but the $(i,j)$ and $(j,i)$ equal to $i$. Moreover, let $X_1 = diag(i,0,-i)$ and $X_2 = diag(0,i,-i)$, these generate the Lie algebra $\mathfrak{t}^2$ of the isotropy subgroup $T^2$. Then, the decomposition of $\mathfrak{su}(3)$ into $\mathfrak{t}^2$ irreducibles (the root space decomposition) is
\begin{eqnarray}\label{reductive}
\mathfrak{su}(3) = \mathfrak{t}^2 \oplus \mathfrak{m}_1\oplus  \mathfrak{m}_2 \oplus  \mathfrak{m}_3,
\end{eqnarray}
where $\mathfrak{t}^2 = \langle X_1, X_2 \rangle$, $ \mathfrak{m}_1= \langle C_{13}, D_{13} \rangle$, $ \mathfrak{m}_2=\langle C_{12}, D_{12}  \rangle$, $ \mathfrak{m}_3= \langle C_{23}, D_{23}  \rangle$. The splitting $\mathfrak{su}(3) = \mathfrak{t}^2 \oplus \mathfrak{m}$, with $\mathfrak{m} =\mathfrak{m}_1\oplus  \mathfrak{m}_2 \oplus  \mathfrak{m}_3 $, equips the bundle $SU(3) \rightarrow \mathbb{F}_3$ with a connection whose horizontal space is $\mathfrak{m}$. In particular $\pi_2(x, \xi)=[0:0:1]$ and $\mathbb{CP}^2 = \pi_2(\mathbb{F}_3)$ is identified with $\mathbb{CP}^2 \cong SU(3)/ S(U(2) \times U(1))$ for an explicit subgroup $S(U(2) \times U(1))$. Under this identification $\mathfrak{m}_1 \oplus \mathfrak{m}_3$ is the horizontal space of a connection on $\pi_2: \mathbb{F}_3 \rightarrow \mathbb{CP}^2$. Then the tangent space to the fibres of the twistor projection $\pi_2$ gives a distribution which is $\mathfrak{m}_2$. Define left invariant one forms on $SU(3)$, such that
$$(\mathfrak{t}^2)^* = \langle \theta_1,\theta_2 \rangle , \  \mathfrak{m}_1^*= \langle e_3, e_4 \rangle, \  \mathfrak{m}_2^* = \langle  \nu_1, \nu_2 \rangle, \ \mathfrak{m}_3^*= \langle e_1 , e_2  \rangle,$$
dual to the respective vectors above. One then defines the anti self-dual forms $\Omega_i$ as given in \ref{asdbasis} and define the $3$ forms
$$\gamma =\left( \Omega_2 \wedge \nu_2 -\Omega_3 \wedge \nu_1 \right) \ , \ \delta = -\Omega_3 \wedge \nu_2 -\Omega_2\wedge \nu_1$$
The Maurer Cartan relations are
\begin{eqnarray}\label{MCforSU(3)}
d\theta^1 = -2 e_{34} - 2 \nu_{12} & , & d\theta^2= -2 e_{12} + 2 \nu_{12} \\ \nonumber
d\nu_1  = \left(  -\theta^2 + \theta^1 \right) \wedge \nu_2 + \Omega_2 &  , &  d\nu_2  =- \left( - \theta^2 + \theta^1  \right) \wedge \nu_1 +\Omega_3 \\ \nonumber
d e_1  = \left( 2 \theta^2 + \theta^1 \right) \wedge e_2 - \nu_1 e_3  - \nu_2 e_4 & , & d e_2  =- \left( 2 \theta^2 + \theta^1 \right) \wedge e_1 - \nu_1 e_4 - \nu_2  e_3 \\ \nonumber
d e_3 = \left( \theta^2 + 2 \theta^1 \right) \wedge e_4 + \nu_1 e_1 - \nu_2 e_2 & , & d e_4  = - \left( \theta^2 + 2 \theta^1 \right) \wedge e_3 + \nu_1 e_2 + \nu_2  e_1.
\end{eqnarray}
These can in turn be used to compute
$$d\delta = 4 \left( e_{1234} -  \nu_{12} \wedge \Omega_1 \right) \ , \ d \gamma =0,$$
and in fact $\gamma= d e_{12} = d\nu_{12} = - de_{34}$ is exact.

\subsection{Bryant and Salamon's $G_2$-Metric}\label{MetricP2}

Using the fact that $\Lambda^2_- (\mathbb{CP}^2) \backslash \mathbb{CP}^2 \cong  \mathbb{R}^+ \times \mathbb{F}_3$ and each $\mathbb{F}_3$ slice is a principal orbit for the $SU(3)$ action, this section reduces the equations of $G_2$ holonomy with $SU(3)$ symmetry to ODE's on $\mathbb{R^+}$. Integrating these, one constructs the Bryant Salamon metric on $\Lambda^2_- (\mathbb{CP}^2)$. The notation tries to match up with the original reference \cite{BS} and also with \cite{CGLP}. Let $\rho \in \mathbb{R}^+$ be the distance along a geodesic emanating from the zero section and intersecting the principal orbits of the $SU(3)$ action orthogonally. The adjoint action of $T^2$ on $\mathfrak{m}$ decomposes into irreducible components as $\mathfrak{m} = \mathfrak{m}_1\oplus  \mathfrak{m}_2 \oplus  \mathfrak{m}_3$ (the root space decomposition after complexification) and any invariant metric can be written as
$$\tilde{g} = d\rho^2 + a^2(\rho) \left( e_1^2 + e_2^2 \right) + b^2(\rho) \left( e_3^2 + e_4^2 \right)  + c^2(\rho) \left( \nu_1^2 + \nu_2^2 \right), $$
for some positive functions $a,b,c$. The $3$ form $\phi$ and $\psi=\ast \phi$ defining the $G_2$ structure are given by
\begin{eqnarray}\nonumber
\phi & = & d\rho \wedge \left( -a^2 e_{34} + b^2 e_{12} + c^2 \nu_{12} \right) + abc \ \gamma \\ \nonumber
\psi  & = & -b^2c^2 \ e_{12} \wedge \nu_{12} + a^2 c^2 \ e_{34} \wedge \nu_{12} + a^2 b^2 \ e_{1234} + abc \ d\rho \wedge \delta.
\end{eqnarray}
The metric $g$ has holonomy in $G_2$ if and only if $d\phi = d \psi =0$. Since $\gamma$ is closed and $d\delta = 4 ( e_{1234} -   \nu_{12} \wedge \Omega_1 )$, the equations reduce to the following ODE's
\begin{eqnarray}\label{metricaODE1}
4abc = \frac{d}{d\rho} (a^2 b^2) = \frac{d}{d\rho} (a^2 c^2) = \frac{d}{d\rho} (b^2 c^2) \ , \ \frac{d}{d\rho} (abc) = a^2 + b^2 + c^2 .
\end{eqnarray}
Recall from section \ref{BSmetricS}, equation \ref{t}, the definition of the following implicit functions of $\rho$
\begin{equation}\label{rhos}
\rho (s) = \int_0^s fds \ \ , \ \ f(s) = (1+ s^2)^{-\frac{1}{4}}  \ \ , \ \ g(s)=\sqrt{2}(1+ s^2)^{\frac{1}{4}}.
\end{equation}
Then as already done for $\Lambda^2_-(\mathbb{S}^4)$, one can regard $s$ as a radial coordinate. Moreover, the solution to the ODE's \ref{metricaODE1}, which gives the Bryant Salamon $G_2$ structure is given by setting $a(\rho)= b(\rho) = 2f^{-1}(s(\rho))$ and $c(\rho) = 2s(\rho)f(s(\rho))$. The $G_2$ structure obtained is
\begin{eqnarray}\label{G2structueCP2}
\tilde{g} & = &  f^2 ds^2 + 4s^2 f^2 \left( \nu_1^2 + \nu_2^2 \right)  + 2g^2 \left( e_1^2 + e_2^2 + e_3^2 + e_4^2\right) \\
\phi & = &  4s^2f^3 ds \wedge \nu_{12} + 4fg^2 ds \wedge\Omega_1 - 2sf^2g^2 \left( \nu_1 \wedge \Omega_2 + \nu_2 \wedge \Omega_3 \right) \\
\psi  & = & 4g^4  e_{1234} -  8s^2 f^2 g^2 \Omega_1 \wedge \nu_{12}+ 2sf^2g^2 ds \wedge \left( \Omega_2 \wedge \nu_2 - \Omega_3 \wedge \nu_3 \right).
\end{eqnarray}
It converges for large $\rho$ to the Riemannian cone over the nearly K\"ahler $\mathbb{F}^3$. To check this we use equation \ref{rhos} to obtain $\rho \sim \sqrt{s}$ and so collecting the leading order terms
\begin{eqnarray}\nonumber
\tilde{g}_C & = & d\rho^2 + \rho^2 \left( 4e_1^{2} + 4 e_2^2 + 4e_3^2 +4 e_4^2 + 4\nu_1^2 + 4\nu_2^2 \right)\\ \nonumber
\phi_C  & = &  \rho^2 d\rho \wedge \left(  \Omega_1 + \nu_{12} \right) - \rho^3 \left( \nu_1 \wedge \Omega_2 + \nu_2 \wedge \Omega_3 \right) \\ \nonumber
\psi_C  & = & \rho^4 \left( ( \sigma_{12} - \Sigma_{12} ) \wedge \nu_{12} + \sigma_{12} \wedge \Sigma_{12} \right)+ \rho^3 d\rho \wedge \left( \Omega_2 \wedge \nu_2 - \Omega_3 \wedge \nu_3 \right).
\end{eqnarray}

\subsection{$G_2$-Monopoles}

This section will use the $SU(3)$ symmetry to construct $G_2$-monopoles and $G_2$-instantons on $\Lambda^2_-(\mathbb{CP}^2) $. The strategy for the construction of the invariant data (homogeneous bundle with invariant connections and Higgs Fields) is as follows (see section \ref{HomogeneousSection} for further details). Given an isotropy homomorphism $\lambda: T^2 \rightarrow G$, one constructs homogeneous principal $G$-bundles via $P_{\lambda} = SU(3) \times_{(T^2, \lambda)} G$ on $\mathbb{F}_3 \cong SU(3) / T^2$. The invariant connections are determined by their left-invariant connection $1$-form $A \in \Omega^1(SU(3), \mathfrak{g})$. Once a complement $\mathfrak{m}$ to $\mathfrak{t}^2$ has been chosen, Wang's theorem \ref{Wang} parametrizes invariant connections in terms of morphisms of $T^2$-representations $\Lambda: \mathfrak{m} \rightarrow \mathfrak{g}$. The decomposition of $\mathfrak{m}$ into irreducible components is
$$\mathfrak{m}= \mathfrak{m}_1 \oplus  \mathfrak{m}_2 \oplus  \mathfrak{m}_3,$$
where each component is labeled by a positive root. Then by Schur's lemma $\Lambda \vert_{m_i}$ will either vanish or map $\mathfrak{m}_i$ into an isomorphic representation inside $\mathfrak{g}$.
In the same way, invariant Higgs fields, i.e. invariant sections of the adjoint bundle $\mathfrak{g}_{P_{\lambda}} = P_{\lambda} \times_{Ad} \mathfrak{g}$, i.e. $SU(3) \times_{Ad \circ \lambda} \mathfrak{g}$, correspond to vectors in the trivial components of the $T^2$-representation $Ad \circ \lambda$ on $\mathfrak{g}$.

\subsubsection{$G=\mathbb{S}^1$-Bundles}\label{S1bundles}

For gauge group $G=\mathbb{S}^1$, the possible isotropy homomorphisms are given by the weights
\begin{equation}\label{isotropyU(1)}
\lambda_{n,l}(e^{i\alpha_1} , e^{i\alpha_2}) = e^{i(n\alpha_1 + l \alpha_2)},
\end{equation}
and so parametrized by two integers $(n,l) \in \mathbb{Z}^2$. Each of these gives rise a $\mathbb{S}^1$-bundle $Q_{n,l}$ over $\mathbb{F}_3$, which we pullback to $\Lambda^2_-(\mathbb{CP}^2)$. We now equip each $Q_{n,l}$ with invariant connections using Wang's theorem: since none of the root spaces is a trivial representation and the $Ad \circ \lambda_{n,l}$ action on $\mathfrak{u}(1)$ is trivial, the canonical invariant connection
$$A^c_{n,l} = n\theta^1 + l\theta^2,$$
is the unique invariant connection. The Maurer Cartan relations for $SU(3)$, in \ref{MCforSU(3)}, give $F_{n,l} = -2n (e_{34} + \nu_{12}) + 2l (\nu_{12} -e_{12})$, which one rearranges to
\begin{eqnarray}\label{S1curvature}
F^c_{n,l} =-2n e_{34} - 2le_{12} + 2(l-n) \nu_{12},
\end{eqnarray}
is a closed, $T^2$-invariant, horizontal $2$-form in $SU(3)$ and descends to a closed $2$-form on $\mathbb{F}_3 = SU(3)/T^2$. Particular cases are $d\theta^1 =F^c_{1,0}$ and $d \theta^2 = F^c_{0,1}$, hence their classes generate $H^2(\mathbb{F}_3, \mathbb{R})$. It is a consequence of the next lemma that $[d\theta_1], [d\theta_2]$ also generate $H^2(\mathbb{F}_3 , \mathbb{Z})$ seen as a lattice inside $H^2(\mathbb{F}_3 , \mathbb{R})$.

\begin{lemma}\label{H2}
$H^2(\mathbb{F}_3 , \mathbb{Z}) \cong H^1(\mathbb{T}^2 , \mathbb{Z})$ is the lattice generated by the roots. Let $\mathcal{O}(1)$ denote the canonical line bundle of $\mathbb{CP}^2$, then $c_1(\pi_1^* \mathcal{O}(1)) = \left[F_{1,0}  \right]$, $c_1(\pi_2^* \mathcal{O}(1)) = \left[ F_{-1,-1} \right]$ and $c_1(\pi_3^* \mathcal{O}(1)) = \left[ F_{0,1} \right]$.
\end{lemma}
\begin{proof}
The first assertion is a consequence of Serre's spectral sequence and the fact that $SU(3)$ is $2$-connected, so $H^2(\mathbb{F}_3 , \mathbb{Z}) \cong H^1(\mathbb{T}^2 , \mathbb{Z})$. This identification can be made explicit by noticing that the integral weights can be taken as generators and also as giving rise to the isotropy homomorphisms generating the group of complex line bundles. Then, given $\alpha \in H^1(T^2, \mathbb{Z}) $, its exponential gives the isotropy homomorphism of the line bundle $L_{\alpha}= SU(3) \times_{T^2, e^{\alpha}} \mathbb{C}$ whose first Chern class is $[d\alpha] \in H^2(\mathbb{F}_3, \mathbb{R}) \cap H^2(\mathbb{F}_3, \mathbb{Z})$. Notice that in this case $\alpha$ is actually the canonical connection of the underlying $\mathbb{S}^1$-bundle and $d\alpha$ its curvature. 
Since $\pi_1$ is holomorphic, $\pi_2$ is real and $\pi_3$ antiholomorhic
$$\pi_1^* \mathcal{K}_{ \mathbb{CP}^2} \cong  (\overline{\mathfrak{m}^{\mathbb{C}}_2} )^*  \otimes (\overline{\mathfrak{m}^{\mathbb{C}}_1})^* \ , \ \pi_2^* \mathcal{K}_{ \mathbb{CP}^2} \cong  (\mathfrak{m}^{\mathbb{C}}_1)^*   \otimes (\mathfrak{m}^{\mathbb{C}}_3)^* \ , \ \pi_3^* \mathcal{K}_{ \mathbb{CP}^2} \cong  (\mathfrak{m}^{\mathbb{C}}_2)^*   \otimes (\overline{\mathfrak{m}^{\mathbb{C}}_3})^*, $$
these are the complex line bundles determined from the isotropy homomorphisms $e^{\alpha_{i}} : T^2 \rightarrow \mathbb{S}^1$ with
\begin{eqnarray}\nonumber
\alpha_1 & = &   -(2 \theta^1 + \theta^2)-(\theta^1-\theta^2)=-3\theta^1  \\ \nonumber
\alpha_2 & = &  (2 \theta^1 + \theta^2)(\theta^1+ 2\theta^2 ) = 3(\theta^1 + \theta^2)\\ \nonumber
\alpha_3 & = &  +( \theta^1 - \theta^2)-(\theta^1+2\theta^2)= -3 \theta^2.
\end{eqnarray}
Since $\mathcal{K}_{ \mathbb{CP}^2} \cong \mathcal{O}_{ \mathbb{CP}^2}(-3)$, the statement follows and $c_1(\pi_i^* \mathcal{O}_{\mathbb{CP}^2}(-1))$ generate the integral second homology the statement follows.
\end{proof}

\begin{lemma}\label{l3}
$F_{1,1}$ generates a subgroup of $H^2(\mathbb{F}_3, \mathbb{Z})$ corresponding to the first Chern classes of the line bundles pulled back from $\mathbb{CP}^2$ via $\pi_2$.
\end{lemma}
\begin{proof}
This is a consequence of the previous lemma. Alternatively the base of the twistor fibration $\pi_2$ is $\mathbb{CP}^2= SU(3)/S( U(2) \times U(1))$ so the bundles that are pull back from the base must have an isotropy homomorphisms $\lambda_{n,l}: T^2 \rightarrow \mathbb{S}^1$, which factors via $T^2 \hookrightarrow S( U(2) \times U(1)) \rightarrow \mathbb{S}^1$. For the choices made before this is a fixed subgroup $S( U(2) \times U(1))$ of $SU(3)$ for which the aforementioned homomorphisms are precisely the ones with $n=l$. In fact, these are the only cases for which the curvature $F_{n,n}$ of the canonical invariant connection stays bounded close to the zero section.
\end{proof}

Since $\mathbb{S}^1$ is Abelian an invariant Higgs field $\Phi$ is just a real valued function of the radial coordinate $\rho$ and to compute the monopole equations one needs
\begin{eqnarray}\nonumber
F_{n,l} \wedge \psi & = & \left( 8 g^4 \left( l - n \right) - 16s^2f^2g^2 (l-n)  \right) e_{1235} \wedge \nu_{12} \\ \nonumber
& = & 8(l-n) g^2 \left( g^2 - 2s^2 f^2 \right) e_{1234} \wedge \nu_{12} \\ \nonumber
& = & 32 (l-n)e_{1234} \wedge \nu_{12} ,
\end{eqnarray}
where it is useful to use $\  g^2 =2 f^{-2}$. Moreover, $d\Phi = \frac{d\Phi}{ d\rho} d\rho$ and so $\ast d\Phi = 64s^2 f^{-2} \frac{d\Phi}{ d\rho} e_{1234} \wedge \nu_{12}$. The monopole equation can then be written as an ODE for $\Phi$. For each $(n,l)$ and a given mass it has a unique solution obtained by solving
\begin{equation}\label{S1ODE}
d\Phi^m_{n,l} = \frac{l-n}{2h^2(\rho)} d\rho \ \ , \ \ \lim_{\rho \rightarrow \infty} \Phi^m_{n,l} = m,
\end{equation}
where recall $h(\rho)=s(\rho)f^{-1}(s^2(\rho))$. Moreover, the connection associated with this is the canonical invariant one $A^c_{n,l}$. This monopole does not extend over the zero section unless $l=n$ in which case $\Phi$ is constant and so for $n \neq l$ gives a Dirac type monopole, see definition \ref{def:Dirac}

\begin{proposition}\label{prop:S1Monopoles}
For $n \neq l$ the monopole $(A^c_{n,l}, \Phi^m_{n,l})$ is a Dirac monopole on $\Lambda^2_-(\mathbb{CP}^2)$ with singular set the zero section. For $n=l$ the connection $A^c_{n,n}$ is a $G_2$-instanton obtained by lifting a self-dual connection on $\mathcal{O}_{\mathbb{CP}^2}(-n)$ via $\pi_2$, their curvature is $F_{n,n} = -2n \left( e_{12} + e_{34} \right)$.
\end{proposition}

The Higgs field is then a harmonic function, which in the case $n \neq l$ is non constant and unbounded at the zero section. For large $\rho$ one uses equation \ref{t}, which gives $s = \frac{\rho^2}{4}+...$, where the dots represent lower order terms. Then, since $h^2(\rho)= s^2 \sqrt{1+ s^2}$ we conclude that $h^2(\rho) = \frac{\rho^6}{64}+...$. Plugging this back in equation \ref{S1ODE} gives $\Phi_{n,l} =  - \frac{32}{5}(l-n)\rho^{-5} +...$, i.e. $\Phi_{n,l}$ decays like the Green's function for the cone metric.

\begin{remark}\label{rem:HYMCP2}
These invariant connections are the pull back of HYM connections on the complex line bundles $L_{n,l}$ associated with $Q_{n,l}$ over the nearly K\"ahler $\mathbb{F}_3$. 
\end{remark}

\subsubsection{$G=SO(3)$-Bundles}

The possible isotropy homomorphisms $\lambda_{n,l}: T^2 \rightarrow  SO(3)$ are also parametrized by two integers $(n,l) \in \mathbb{Z}^2$. These are constructed by using \ref{isotropyU(1)} from the previous example and letting the image $\mathbb{S}^1$ be the maximal torus in $SO(3)$. Associated with each $\lambda_{n,l}$ is the principal $SO(3)$ bundle $P_{n,l} = SU(3) \times_{T^2, \lambda_{n,l}} SO(3)$. These are reducible and one can also construct there reducible connections induced by the canonical invariant ones on the respective $\mathbb{S}^1$ bundles. Let $T_1,T_2,T_3$ be an orthonormal basis of $\mathfrak{so}(3)$, such that $[T_i,T_j]=2\epsilon_{ijk} T_k$ and fix $\frac{T_1}{2}$ as the generator of the maximal torus. The canonical invariant connection on $P_{n,l}$ is then $A^c_{n,l} = \left( n\theta^1 + l\theta^2 \right) \otimes \frac{T_1}{2}$, with curvature
\begin{equation}\label{SO(3)curvature}
F^c_{n,l}= \left(  -n e_{34} - le_{12} + (l-n) \nu_{12}\right) \otimes T_1.
\end{equation}
Other invariant connections are given by morphisms of $T^2$-representations
$$\Lambda: (\mathfrak{m}_1 \oplus  \mathfrak{m}_2 \oplus  \mathfrak{m}_3 , Ad) \rightarrow  (\mathfrak{so}(3) , Ad \circ \lambda_{n,l}).$$
Let $L_{n,l}$ denote the real two dimensional representation of $T^2$, where the first $\mathbb{S}^1$ acts by rotations with degree $n$ and the second $\mathbb{S}^1$ acts by rotations with degree $l$ (this is the same as the complex representation of $T^2$ induced with weight $(n,l) \in \mathbb{Z}^2$, i.e. by exponentiating $n\theta^1 + l \theta^2 \in (\mathfrak{t}^2)^*$). Identifying the corresponding representations
$$\Lambda: (2,1) \oplus (1,-1)\oplus  (1,2)  \rightarrow  (0,0) \oplus (n,l).$$
These are irreducible and it follows from Schur's lemma, that $\Lambda$ must vanish unless $(n,l)$ is one of $(2,1)$, $(1,2)$, $(1,-1)$. In each of these cases $\Lambda \vert_{m_i}$ is either $0$ or an isomorphism for the corresponding $(n,l)$. Up to invariant gauge transformations such an isomorphism is determined by a constant. Then, it is possible to make $\Lambda$ be one of the following
\begin{eqnarray}
A_{2,1} & = & \left( 2 \theta^1 + \theta^2 \right) \otimes  \frac{T_1}{2} + a \left( \sigma_1 \otimes T_2 + \sigma_2 \otimes T_3 \right) \\ \label{eq:InvariantConnectionGood}
A_{1,-1} & = & \left(  \theta^1 - \theta^2 \right) \otimes \frac{T_1}{2} + a \left( \nu_1 \otimes T_2 + \nu_2 \otimes T_3 \right) \\
A_{1,2} & = & \left( \theta^1 + 2 \theta^2 \right) \otimes  \frac{T_1}{2} + a \left( \Sigma_1 \otimes T_2 + \Sigma_2 \otimes T_3 \right) ,
\end{eqnarray}
with $a \in \mathbb{R}$ a function of the radial coordinate $\rho$. Invariant Higgs fields $\Phi = \Phi(\rho)$ must have values in the components corresponding to the trivial $T^2$-representation, i.e. $\Phi \in (0,0)$ and one writes
\begin{equation}\label{eq:InvariantHiggsF}
\Phi = \phi T_1,
\end{equation}
with $\phi \in \mathbb{R}$ a function of the radial coordinate $\rho$.

\begin{lemma}\label{lemmabundles}
The above $SO(3)$ bundles $P_{n,l}$ for $(n,l)=(2,1),(1,-1),(1,2)$ extend over the zero section giving rise to a bundle over $\Lambda^2_-(\mathbb{CP}^2)$ if and only if $(n,l)=(1,-1)$.
\end{lemma}
\begin{proof}
One needs to show that only when $(n,l)=(1,-1)$ the bundle $E_{n,l} = P_{n,l} \times_{SO(3)} \mathbb{R}^3$ associated via the standard representation is trivial along the fibres of the projection $\pi_2 : \mathbb{F}_3 \rightarrow \mathbb{CP}^2$. Equivalently one can show that only for $(n,l)=(1,-1)$, is the bundle $E_{n,l} \rightarrow \mathbb{F}_3$ isomorphic to a bundle pulled back from $\mathbb{CP}^2$ via $\pi_2$. To do this notice that $w_1(E_{n,l})=0$ for all $(n,l)$, so it is enough to show that $w_2(E_{n,l}), p_1(E_{n,l})$, are pulled back from $\mathbb{CP}^2$ via $\pi_2^*$ only for $(n,l)=(1,-1)$. At this point it is convenient to work with $U(2)$ bundles to compute the characteristic classes. Consider the group homomorphism $\tilde{\lambda}_{n,l} : T^2 \rightarrow U(2)$ given by
$$\tilde{\lambda}_{n,l}(\alpha_1 , \alpha_2) = \left( e^{i\frac{n\alpha_1 +l \alpha_2}{2}}, 
\begin{pmatrix} 
e^{i\frac{n\alpha_1 +l \alpha_2}{2}} & 0 \\
0 &  e^{-i\frac{n\alpha_1 +l \alpha_2}{2}}    \end{pmatrix} \right) \in (U(1) \times SU(2))/\mathbb{Z}_2 \cong U(2).$$
It has the property that after composed with the map $U(2) \rightarrow SO(3)$ given by 
$$A \mapsto diag( det(A)^{-1/2}, det(A)^{-1/2} ) A,$$
it agrees with $\lambda_{n,l}$. Define $W_{n,l}$ as the rank-$2$ complex vector bundle associated via the canonical $U(2)$ representation with $SU(3) \times_{(T^2, \tilde{\lambda}_{n,l})} U(2)$. Then, $\underline{\mathbb{R}} \oplus E_{n,l} \cong \mathfrak{g}_{W_{n,l}}$ and regarding characteristic classes
$$w_2 (E_{n,l}) = c_1(W_{n,l})  \mod 2 \ , \ p_{1}(E_{n,l}) = c_1(W_{n,l})^2 - 4 c_2(W_{n,l}).$$
The canonical invariant connection of such a bundle is $\tilde{A}^c_{n,l} = (n\theta^1 + l \theta^2) \otimes diag(i,0)$, and its curvature is given by $\tilde{F}^c_{n,l}= ( n d \theta^1 + l d\theta^2) \otimes diag(i,0) \in \Omega^2(\mathbb{F}_3, \mathfrak{u}(2))$. Using $c_1(W_{n,l})=i[tr(\tilde{F}^c_{n,l})]$ and $c_2(W_{n,l})= \frac{1}{2}(tr(\tilde{F}^c_{n,l} \wedge \tilde{F}^c_{n,l}) - tr(\tilde{F}^c_{n,l})^2 )$ and inserting the formula above for the curvature gives
$$c_1(W_{n,l}) = -[nd\theta^1 + l d\theta^2] \ , \ c_2(W_{n,l})=0.$$
First focus on $w_2(E_{n,l})$, from lemma \ref{l3} the only classes in $H^2(\mathbb{F}_3, \mathbb{Z})$ which are pulled back from $\mathbb{CP}^2$ via $\pi_2$ are those for which $n=l$. So one can write
$$w_2(E_{n,l}) = [nd\theta^1 + ld\theta^2] = l [d\theta^1 + d\theta^2 ]+ (n-l) [d\theta^1],$$
and this equals $l [d\theta^1 + d\theta^2 ] \in H^2(\mathbb{F}_3, \mathbb{Z}_2)$ if and only if $n-l$ is even. Then $(n,l)=(1,-1)$ is the only case in $(n,l)= \lbrace (2,1),(1,-1),(1,2) \rbrace$ for which this holds. Next one needs to check that $p_1 (E_{1,-1})= c_1(E_{1,-1})^2$ is also the pull back of a class via $\pi_2$. To do this one computes $p_1 (E_{1,-1})= [-2e_{1234}- 4 \nu_{12} \wedge \Omega_1]$ and using the fact that $d \delta = 4 (e_{1234}- \nu_{12}\wedge \Omega_1)$ one concludes that $[4 \nu_{12} \wedge \Omega_1] = [4e_{1234}]$ and so
$$p_1 (E_{1,-1})= [-8 e_{1234}],$$
which is indeed the pullback via $\pi_2$ of a multiple of the fundamental class of $\mathbb{CP}^2$. And so $P_{1,-1}$ does extend over the zero section while the other two cases do not.
\end{proof}

Having in mind this proposition focus for now on the case $(n,l)=(1,-1)$. The curvature of the invariant connection $A_{1,-1}$ is computed via
$$F_{1,-1} = F^c_{1,-1} + d_{A^c_{1,-1}}  a \left( \nu_1 \otimes T_2 + \nu_2 \otimes T_3 \right) + \frac{a^2}{2}  \left[ \left( \nu_1 \otimes T_2 + \nu_2 \otimes T_3 \right) \wedge  \left( \nu_1 \otimes T_2 + \nu_2 \otimes T_3 \right) \right]. $$
Denote these by $I_1,I_2,I_3$ respectively, then $I_1=F^c_{1,-1} = \left( \Omega_{1} -2\nu_{12} \right) \otimes T_1$, is the curvature of the canonical invariant connection. Use the Maurer Cartan relations \ref{MCforSU(3)} to compute the other terms and the dot $\cdot$ to denote differentiation with respect to $s$, then
\begin{eqnarray}\nonumber
I_2 & = & \dot{a} \left( ds \wedge \nu_1 \otimes T_2 + ds \wedge \nu_2 \otimes T_3 \right) + a \left[ \left(  \theta^1 - \theta^2 \right) \otimes  \frac{T_1}{2} \wedge \left( \nu_1 \otimes T_2 + \nu_2 \otimes T_3 \right) \right] \\ \nonumber
& & + a \left( d\nu_1 \otimes T_2 + d\nu_2 \otimes T_3 \right) \\ \nonumber
& = &\dot{a} \left( ds \wedge \nu_1 \otimes T_2 + ds \wedge \nu_2 \otimes T_3 \right) + a (\theta^1 - \theta^2) \wedge \left(  \nu_1 \otimes T_3 - \nu_2 \otimes T_2 \right)  \\ \nonumber
& &  a \left(  \left( \theta^1 - \theta^2 \right) \wedge \nu_2 + \Omega_2 \right) \otimes T_2 - a\left( - \left(  \theta^1 - \theta^2 \right) \wedge \nu_1 - \Omega_3 \right) \otimes T_3 \\ \nonumber
& = & \left( \dot{a}  ds \wedge \nu_1 + a \Omega_2 \right) \otimes T_2 + \left( \dot{a} ds \wedge \nu_2 + a \Omega_3 \right) \otimes T_3 ,
\end{eqnarray}
while
\begin{eqnarray}\nonumber
I_3 & = & \frac{a^2}{2} \left( \nu_{12} \otimes [T_2,T_3] + \nu_{21} \otimes [T_3,T_2] \right) = 2a^2 \nu_{12} \otimes T_1.
\end{eqnarray}
Put all these together and obtain
\begin{eqnarray}\nonumber
F_{1,-1} & = &   \left( 2 (a^2-1) \nu_{12} + \Omega_1 \right) \otimes T_1  + \left( \dot{a}  ds \wedge \nu_1 + a \Omega_2 \right) \otimes T_2 + \left( \dot{a} ds \wedge \nu_2 + a \Omega_3 \right) \otimes T_3 .
\end{eqnarray}
The computation of $F_{A_{1,-1}} \wedge \psi$ requires the $G_2$-structure as computed in section \ref{MetricP2}. It is useful to recall that $2g^2= 4f^{-2}$, which helps in computing
\begin{eqnarray}\label{ladoesquerdo}
F_{A_{1,-1}} \wedge \psi & = & 16 f^{-4} \dot{a} \left(  ds \wedge \nu_1 \otimes T_2 + ds \wedge \nu_2 \otimes T_3 \right) \wedge e_{1234} \\ \nonumber
& & + \left( 32 f^{-4} (a^2-1) + 32s^2 \right) \sigma_{12}  e_{1234} \otimes T_1 + 16 s a \left(    \nu_1 \otimes T_2 + \nu_2 \otimes T_3 \right)  \wedge e_{1234} \\ \nonumber
& = & 32 \left(  f^{-4} a^2 -1 \right) e_{1234}  \nu_{12} \otimes T_1 + 16 \left( f^{-4}\dot{a} + s a \right) ds  e_{1234} \wedge (\nu_1 \otimes T_2+ \nu_2 \otimes T_3).
\end{eqnarray}
The other ingredient of the equations is the covariant derivative of the Higgs field $\Phi = \phi T_1$. The Bianchi identity for the connection $A^c_{1,-1}$ gives $d_{A^c_{1,-1}} T_1=0$ and so inserting this into $\nabla_{A_{1,-1}} \Phi$ gives
\begin{eqnarray}\nonumber
\nabla_{A_{1,-1}} \Phi & = & \nabla_{A_{1,-1}^c} \Phi + \left[ a(\nu_1 \otimes T_2 + \nu_2 \otimes T_3), \phi T_1 \right] \\ \nonumber
 & = & \dot{\phi} \ ds \otimes T_3 + 2a\phi \left( T_2 \otimes \nu_2 - T_3 \otimes \nu_1 \right),
\end{eqnarray}
and
\begin{eqnarray}\label{ladodireito}
\ast \nabla_{A_{1,-1}} \Phi & = & 64 s^2 f^{-1} \dot{\phi} e_{1234} \wedge \nu_{12}  \otimes T_1 + 2a\phi \left( T_2 \otimes  \ast \nu_2 - T_3 \otimes \ast \nu_1 \right) \\ \nonumber
& = & 64 s^2 f^{-1} \dot{\phi} \ e_{1234} \nu_{12}  \otimes T_1  + 32 f^{-3} a \phi \ ds  e_{1234} \wedge \left( \nu_1 \otimes T_2 + \nu_2 \otimes T_3 \right).
\end{eqnarray}

Equating both sides of the monopole equation, i.e. equation \ref{ladoesquerdo} on the left hand side with equation \ref{ladodireito} on the right gives the following set of ODE's
\begin{eqnarray}\label{ola1}
64 s^2 f^{-1} \dot{\phi} & = & 32  (f^{-4}a^2 -1 ) \\  \label{ola2}
16 \left( f^{-4}\dot{a} + s a \right) & = &  32 f^{-3} a \phi .
\end{eqnarray}

\begin{proposition}\label{ODEtransf2}
As a set $\mathcal{M}_{inv}(\Lambda^2_-(\mathbb{CP}^2), P_{1,-1})$ is given by those connections and Higgs fields as in equations \ref{eq:InvariantConnectionGood} and \ref{eq:InvariantHiggsF} such that $(\phi, b=f^{-2}(s^2) a)$ satisfy the ODE's
\begin{eqnarray}\label{ola11}
\frac{d\phi}{d \rho} & = & \frac{1}{2h^2} \left( b^2 -1 \right)  \\ \label{ola22}
\frac{db}{d\rho} & = & 2b \phi .
\end{eqnarray}
with $h^2(\rho) = s^2(\rho) f^{-2}(s^2(\rho)) = s^2(\rho) \sqrt{s^2(\rho) + 1}$ and $b(0)=1,\dot{b}(0)=0$ and $\lim_{\rho \rightarrow + \infty} f^2(s^2) b =0$.
\end{proposition}
\begin{proof}
This amounts to substitute $b=f^{-2}(s^2) a$ and change coordinates from $s$ to $\rho$ in equations \ref{ola1} and \ref{ola2}. The first equation follows immediately and the second one from noticing that $f^{-2}\frac{da}{d\rho} + sf a = \frac{db}{d\rho}$. The initial conditions on $b$ follow from the requirements that the connection and Higgs field extend over the zero section. This requires the curvature of the connection and the Higgs field to be bounded, which requires $\dot{a}(0)=0$ and $a(0)=1$ for the first and $\dot{\phi}(s)$ to be bounded as $s \rightarrow 0$. Since $f(0)=1$ and $\dot{f}(0)=0$ the conditions on $a$ end up being equivalent to $b(0)=1$ and $\dot{b}(0)=0$. From the first ODE and the fact that $h^2(s) \sim s^2$ for small $s$ it follows that these conditions are also sufficient. Recall from proposition \ref{prop:AsymptoticMonAC} that the connection of a finite mass monopole is asymptotic to the pullback of an HYM connection on the nearly K\"ahler $\mathbb{F}_3$. In this case, it must be to $A^c_{1,-1}$ and so $\lim_{\rho \rightarrow + \infty} a =0$, i.e. $\lim_{\rho \rightarrow + \infty} f^2(s^2) b =0$.
\end{proof}

\begin{remark}\label{rem:goodRemark}
The equations in proposition \ref{ODEtransf2} are the same as the ones in proposition \ref{ODEtheorem}. As in there, we have reduced the problem to of solving the ODE's for a spherically symmetric monopole in $\mathbb{R}^3$ (with a given non-Euclidean metric though). Moreover, one can check that $h(\rho) \geq \rho$, is real analytic and as already remarked before behaves like: for small $\rho$, $h(\rho) = \rho + o(\rho^3)$ and for large $\rho$ it grows as $ \rho^3$. 
\end{remark}

\begin{theorem}\label{hconst2}
The moduli space $\mathcal{M}_{inv}(\Lambda^2_-(\mathbb{CP}^2), P_{1,-1})=\mathcal{M}_{inv}$ is not empty and the following hold:
\begin{enumerate}
\item For all $(A, \Phi) \in \mathcal{M}_{inv}$, $\Phi^{-1}(0) = \mathbb{CP}^2$ is the zero section and the mass gives a bijection
$$m : \mathcal{M}_{inv} \rightarrow \mathbb{R}^+.$$
Let $(A_{m}, \Phi_{m}) \in \mathcal{M}_{inv}$ be a monopole with mass $m(A_{\lambda}, \Phi_{\lambda}) = m \in \mathbb{R}^+$, there is a gauge in which
$$(A_m , \Phi_m ) = \left( A^c_{1,-1} + f^2 b_{m}\left( \nu_1\otimes T_2 + \nu_2 \otimes T_3 \right), \phi_{m} T_1  \right),$$
with $b_{m}(0)=1$ and $ \phi_{m}(0)=0$. In this gauge, the curvature of the connection $A_{m}$ is
\begin{eqnarray}\nonumber
F_{A_m} & = &  \frac{d}{d \rho} \left( f^2 b_{m} \right) \left( T_2 \otimes d\rho \wedge \nu_1 + T_3 \otimes d\rho \wedge \nu_3 \right) + f^2 b_{m} \left( T_2 \otimes \Omega_2 + T_3 \otimes \Omega_3 \right)\\
& & + \left(2 \left( f^4 b_{m}^2 -1 \right)\nu_{12} + \Omega_1 \right) \otimes T_1.
\end{eqnarray}
\item Let $R >0$, and $\lbrace (A_{\lambda}, \Phi_{\lambda}) \rbrace_{\lambda \in [\Lambda , + \infty )} \in \mathcal{M}_{inv}(\Lambda^2_-(M),P)$ be a sequence of monopoles with mass $\lambda$ converging to $+\infty$. Then there is a sequence $\eta(\lambda, R)$ converging to $0$ as $\lambda \rightarrow + \infty$ such that for all $x \in M$
$$\exp_{\eta}^* (A_{\lambda}, \eta \Phi_{\lambda}) \vert_{\Lambda^2_-(M)_x}$$
converges uniformly to the BPS monopole $(A^{BPS}, \Phi^{BPS})$ in the ball of radius $R$ in $(\mathbb{R}^3,g_E)$. Here $\exp_{\eta}$ denotes the exponential map along the fibre $\Lambda^2_-(M)_x \cong \mathbb{R}^3$.
\item Let $\lbrace (A_{\lambda}, \Phi_{\lambda}) \rbrace_{\lambda \in [\Lambda, +\infty)}$ be the sequence above. Then the translated sequence
$$\left(A_{\lambda}, \Phi_{\lambda}- \lambda \frac{\Phi_{\lambda}}{\vert \Phi_{\lambda} \vert} \right),$$
converges uniformly with all derivatives on $(\Lambda^2_-(\mathbb{S}^4) \backslash \mathbb{S}^4, g )$, to a monopole reducible to the mass $0$ Dirac monopole $(A^D, \Phi^D_0)$.
\end{enumerate}
\end{theorem}
\begin{proof}
The proof is the same as in theorem \ref{hconst} and amounts to use remark \ref{rem:goodRemark} to reduce the proof to the analyzes in the Appendix.
\end{proof}

\begin{remark}
\begin{itemize}
\item These monopoles converge to the canonical invariant connection $A^c_{1,-1}$. This is reducible to a HYM connection on $L_{1,-1}$ over the nearly K\"ahler $\mathbb{F}_3$ as alluded in remark \ref{rem:HYMCP2}. In fact, $c_1(L_{1,-1})$ is a monopole class as defined in section $4.1.3$ of \cite{Oliveira2014}.
\item The energy of these monopoles is not finite (as they are asymptotic to a nonflat connection on $\mathbb{F}_3$). However, the Intermediate energy is indeed finite and the formula \ref{G2idIntermediate} in proposition \ref{EnergyIdentityProp} can be used to compute
\begin{eqnarray}\nonumber
E^I (A_m, \Phi_m) & = & \lim_{\rho \rightarrow \infty} 2\phi_m(\rho) \int_{\mathbb{F}^3} 16 e_{1234} \wedge 2\nu_{12} = 4 \pi m \langle [\mathbb{F}_3] , c_1(L_{1,-1}) \cup [i^* \psi] \rangle.
\end{eqnarray}
\end{itemize}
\end{remark}

The next result regards the bundles $P_{1,2}$ as well as $P_{2,1}$. Recall from lemma \ref{lemmabundles} that these do not extend over the zero section and so are defined on $\Lambda^2_-(\mathbb{CP}^2) \backslash \mathbb{CP}^2$. However the monopole equations can still be integrated to give monopoles on the complement of the zero section and in the following result these solutions are shown to have unbounded Higgs fields directly from the ODE's. In fact from the proof one can see that the Higgs field explodes as we approach the zero section.

\begin{proposition}
There is no smooth invariant monopole on the bundles $P_{2,1}$ and $P_{1,2}$ over $\Lambda^2_-(\mathbb{CP}^2) \backslash \mathbb{CP}^2$, whose Higgs field is bounded.
\end{proposition}
\begin{proof}
Start with the case $(n,l)=(2,1)$, most of the computations are similar to the ones above and so will be omitted. In the case at hand, there are no solutions to the monopole ODE's that can be extended to the zero section as the bundle itself does not extend over the zero section as shown in lemma \ref{lemmabundles}. The curvature and covariant derivative of the Higgs field are respectively given by
\begin{eqnarray}\nonumber
F_{2,1} & = &  -\dot{a} \left( ds \wedge e_3 \otimes T_2 + ds \wedge e_4 \otimes T_3 \right) + \left( -2 (a^2-1) \Omega_1 - \nu_{12} \right) \otimes T_1  \\  \nonumber
& &+ a \left( \nu_2 \wedge e_2  -  \nu_1 \wedge e_1  \right) \otimes T_2 - a \left(  \nu_1 \wedge e_2 + \nu_2 \wedge e_1 \right) \otimes T_3 \\ \nonumber
\nabla_{A_{2,1}} \Phi & = & \dot{\phi} \ ds \otimes T_1 - 2a\phi \left( T_2 \otimes e_3 - T_3 \otimes e_4 \right),
\end{eqnarray}
Equating $\ast \nabla_{A_{2,1}} \Phi = F_{2,1} \wedge \psi$ gives the following equations
\begin{eqnarray}\nonumber
\frac{d\phi}{d\rho} = -\frac{1}{2h^{2}} \left( b^2 +1 \right)  \ , \  \frac{db}{d\rho} = -2b \phi ,
\end{eqnarray}
where $b= s a$ and as in the previous section $h^2( \rho)= s^2(\rho) \sqrt{s^2(\rho) + 1}$. These equations will never give bounded solutions. In fact notice that since $1+b^2 > 0$ and $h(\rho) = \rho + ...$, for small $\rho$, we have $\frac{d \phi}{d \rho}= - \frac{1}{2 \rho^2}+ ...$ and so $\phi(\rho)= \frac{1}{2 \rho}+...$ can not be bounded as $\rho \rightarrow 0$. The case $(n,l)=(1,2)$ is similar
\begin{eqnarray}\nonumber
F_{1,2} & = &  \dot{a} \left( ds \wedge e_1 \otimes T_2 + ds \wedge e_2 \otimes T_3 \right) + \left( 2 (a^2-1) \Omega_1 + \nu_{12} \right) \otimes T_1  \\ \nonumber
& & - a \left(  \nu_1 \wedge e_3 + \nu_2 \wedge e_4 \right) \otimes T_2 - a \left(  \nu_1 \wedge e_4 - \nu_2 \wedge e_3 \right) \otimes T_3 \\ \nonumber
\nabla_{A_{1,2}} \Phi & = & \dot{\phi} \ ds \otimes T_1 + 2a\phi \left( T_2 \otimes e_2 - T_3 \otimes e_1 \right),
\end{eqnarray}
and the monopole equations for these with $b=sa$ and $h^2(\rho)= s^2(\rho) \sqrt{s^2( \rho) + 1}$ as before, are
\begin{eqnarray}\nonumber
\frac{d\phi}{d\rho}  = \frac{1}{2h^{2}} \left(1 + b^2 \right)  \ , \ \frac{db}{d\rho} = 2b \phi .
\end{eqnarray}
Once again as it was the case for $(n,l)=(1,2)$, the Higgs field is unbounded at the zero section.
\end{proof}

\subsubsection{$G= SU(3)$-Bundles}

For gauge group $G=SU(3)$, the possible isotropy homomorphisms $\lambda: T^2 \rightarrow SU(3)$ are parametrized by automorphisms of $T^2$ by identifying the image $T^2$ with the maximal torus in $SU(3)$. These depend on four integers $(n_{11},n_{12},n_{21},n_{22}) \in \mathbb{Z}^4$ each corresponding to the degree of a different map $\pi_{i} \circ \lambda \circ i_j : \mathbb{S}^1 \rightarrow \mathbb{S}^1$. Explicitly, such an homomorphism is given by
\begin{equation}\label{eq:SU(3)IH}
\lambda(e^{i \alpha_1},e^{i\alpha_2}) = i(e^{i (n_{11}\alpha_1 + n_{12} \alpha_2)} ,e^{i (n_{21}\alpha_1 + n_{22} \alpha_2)} ),\end{equation}
where $i: T^2 \hookrightarrow SU(3)$ is a fixed embedding of the maximal torus (as in \ref{maximaltori}). For each of these homomorphisms one obtains a bundle $P_{\lambda} = SU(3) \times_{\lambda} SU(3)$. The reductive decomposition \ref{reductive} equips each of these with a canonical invariant connection $A^c_{\lambda} = \left( n_{11} X_1 + n_{21}X_2 \right) \otimes \theta^1 + \left( n_{12} X_1 + n_{22}X_2 \right) \otimes \theta^2$, whose curvature is represented by the horizontal form
\begin{eqnarray}\nonumber
F^c_{\lambda} & = & -2 \left( n_{11} X_1 + n_{21}X_2 \right) \otimes ( e_{34} + \nu_{12}) + 2 \left( n_{12} X_1 + n_{22}X_2 \right) \otimes ( \nu_{12} - e_{12} ) \\  \nonumber
& = & -2 \left( n_{11} X_1 + n_{21}X_2 \right) \otimes e_{34}  -2  \left( n_{12} X_1 + n_{22}X_2 \right) \otimes e_{12} \\
& & + 2 \left( (n_{12} -n_{11})X_1 + (n_{21} + n_{22})X_2 \right) \otimes \nu_{12}. 
\end{eqnarray}
Other invariant connections are given by morphisms of $T^2$-representations $\Lambda : ( \mathfrak{m}, Ad ) \rightarrow ( \mathfrak{su}(3), Ad \circ \lambda )$. The following lemma is a tautology which will be helpful in decomposing the right hand side into irreducible components

\begin{lemma}
Let $\exp (i h) : T^n \rightarrow \mathbb{C}^* = GL(\mathbb{C})$ be an irreducible $T^n $ representation with weight vector $dh \in (\mathfrak{t}^n)^*$, and $\lambda: T^n \rightarrow T^n$ is a group homomorphism, then $\exp(ih) \circ \lambda = \exp(i\lambda^* h)$.
\end{lemma}

Since as $T^2$-representations $(\mathfrak{m}_{\mathbb{C}}, Ad) =(2,1) \oplus (1,-1) \oplus (1,2)$, and $(n,l) \in \mathbb{Z}^2$ denotes the representation $e^{i (n\alpha_1 + l \alpha_2) }$ and $\mathfrak{su}_{\mathbb{C}}(3)= \mathfrak{t}^2_{\mathbb{C}} \oplus \mathfrak{m}_{\mathbb{C}})$ the lemma splits the $(\mathfrak{su}_{\mathbb{C}}(3), Ad \circ \lambda)$ representation in the right hand side, as
$$\mathbb{C} \oplus \mathbb{C} \oplus (2n_{11} +n_{21} ,2n_{12} + n_{22}) \oplus (n_{11} -n_{21} ,n_{12} - n_{22})  \oplus (n_{11} +2n_{21} ,n_{12} +2 n_{22}).$$

Restrict to the special case where $n_{11}=n_{22}=1$ and $n_{21} = n_{12}=0$ and denote the bundle so obtained by $P_{\lambda}$. Pick an invariant connection given by $\Lambda: \mathfrak{m} \rightarrow \mathfrak{m}$ an isomorphism of representations. For each $\rho \in \mathbb{R}^+$, these depend on $a_1,a_2,a_3 \in \mathbb{R}$, each corresponding to a scaling factor associated with an isomorphism between the corresponding irreducible components and induce the connection $1$-forms given by
\begin{eqnarray}\label{SU(3)connection}
A_{\lambda} & = & X_1  \otimes \theta^1 + X_2 \otimes \theta^2 - a_1 \left( C_{13} \otimes e_3 + D_{13} \otimes e_4 \right)\\ \nonumber
& & + a_2 \left( C_{12} \otimes \nu_1 + D_{12} \otimes \nu_2 \right) + a_3 \left( C_{23} \otimes e_1 + D_{23} \otimes e_2 \right) .
\end{eqnarray}
After a computation which is omitted the curvature is
\begin{eqnarray}\nonumber
F_{\lambda} & = & - \frac{da_1}{d \rho} \left( C_{13} \otimes d\rho \wedge e_3 + D_{13} \otimes d\rho \wedge e_4 \right) + \frac{da_2}{d\rho} \left( C_{12} \otimes d\rho \wedge \nu_{1} + D_{12} \otimes d\rho \wedge \nu_2 \right) \\ \nonumber
& & +  \frac{da_3}{d\rho} \left( C_{23} \otimes d\rho \wedge e_{1} + D_{23} \otimes d\rho \wedge e_2 \right) \\ \nonumber
& & + X_1 \otimes \left( 2(a_1^2 -1) e_{34} + 2(a_2^2 -1)\nu_{12} \right) + X_2 \otimes \left( 2(a_3^2 -1)e_{12} + 2(1-a_2^2) \nu_{12} \right) \\ \nonumber
& & + \left( a_1 - a_2 a_3 \right) \left( C_{13} \otimes (-\nu_1 e_1 +\nu_2 e_2) - D_{13} \otimes (\nu_1 e_2 + \nu_2 e_1) \right) \\ \nonumber
& & + \left( a_2 - a_1 a_3 \right) \left( C_{12} \otimes \Omega_2 + D_{12} \otimes \Omega_3 \right) \\ \label{SU(3)curv}
& & + \left( a_3 - a_1 a_2 \right) \left( C_{23} \otimes (\nu_1 e_1 +\nu_2 e_2) + D_{23} \otimes (\nu_1 e_2 - \nu_2 e_1) \right) .
\end{eqnarray}

\begin{remark}\label{FlatConnection1}
The connection extends over the zero section to a connection on the whole $\Lambda^2_-(\mathbb{CP}^2)$ if and only if its curvature \ref{SU(3)curv} is bounded. This is equivalent to the statement that $a_2^2(0)=1$, $\dot{a}_2(0)=0$ and $a_1(0)= a_2(0)a_3(0)$. For example, the special cases where $a_2=-1, a_2=a_3=\pm 1$ and $a_2=1, a_1 = - a_3= \pm 1$, can be easily checked (using the formula above) to give rise to flat connections and these do extend over the zero section.
\end{remark}

The invariant Higgs field $\Phi \in \Omega^0(SU(3), \mathfrak{su}(3))$ must have values in $\mathfrak{t}^2 \subset \mathfrak{su}(3)$, so can be written $ \Phi = \phi_1 \ X_1 + \phi_2 \ X_2$, where $\phi_1, \phi_2$ are functions of the radial coordinate. After a short computation
\begin{eqnarray}\nonumber
\nabla_A \Phi & = & \frac{d\phi_1}{d \rho}  d\rho \otimes  X_1 +\frac{d\phi_1}{d\rho} d\rho \otimes  X_2+  a_1(2\phi_1 + \phi_2) \left( D_{13} \otimes e_3 - C_{13} \otimes e_4 \right) \\ \nonumber
& & -a_2 (\phi_1 - \phi_2)\left( D_{12} \otimes \nu_1 - C_{12} \otimes \nu_2 \right)  -a_3 (\phi_1 + 2\phi_2)\left( D_{23} \otimes e_1 - C_{23} \otimes e_2 \right).
\end{eqnarray}
Omitting some more computations the monopole equation $F_A \wedge \psi = \ast \nabla_A \Phi$ gives rise to the following set of ODE's
\begin{eqnarray}\nonumber
\frac{d\phi_1}{d \rho} & = & \frac{1}{2h^2} \left(  f^{-4} a_2^2 -s^2a_1^2 -1 \right) \\ \nonumber
\frac{d\phi_2}{d\rho} & = & \frac{1}{2h^2}   \left( s^2 a_3^2 -  f^{-4} a_2^2 +1 \right)  \\ \nonumber
 s\frac{da_1}{d\rho} + f^{-1} \left( a_1 - a_2 a_3 \right)  & = &- sa_1(2\phi_1 + \phi_2) \\ \nonumber
 f^{-2}\frac{da_2}{d\rho}  +sf  \left( a_2 - a_1 a_3 \right)  & = &  f^{-2}a_2(\phi_1 - \phi_2) \\ \nonumber
s \frac{da_3}{d\rho} + f^{-1}  \left( a_3 - a_1 a_2 \right) & = & sa_3(\phi_1 + 2\phi_2),
\end{eqnarray}
where $h^2(\rho)= s^2 (\rho) f^{-2}(s(\rho))=s^2(\rho) \sqrt{s^2(\rho)+1}$. Introduce the rescaled fields $b_1 = sa_1$, $b_2 = f^{-2}a_2$, $b_3 = s a_3$. Then the ODE's above can be written as
\begin{eqnarray}\nonumber
\frac{d\phi_1}{d\rho} & = & \frac{1}{2h^2} \left(  b_2^2 - b_1^2 -1 \right) \\ \nonumber
\frac{d\phi_2}{d\rho} & = & \frac{1}{2h^2}   \left( b_3^2 - b_2^2 +1 \right)  \\ \nonumber
 \frac{db_1}{d\rho}  & = & \frac{f}{s} b_2 b_3 - b_1(2\phi_1 + \phi_2) \\ \nonumber
\frac{db_2}{d\rho} & = &  \frac{f}{s} b_1 b_3 +  b_2(\phi_1 - \phi_2) \\ \nonumber
\frac{db_3}{d\rho}  & = &  \frac{f}{s} b_1 b_2 +  b_3(\phi_1 + 2\phi_2) .
\end{eqnarray}

\begin{theorem}\label{SU(3)monopole}
There is a $1$-parameter family of solutions to the system of equations above, parametrized by their mass $m \in \mathbb{R}^+$. Moreover, such a solution gives rise to a smooth $G_2$-monopole on the bundle $P_{\lambda}$, which in the previous gauge is given by the Higgs field $\Phi = \phi_{m} (X_1 - X_2)$ and the connection $A_{m} = X_1  \otimes \theta^1 + X_2 \otimes \theta^2 +f^2 a_{m} \left( C_{12} \otimes \nu_1 + D_{12} \otimes \nu_2 \right)$, whose curvature is
\begin{eqnarray}\nonumber
F_{m} & = &  \left( -2 e_{34} + 2(f^4 a^2_{m} -1)\nu_{12} \right) \otimes X_1+  \left( -2 e_{12} + 2(1-f^4 a^2_{m}) \nu_{12} \right) \otimes X_2 \\ \nonumber
& & + \left(f^4 a^2_{m} \Omega_2  + \frac{d}{d\rho} \left( f^4 a^2_{m} \right) d\rho \wedge \nu_{1} \right) \otimes C_{12} + \left( f^4 a^2_{m} \Omega_3 + \frac{d}{d\rho} \left( f^4 a^2_{m} \right) d\rho \wedge \nu_2  \right) \otimes D_{12} .
\end{eqnarray}
\end{theorem}
\begin{proof}
The particular solutions stated above follow from an ansatz that reduces the system to the same ODE's that have been obtained in all the other cases (i.e. the ones for spherically symmetric monopoles in $(\mathbb{R}^3, d\rho^2 + h^2(\rho)g_{\mathbb{S}^2})$). Set $b_1 = b_3 = 0$, then the third and fifth equations are trivially satisfied. The other equations are
\begin{eqnarray}\nonumber
\frac{d\phi_1}{d\rho} & = & - \frac{d\phi_2}{d\rho} = \frac{1}{2h^2} \left(  b_2^2  -1 \right) \\ \nonumber
\frac{db_2}{d\rho} & = &  b_2(\phi_1 - \phi_2).
\end{eqnarray}
If it is further supposed that $\phi_1=-\phi_2=\phi$, one obtains
\begin{eqnarray}\nonumber
\frac{d\phi}{d\rho} & = & \frac{1}{2h^2} \left(  b_2^2  -1 \right) \\ \nonumber
\frac{db_2}{d\rho} & = &  2 \phi b_2,
\end{eqnarray}
and the existence result from the Appendix \ref{AppendixA} can be applied, to give a family of solutions. These are parametrized by $m \in \mathbb{R}^+$ and given by $(\phi, b)=(\phi_m, a_m)$, where $(\phi_m, a_m)$ is the solution provided by theorem \ref{teofinal}. One then computes the formula in the statement for their curvature, which is bounded at $\rho=0$ and so extends to a solution on $\Lambda^2_-(\mathbb{CP}^2)$, see remark \ref{FlatConnection1}.
\end{proof}

\begin{remark}
The Ambrose-Singer theorem identifies the Lie algebra of the holonomy group with the values of the curvature. This allows the conclusion that the holonomy of the monopoles above is contained in a $S(U(1) \times SU(2))$ subgroup of $SU(3)$ and so these are reducible to $SU(2)$.
\end{remark}

\subsection{$G_2$-instantons}
This subsection constructs $G_2$-instantons on bundles over $\Lambda^2_-(\mathbb{CP}^2)$ equipped with the Bryant-Salamon $G_2$ structure. One must remark that $G_2$-instantons for the the other Bryant-Salamon metrics also exist, as constructed in subsection \ref{sec:G2Inst} for $\Lambda^2_-(\mathbb{S}^4)$ and by Andrew Clarke in \cite{Clarke14} for $\mathcal{S}(\mathbb{S}^3)$.

\subsubsection{$G=\mathbb{S}^1$-Bundles}

In the case $n=l$, lemma \ref{H2} states that the bundle is $P_{n,n}= \pi_2^*\mathcal{O}_{ \mathbb{CP}^2}(-n)$. The bundles $\mathcal{O}_{ \mathbb{CP}^2}(-n)$ are self-dual and one can check that the canonical invariant connection associated with these will give rise to a $G_2$-instanton on $\Lambda^2_-(\mathbb{CP}^2)$. This is stated in proposition \ref{prop:S1Monopoles}.

\subsubsection{$G=SO(3)$-Bundles}

Irreducible $G_2$-instanton in the bundle $P_{1,-1}$ can be obtained by solving the ODE's in proposition \ref{ODEtransf2} for $\phi=0$. This implies $b^2=1$, i.e. $b= \pm 1$, $a= \pm  f^2(s^2)$ and $\frac{da}{ds} = \mp sf^6(s^2)$, the solution is a smooth irreducible $G_2$-instanton on $P_{1,-1} \rightarrow \Lambda^2_-(\mathbb{CP}^2)$.

\begin{theorem}\label{insttheorem2}
The connection on $P_{1,-1}$ over $\Lambda^2_-(\mathbb{CP}^2)$ given by $A = A^c_{1,-1} + f^2(s^2)\left( \nu_1 \otimes T_2 + \nu_2 \otimes T_3 \right)$ is an irreducible $G_2$-instanton with curvature
\begin{eqnarray}\nonumber
F_{1,-1} & = & \frac{2s^2}{s^2 + 1} \nu_{12}  \otimes T_1 +\Omega_1 \otimes T_1 \pm \frac{1}{\sqrt{s^2 + 1}}\left( \Omega_2 \otimes T_2 + \Omega_3 \otimes T_3 \right) \\  \nonumber
& &  \mp \frac{s}{(1 + s^2 )^{\frac{3}{2} } } \left(  ds \wedge \nu_1 \otimes T_2 + ds \wedge \nu_2 \otimes T_3 \right).
\end{eqnarray}
\end{theorem}

\begin{remark}
This instanton converges to the canonical invariant connection $A^c_{1,-1}$, which recall is the pullback to the cone of a reducible HYM connection on $\mathbb{F}_3$ equipped with its standard nearly K\"ahler structure. However, in contrast with the monopoles from theorem \ref{hconst2} which converge to $A^c_{1,-1}$ at an exponential rate, the instantons from theorem \ref{insttheorem2} converge to $A^c_{1,-1}$ at a polynomial rate.
\end{remark}

\subsubsection{$G=SU(3)$-Bundles}

To obtain irreducible $G_2$-instantons, one solves the system of ODE's \ref{SU(3)monopole}.

\begin{theorem}\label{SU(3)instantons}
Let $P_{\lambda}$ be the $SU(3)$-principal bundle obtained via the isotropy homomorphism in equation \ref{eq:SU(3)IH} with $n_{11}=n_{12}=n_{21}=n_{22}=1$. There are two families of irreducible $G_2$-instantons on $P_{\lambda}$ parametrized by $c \geq 0$. These are respectively given by
\begin{eqnarray}
A_{\lambda} & = & X_1  \otimes \theta^1 + X_2 \otimes \theta^2 - \frac{ u_{c}(s)}{\sqrt{1+s^2}} \left( C_{12} \otimes \nu_1 + D_{12} \otimes \nu_2 \right) \\
& & \mp \frac{\sqrt{u_c^2(s)-1}}{s} \left( C_{13} \otimes e_3 + D_{13} \otimes e_4 - C_{23} \otimes e_1 - D_{23} \otimes e_2 \right)
\end{eqnarray}
and
\begin{eqnarray}
A_{\lambda} & = & X_1  \otimes \theta^1 + X_2 \otimes \theta^2 + \frac{ u_{c}(s)}{\sqrt{1+s^2}} \left( C_{12} \otimes \nu_1 + D_{12} \otimes \nu_2 \right) \\
& & \mp \frac{\sqrt{u_c^2(s)-1}}{s} \left( C_{13} \otimes e_3 + D_{13} \otimes e_4 + C_{23} \otimes e_1 + D_{23} \otimes e_2 \right),
\end{eqnarray}
where
\begin{equation}
u_c(s)= 1 - 2c \frac{ s^2 }{s^2(1+c) + 2 \left( \sqrt{1+ s^2} +1 \right)} .
\end{equation}
In particular, the case $c=-1$, recovers the flat connections alluded to in remark \ref{FlatConnection1}.
\end{theorem}

\begin{proof}
For $\Phi=0$ one has to set $\phi_1=\phi_2=0$ in the system of equations above. This gives the equations
\begin{eqnarray}\nonumber
1 & = & b_2^2 - b_3^2 = b_2^2 - b_1^2 \\ \nonumber
\frac{db_i}{d\rho} & = & \frac{f}{s} b_j b_k ,
\end{eqnarray}
for $i,j,k \in {1,2,3}$ and $i \neq j \neq k$. In order to guarantee that the connection extends over the zero section, its curvature must be bounded and from remark \ref{FlatConnection1} together with the definitions of the $b_i's$ one must have $b_2(0)^2= a_2(0)^2 =1$, $b_1(0)= b_2(0)=0$ and $\dot{b}_3(0)= (-1)^k\dot{b}_1(0)$, where $a_2(0)=(-1)^k$. Moreover, from the equations above $\frac{db_1^2}{d\rho}=\frac{db_2^2}{d\rho}=\frac{db_3^2}{d\rho}=\frac{f}{s} b_1 b_2 b_3$ and so the three last equations are indeed compatible with the constraints imposed by the first two ones. These also imply that $b_1=\pm b_3=(-1)^k b_3$, and the system gets reduced to solve
\begin{eqnarray}\nonumber
b_2^2 - b_1^2 & = & 1 \\ \nonumber
\frac{db_1}{d\rho} & = & \frac{f}{s} (-1)^k b_2 b_1 \\ \nonumber
\frac{db_2}{d\rho} & = &  \frac{f}{s} (-1)^k b_1^2 .
\end{eqnarray}
Inserting the first equation (the constraint) into the last one and using $\frac{d}{d\rho} = f^{-1} \frac{d}{ds}$ gives the following nonlinear singular ODE
\begin{equation}
\frac{db_2}{ds} = (-1)^k \frac{f^2}{s} ( b_2^2-1). 
\end{equation}
For $k$ even there is a $1$-parameter family of solutions given by $b_2(s)=-u_c(s)$, for all $c \geq -1$ and $b_1(s) = b_3(s) = \pm \sqrt{u_c^2(s)-1}$. In the same way for $k$ odd, there is a $1$-parameter given by $b_2(s)=u_c(s)$ for all $c \geq 0$ and so $b_1(s) = \pm \sqrt{u_c^2(s)-1}$. These give rise to the connections on the statement and to check the connections extend one needs to show that $\frac{\sqrt{u_{c}^2(s)-1}}{s}$ is bounded at $s=0$ which is indeed the case.
\end{proof}

\appendix

\section{Symmetric Monopoles on $\mathbb{R}^3$}\label{AppendixA}

Let $g$ be a spherically symmetric metric on $\mathbb{R}^3$. Then, on $\mathbb{R}^3 \backslash \lbrace 0 \rbrace = \mathbb{R}^+ \times \mathbb{S}^2$, one can write
\begin{equation}\label{invariantmetric}
g = dr^2 + h^2(r) g_{\mathbb{S}^2},
\end{equation}
with $h(r)= r + h_3 r^3 + ...$, in order for the metric to be smooth and have bounded curvature at $r=0$. This section studies spherically symmetric monopoles on the trivial $SU(2)$ bundle over $(\mathbb{R}^3,g)$. Under suitable conditions on $h$ spherically symmetric solutions are constructed and these solve a system of nonlinear first order ODE's for two real valued functions $a, \phi$. These ODE's  have a singularity at $r=0$ and are given by
\begin{eqnarray}\label{monopoleode}
\dot{\phi} & = & \dfrac{1}{2 h^2} (a^2 -1 ) \\ \label{monopoleode2}
\dot{a} & = &  2 \phi a.
\end{eqnarray}
together with the conditions $a(0)=1 , \phi(0)=0$ and that $a$ grows at most polynomially in $r$, i.e. $\lim_{r \rightarrow + \infty} r^{-k} a =0$, for some $k \in \mathbb{Z}$. The first two of these are necessary and sufficient to guarantee the solution extends over $r=0$ (they guarantee the curvature and the Higgs field are bounded \cite{2Sibner1984}). To understand the third condition, recall that there is a unique spherically symmetric connection $\nabla_{\infty}$ on the Hopf bundle over the $\mathbb{S}^2$. Then, one must require that over the $2$ sphere at infinity, the connection is asymptotic to the reducible connection induced by $\nabla_{\infty}$. Using any metric with polynomial volume growth (the Euclidean metric for example) in order to compare connections certainly implies the condition that $a$ must grow at most polynomially. In fact, for the applications in the current thesis, the metric $g$ itself has polynomial volume growth and requiring that the connection is asymptotic to $\nabla_{\infty}$ with respect to $g$ does imply that there is $k \in \mathbb{Z}$ such that $\lim_{r \rightarrow + \infty} r^{-k} a =0$.\\

Notice that in case $\phi$ does not explode at a finite $r$, then $sign(a)$ is preserved by the evolution. As changing $a$ by $-a$ keeps the equations invariant there is no loss in restricting to the case $a>0$. All the results of this section can be interpreted as properties of this system of ODE's and that is in fact the relevant point of view for the applications in the current thesis.
The moduli space $\mathcal{M}_{inv}$ of spherically invariant monopoles on $(\mathbb{R}^3,g)$ modulo the action of the spherically symmetric gauge transformations is defined by
\begin{equation}
\mathcal{M}_{inv} = \Big\lbrace (a, \phi ) \ \Big\vert \ \text{solving \ref{monopoleode} with $a(0)=1 , \phi(0)=0$ and $\exists_{k \in \mathbb{Z}} \lim_{r \rightarrow + \infty} r^{-k}a =0$} \Big\rbrace.
\end{equation}
The metric $g$ will be called non-parabolic if its Green's function $G$ is bounded above, then it is uniquely defined by
$$G(r) = - \int \frac{1}{2h^2(r)} dr \ , \ \lim_{r \rightarrow \infty} G=0.$$
It will be shown that spherically invariant solutions to the Bogomolnyi equations actually have bounded Higgs field $\Phi$ and a well defined mass
$$m(A,\Phi)= \lim_{r \rightarrow \infty} \vert \Phi (r) \vert.$$
The Bogomolny equation inherits a scaling property from the conformal invariance of the ASD equations in $4$ dimensions. The precise result is
\begin{proposition}\label{scaling}
Let $(\nabla_A, \Phi)$ be a monopole on $(M^3,g)$, where $M^3$ is a Riemannian $3$ manifold. Then $(\nabla_A , \delta^{-1} \Phi)$ is a monopole for $(M^3 , \tilde{g} = \delta^2 g)$.
\end{proposition}
\begin{proof}
In general, if $\omega$ is a $k$ form and $\tilde{\ast}$ the Hodge operator for the metric $\tilde{g}$, then $\tilde{\ast} \omega = \delta^{n-2k} \ast \omega$ ($n=3$). This implies that $\tilde{\ast} F_A = \delta^{-1} \ast F_A = \delta^{-1} \nabla_A \Phi$, and the result follows.
\end{proof}

This scaling behavior will be very important for our purposes in the following way. Let $s_{\delta}(x) = \delta x$ denote the scaling map on $\mathbb{R}^3$. Then, proposition \ref{scaling} maps a monopole $(A,\Phi)$ for the metric $g$ into a monopole $s_{\delta}^*(A, \delta \Phi)$ for the metric $\delta^{-2}s_{\delta}^* g$. In the case where $g=g_E$ is the Euclidean metric there is a unique mass $1$ and charge $1$ monopole known as the $BPS$ monopole \cite{BPS}, this is spherically symmetric and denoted by $(A^{BPS}, \Phi^{BPS})$. Moreover, the Euclidean metric is scale invariant and so from $(A^{BPS}, \Phi^{BPS})$ one can construct a whole family of monopoles $(A_{m}^{BPS}, \Phi_{m}^{BPS})=s_{m}^*(A^{BPS}, m \Phi^{BPS})$, related by scaling and parametrized by their mass $m \in \mathbb{R}^+$. The solutions constructed in this appendix are modeled on these and the main result is

\begin{theorem}\label{teofinal}
Let $g$ be spherically symmetric, real analytic and non-parabolic. Then, $\mathcal{M}_{inv}$ is nonempty and consists of real analytic monopoles. Moreover, the following hold:
\begin{enumerate}
\item For all monopoles in $\mathcal{M}_{inv}$, the Higgs field is bounded and $\Phi^{-1}(0)=0$ is the origin in $\mathbb{R}^3$. Moreover, the mass is well defined and gives a bijection
$$m : \mathcal{M}_{inv} \rightarrow \mathbb{R}^+.$$
\item Let $\lbrace (A_{\lambda}, \Phi_{\lambda}) \rbrace_{\lambda \in [\Lambda, +\infty)} \in \mathcal{M}_{inv}$ a sequence of monopoles with mass $\lambda$ converging to $+\infty$. Then, for all $R>0$ there is a sequence $\eta(\lambda, R)$ converging to $0$ as $\lambda$ converges to $+ \infty$, such that the rescaled monopole
$$s_{\eta}^* (A_{\lambda}, \eta \Phi_{\lambda})$$
converges uniformly with all derivatives to the BPS monopole $(A^{BPS}, \Phi^{BPS})$ in the ball of radius $R$ in $(\mathbb{R}^3,g_E)$.
\item Let $\lbrace (A_{\lambda}, \Phi_{\lambda}) \rbrace_{\lambda \in [\Lambda, +\infty)}$ be the sequence above. Then the translated sequence
$$\left( A_{\lambda}, \Phi_{\lambda}- \lambda \frac{ \Phi_{\lambda}}{\vert \Phi_{\lambda} \vert } \right),$$
converges uniformly with all derivatives  on $(\mathbb{R}^3 \backslash \lbrace 0 \rbrace, g )$ to a reducible monopole made of two copies of the $g$-Dirac monopole $(A^D, \Phi^D = G)$ with zero mass.
\end{enumerate}
\end{theorem}

\begin{remark}\label{remark5}
The above statement is not at all surprising and in fact it is possible to prove that if in the complement of some ball $h^2(r) \geq cr^{1+ \epsilon}$, for some $c,\epsilon > 0$ ($g$ is non-parabolic in this case). Then, there is a spherically symmetric finite energy solution to the Yang-Mills-Higgs equations $d_A^{\ast} F_A = [\nabla_A \Phi , \Phi]$, $\Delta_A \Phi = 0$ in $(\mathbb{R}^3,g)$ with bounded Higgs field. This can be achieved by direct minimization of the spherically invariant Yang-Mills-Higgs functional on $(\mathbb{R}^3,g)$.
\end{remark}

The proof of theorem \ref{teofinal} occupies the rest of this section, which is organized in the following way.
In section \ref{SectionInv} the reduction to an ODE of the Bogomolny equations in $(\mathbb{R}^3,g)$ is outlined and an explicit formula for the BPS monopole with the Euclidean metric is given. For general spherically symmetric metrics $g$, the solutions to the ODE's \ref{monopoleode} are not known. Besides these ODE's being nonlinear, there is a singularity at the origin, $r=0$. The initial conditions one would like to give at $r=0$ do not satisfy the Lipschitz hypothesis required by the standard existence and uniqueness theorem for ODE's. It is then convenient to go back to elliptic PDE theory and obtain a solution on the ball $B_{\delta}(0)$ which can be used to give initial conditions at the Lipschitz point $r= \delta$.
Instead of solving the monopole equations for the metric $g$ in the ball $B_{\delta}$, use the scale invariance of the Bogomolny equations stated in proposition \ref{scaling} in order to solve the equations for the metric $g_{\delta} = \delta^{-2}s_{\delta}^*g$ on its unit ball. Then, one obtains

\begin{proposition}\label{prop:ExistenceOfMonopoles}
For each $m \in \mathbb{R}^+$, there is $\Delta(m)>0$, such that for each $\delta \leq \Delta(m)$ there is a spherically symmetric, real analytic monopole $(\tilde{A}^{\delta}_m, \tilde{\Phi}^{\delta}_m)$ for $g_{\delta}$ in $B_1(0)$.
\end{proposition}

This is basically proposition \ref{existenceprop} in \ref{Analysis}. Then given a monopole $(\tilde{A}^{\delta}_m, \tilde{\Phi}^{\delta}_m)$ for $g_{\delta}$ in $B_1(0)$, proposition \ref{scaling} gives that $(A^{\delta}_m, \Phi^{\delta}_m) = s_{\delta^{-1}}^*(\tilde{A}^{\delta}_m, \delta^{-1} \tilde{\Phi}^{\delta}_m)$ is a monopole for $g$ on $B_{\delta}(0)$. A first step towards the proof of the first item in theorem \ref{teofinal} is achieved by applying the ODE analysis in section \ref{Faraway} to the solutions constructed on $B_{\delta}(0)$ which provide initial conditions for the ODE's at $r=\delta$. This analysis gives,

\begin{proposition}\label{prop:AllMonopoles}
There is a one parameter family of spherically symmetric monopoles on $(\mathbb{R}^3,g)$. Moreover, these can all be obtained by extending the monopoles $(A^{\delta}_m, \Phi^{\delta}_m)$ on $(B_{\delta}(0), g)$ for $(m, \delta)$ such that $m \in \mathbb{R}^+$ and $0 < \delta \leq \Delta(m)$.
\end{proposition}
\begin{proof}
In Lemma \ref{power} a Taylor expansion for solutions of the ODE is obtained. It gives a recursive formula which depends only on $1$ parameter $\dot{\phi}(0)$. The lemma does not address the question of convergence and there are basically $3$ different possibilities.
\begin{enumerate}
\item Case $\dot{\phi}(0)=0$, is the easiest one. In this case there is indeed a unique solution given by $a=1$ and $\phi=0$ and recovers back the flat connection. In terms of the notation in lemma \ref{power} note that this corresponds to $v=0$.

\item Case $\dot{\phi}(0) >0$, for which there are no solutions, as proved in section \ref{Faraway}, corollary \ref{NoSolutions}. We remark that in that corollary $v= \log(a^2)$, as defined in the beginning of section \ref{Faraway}.

\item Case $\dot{\phi}(0)<0$, this is the case for which the PDE analysis shows existence of solutions. If one can find in the $2$ parameter family constructed by the analysis a solution for each value of $\dot{\phi}(0) < 0$. Then, lemma \ref{power} gives uniqueness of solutions for each value of $\dot{\phi}(0) < 0$ and makes of this a genuine global coordinate for $\mathcal{M}_{inv}$.
\end{enumerate}

To proceed one needs to show that the formal $1$-parameter family of solutions given by the power series from lemma \ref{power} does correspond to actual genuine real analytic solutions. Instead of proving that there is a positive radius of convergence, we shall use the PDE analysis from section \ref{Analysis} to construct spherically symmetric and real analytic monopoles on a small ball around the origin and so settle down the question of existence of solutions. The reason why these monopoles are spherically symmetric and real analytic is that the PDE analysis that we develop is totally spherically symmetric and real analytic. Namely, it amounts to solve a Dirichelet problem with a spherically symmetric boundary condition and the PDE has spherically symmetric and real analytic coefficients (as both, the metric, and the BPS monopole used as an approximate solution are spherically symmetric and real analytic).\\
The fact that the series expansion only depends on $1$ parameter comes into play in order to analyze the uniqueness of these monopoles. Since the power series only depends on $-\dot{\phi}(0) \in \mathbb{R}^+$, these solutions are unique for each value of $-\dot{\phi}(0) \in \mathbb{R}^+$ for which a solution can be produced.\\
The next thing to show is that all possible solutions can be reached from our PDE construction. More precisely, we need to show that the PDE construction of the solutions $(A^{\delta}_m, \Phi^{\delta}_m)$ for $(m, \delta)$ with $m \in \mathbb{R}^+$ and $0 < \delta \leq \Delta(m)$ produces monopoles with all negative values of $\dot{\phi}(0)$. This is the reason why one uses two parameters in the construction of monopoles, i.e. with the two parameters $(m,\delta)$ it is easier to tune the properties of the monopoles constructed so that all negative values of $-\dot{\phi}(0)$ are achieved. Estimate \ref{epsilonbounds} in lemma \ref{BoundsLemma} gives bounds on $\dot{\phi} \in \left[ I(m, \delta) , J(m, \delta) \right]$. Then, lemma \ref{seqs} gives two sequences of $(m_n,\delta_n)$. The first makes the lower bound $I_n=I(m_n,\delta_n)$ get as close to zero as one wants, while the second one makes the upper bound $J_n=J(m_n,\delta_n)$ get as close to $-\infty$ as one wants. The fact that all intermediate values are obtained follows from continuity.\\
The final thing one needs to do is to show that the monopoles $\left(A^{\delta}_m, \Phi^{\delta}_m \right)$ on $B_{\delta}(0)$ do extend to the whole of $\mathbb{R}^3$. To do this notice that the values of the monopole $\left(A^{\delta}_m, \Phi^{\delta}_m \right)$ at the point $r= \delta$ can be used as initial conditions for the ODE's \ref{monopoleode} and \ref{monopoleode2}. This is due to the fact that the only singular point of these ODE's happens at $r=0$. Away from it we can apply the standard existence and uniqueness theorem for ODE's to produce a solution in a neighborhood of $r= \delta$. If we can show that this solution does not explode at any finite value of $r$, then it extends up to $r=\infty$.\\
The fact that the solution does not explode at any finite value of $r$ follows from the ODE analysis of section \ref{Faraway}. Namely, from lemma \ref{maxmin} we know that $v \leq 0$. This and the ODE comparison in lemma \ref{vbounds} are then used to prove proposition \ref{vbounds2}, which gives bounds for the function $v=2 \log(a)$ for all values of $r \geq \delta$. Rigorously, there are negative functions $v_u$ and $v_d$, such that $v_d \leq v \leq v_u$ and we explicitly know the function $v_d$ in the lower bound. In particular, since $v_d$ is bounded in any finite interval of the form $[\delta, R]$, for $R > \delta >0$, we conclude that $v$ does not explode at any finite $r$. Hence, the monopoles $\left(A^{\delta}_m, \Phi^{\delta}_m \right)$ do extend over the whole of $\mathbb{R}^3$.
\end{proof}

\begin{proposition}\label{convergencetoBPS}
Let  $R >0$, then there is a sequence $\delta$ converging to zero, such that $s_{\frac{\delta}{R}}^* \left(A^{\delta}_R, \frac{\delta}{R} \Phi^{\delta}_R \right)$ converges uniformly with all derivatives to $(A^{BPS}, \Phi^{BPS})$ on the Euclidean ball $B_R(0)$.
\end{proposition}
\begin{proof}
One needs to prove that for all $\epsilon >0$, there is $\delta$, such that
$$\Vert s_{\frac{\delta}{R}}^* \left(A^{\delta}_R, \frac{\delta}{R} \Phi^{\delta}_R \right) - (A^{BPS}, \Phi^{BPS}) \Vert_{C^{\infty}(B_R)} \leq \epsilon.$$
In a first step one can consider $s_{\delta}^* \left(A^{\delta}_R, \delta \Phi^{\delta}_R \right)=  \left(\tilde{A}^{\delta}_R, \tilde{ \Phi}^{\delta}_R \right)$, then the estimate in proposition \ref{existenceprop} gives that for all $\epsilon >0$, there is $\Delta(R, \epsilon)$, such that for $\delta \leq \Delta( R, \epsilon)$
$$\Vert s_{\delta}^* \left(A^{\delta}_R, \delta \Phi^{\delta}_R \right)  - (A^{BPS}_R, \Phi^{BPS}_R) \Vert_{C^{\infty}(B_1)} \leq \epsilon,$$
for the norm induced by the Euclidean metric. Since the Euclidean metric is invariant by scaling and $ (A^{BPS}_R, \Phi^{BPS}_R) = s_{R}^*(A^{BPS}, R \Phi^{BPS})$ one can scale everything by $R^{-1}$ and obtain the desired result for $\delta = \Delta( R, \epsilon)$.
\end{proof}

The next proposition will finish the proof of both the first and second items in theorem \ref{teofinal}. The first item will be immediate from the statement and for the second item one needs to combine the statement with the previous proposition \ref{convergencetoBPS}, in order to match those monopoles with the large mass limit.

\begin{proposition}\label{MassBijection}
For all monopoles in $\mathcal{M}_{inv}$, the mass is well defined and gives a bijection
$$m: \mathcal{M}_{inv} \rightarrow \mathbb{R}^+.$$
Moreover, given $R>0$ fixed and $\delta \rightarrow 0$, the sequence of monopoles $\left(A^{\delta}_R, \Phi^{\delta}_R \right)$ has mass $m(\delta) \rightarrow +\infty$.
\end{proposition}
\begin{proof}
One already knows that $\mathcal{M}_{inv} \cong \mathbb{R}^+$ corresponding to each value of $-\dot{\phi}(0)$ and this can be used to topologise $\mathcal{M}_{inv}$ as a $1$ dimensional manifold. The next step one needs to take care is in showing that the map $m$ is surjective. From proposition \ref{ola1} and its corollary \ref{ola} one knows that for all $0< \epsilon < \epsilon_0$, $m >0$ and $\delta \leq \Delta(m , \epsilon)$ there are bounds $m(A^{\delta}_m, \Phi^{\delta}_m)  \in \left[ \Phi_- (m, \epsilon) , \Phi_+ (m, \epsilon) \right] $, given by 
\begin{eqnarray}\nonumber
\Phi_-(m, \epsilon) =  \frac{1}{\delta} \left( m \coth(m) -1 - 2 \epsilon \right) \ , \ \Phi_+(m, \epsilon) = \frac{1}{\delta} \left( m\coth(m) + 2 \epsilon + G(\delta)  \right)  + 2G(\delta) .
\end{eqnarray}
Take both $m, \epsilon$ converging to zero in the same way as in the first sequence in lemma \ref{seqs} with $\epsilon_n = m_n^{a}$, with $a < 1$. Then, as done in the same lemma, one can check that
$$ \lim_{m_n \rightarrow 0} \vert \Phi_+(m_n, \epsilon_n) \vert = 0.$$
The other extreme can be made using the second sequence in lemma \ref{seqs}, this keeps $m$ fixed but sends $\epsilon \rightarrow 0$, moreover the choice of $\delta_n \leq \Delta(m, \epsilon_n)$ is such that $\frac{\epsilon_n}{\delta_n}$ still converges to $0$. Then 
$$\lim_{\epsilon_n \rightarrow 0} \vert \Phi_-(m, \epsilon_n) \vert = + \infty,$$
which gives the surjectivity of the mass onto the positive real line. This second sequence also establishes that the mass of the monopoles $\left(A^{\delta}_R, \Phi^{\delta}_R \right)$ diverges. Just take $m=R$ fixed and $\delta$ converging to zero as it was just done. The last step is to show that the derivative of the map $m$ is everywhere injective. As $m$ is a map between $1$ dimensional manifolds, this together with the surjectivity proved above imply the mass is actually a diffeomorphism. Let $(A, \Phi) \in \mathcal{M}_{inv}$, then any $v \in T_{(A, \Phi)}\mathcal{M}_{inv} \subset \Omega^1 \oplus \Omega^0(\mathbb{R}^3, \mathfrak{su}(2))$ is represented by two functions $(b, \psi)$ of $r$ solving the linearized monopole ODE's. This mean that $b(0)=\psi(0)=0$ and they solve $\dot{\psi} = \frac{ab}{h^2}$, $\dot{a} = 2\phi b + 2a \psi$. Differentiating the first of these equations and using the second to substitute for $b$ gives a second order ODE for $\psi$
\begin{equation}
\ddot{\psi} + \left( 2 \partial_r \left( \log(h) \right) - 4 \phi \right)\dot{\psi} - 2a^2 \psi =0.
\end{equation}
Solutions to this satisfy a maximum principle
\begin{itemize}
\item If $\psi$ has a maximum at $M$, then $\ddot{\psi}(M) \leq 0$ and $\dot{\psi}(M) =0$ and so $\psi(M) \leq 0$,
\item If $\psi$ has a minimum at $m$, then $\ddot{\psi}(m) \geq 0$ and $\dot{\psi}(m) =0$ and so $\psi(m) \geq 0$.
\end{itemize}
The derivative of the mass is
$$dm(v) = 2\psi(\infty) : \mathbb{R} \rightarrow \mathbb{R}.$$
If $v=(b, \psi)$ is in the kernel of $dm$, then $\psi(\infty)=0$. The argument using these maximum principles is as follows. If $\psi(0)=0$, one concludes that $\psi$ must have a positive maximum or a negative minimum. Both of these hypothesis are impossible due to the maximum principle unless if $\psi=0$ and hence also $b=0$, i.e. $v=0$.
\end{proof}

The last item which remains to be shown is that in the large mass limit after bubbling a BPS monopole at $0$, one is left with a $g$-Dirac monopole on the exterior.

\begin{proposition}\label{bubbling}
 Let $\lbrace (A_{\lambda}, \Phi_{\lambda}) \rbrace_{\lambda \in [\Lambda, +\infty)}$ be a sequence of monopoles with mass $\lambda \rightarrow \infty$. Then the translated monopole sequence
$$\left(A_{\lambda}, \Phi_{\lambda}- \lambda  \frac{\Phi_{\lambda}}{\vert \Phi_{\lambda} \vert}\right),$$
converges uniformly with all derivatives to a monopole on $\mathbb{R}^3 \backslash \lbrace 0 $ reducible to the $g$-Dirac monopole $(A^D, \Phi^D = G )$ with mass $0$ for the metric $g$.
\end{proposition}
\begin{proof}
Working in a fixed gauge to this amounts to prove that given $R >0$ and $\epsilon>0$, there is a $\lambda$ such that $\Vert (\Phi_{\lambda}- \lambda \frac{\Phi_{\lambda}}{\vert \Phi_{\lambda} \vert} )  - 2 G \frac{\Phi_{\lambda}}{\vert \Phi_{\lambda} \vert} \Vert_{C^{\infty}[R, +\infty)} \leq \epsilon$. For this one needs to study the function
$$u= \left( -\frac{\lambda}{2} + G \right) - \phi_{\lambda} ,$$
where $\phi_{\lambda}$ is the scalar such that $\Phi_{\lambda}= 2\phi_{\lambda}  \frac{\Phi_{\lambda}}{\vert \Phi_{\lambda} \vert}$. Then $\dot{u}=  -\frac{1}{2h^2} - \frac{1}{2h^2} (a_{\lambda}^2-1)= -\frac{a_{\lambda}^2}{2h^2}$, which shows that $\dot{u}<0$. This, together with $\lim_{r \rightarrow \infty} u = 0$ can be integrated to give
$$u(r) \leq G(r)   \sup_{t \in [R, + \infty)}  a_{\lambda}^2(t) . $$
Moreover, $G$ is bounded in $[R, + \infty)$ and $a_{\lambda}^2$ is decreasing, so that $a_{\lambda}^2 \leq a_{\lambda}^2(R)$. Now it is time to pick $\delta,m$ such that $a_{\lambda}= a_m^{\delta}$. This may be done with $\delta(\lambda), m(\lambda)$ as in the second sequence in lemma \ref{seqs}, but such that  such that $m(\lambda)$ also converges to $\infty$ (see the proof of lemma \ref{seqs}). Then $\delta(\lambda)$ converges to $0$ and as $a_{\lambda}^2$ is decreasing $a_{\lambda}^2 \leq a_{\lambda}^2(\delta) \sim m e^{-m}$ by the estimates in lemma \ref{BoundsLemma}, which converges to $0$.
\end{proof}

\subsection{The $SU(2)$ Invariant Bogomolny Equations}\label{SectionInv}

As $\mathbb{R}^3 \backslash 0 \cong \mathbb{R}_+ \times \mathbb{S}^2$, one pulls back the homogeneous bundle
$$P_k = SU(2) \times_{\lambda_k} SU(2),$$
from $\mathbb{S}^2 \cong SU(2)/U(1)$. Where $\lambda_k: U(1) \rightarrow SU(2)$ is the isotropy homomorphism given by taking $\lambda_k(e^{i \alpha}) =diag (e^{ik\alpha}, e^{-ik\alpha})$, for $k \in \mathbb{Z}$. Let $T_1, T_2, T_3$ be a basis of $\mathfrak{su}(2)$, such that $[T_i,T_j]=2\epsilon_{ijk}T_k$, and $\omega_1, \omega_2 , \omega_3$ the dual coframe. Let $\mathfrak{su}(2) = \mathfrak{h} \oplus \mathfrak{m}$, with $\mathfrak{h}=T_1$ and $\mathfrak{m}=\langle T_2 , T_3 \rangle$, this splitting equips the Hopf bundle $SU(2) \rightarrow \mathbb{S}^2$ with an $SU(2)$ invariant connection whose horizontal space is $\mathfrak{m}$. This induces a connection in each $P_k$ known as the canonical invariant connection. It is encoded by the $1$-form $A^c_k =k T_1 \otimes \omega^1 \in \Omega^1(SU(2), \mathfrak{su}(2))$. By Wang's theorem \ref{Wang}, other invariant connections differ from it by morphisms of $U(1)$-representations $( \mathfrak{m}, Ad ) \rightarrow ( \langle T_2 , T_3 \rangle, Ad \circ \lambda_k )$. Invoking Schur's lemma these vanish for all $k \neq \pm 1$, and are isomorphisms for $k = \pm 1$. Suppose $k= 1$, then
$$A= A^c + a(r) (T_2 \otimes \omega^2 + T_3 \otimes \omega^3 ),$$
with $a : \mathbb{R}^+_0 \rightarrow \mathbb{R}$. The curvature of such a connection is given by $F_A = 2( a^2 -1) T_1 \otimes \omega^{23} + \dot{a} \left( T_2 \otimes dr \wedge \omega^2 + T_3 \otimes dr \wedge \omega^3. \right)$. For each $r \in \mathbb{R}^+$ an invariant Higgs field $\Phi(r) \in \Omega^0(\lbrace r \rbrace \times SU(2), \mathfrak{su}(2))$ must be a constant in the trivial component of the $U(1)$ representation $(\mathfrak{su}(2), Ad \circ \lambda )$, i.e. $\Phi = \phi (r) \ T_1$, with $\phi  : \mathbb{R}_+ \rightarrow \mathbb{R}$. Its covariant derivative $\nabla_A \Phi$ with respect to the connection $A$ is $\nabla_A \Phi = \dot{\phi} T_1 \otimes dr + 2a\phi \left(T_2 \otimes \omega^3 - T_3 \otimes \omega^2\right)$. The metric \ref{invariantmetric} on $\mathbb{R}^+ \times \mathbb{S}^2$ can then be written as $g = dr^2 + 4h^2(r) ( \omega_2 \otimes \omega_2+ \omega_3 \otimes \omega_3)$ and is invariant under the $SU(2)$ action, i.e. spherically symmetric. The Bogomolny equation $\ast \nabla_A \Phi = F_A$ turns into the ODE's \ref{monopoleode} and \ref{monopoleode2} and explicit solutions to these are known in two different cases.\\

First, and most important here is the Euclidean case $h(r)=r$. Some special solutions are the flat connection $\vert a \vert = 1$ and $\phi=0$ and the Dirac monopole with $a=0$ or $\phi= m - \frac{1}{2r}$, for $m \in \mathbb{R}$. For $a \neq 0$, the general solution to the ODE's is
\begin{eqnarray}
\phi_{C,D}^{BPS}= \frac{1}{2} \left( \frac{1}{r}  - \frac{C}{\tanh(Cr+D)} \right) \ , \ a_{C,D}^{BPS} =  \frac{Cr}{\sinh(Cr+D)}.
\end{eqnarray}
The solutions with $D=0$ and $C=m < \infty$ are the only ones that extend over the origin, giving rise to irreducible monopoles on $\mathbb{R}^3$. These are the so called BPS monopole $(a^{BPS}_m, \phi^{BPS}_m)$ and first appeared in \cite{BPS}. For small $r$
\begin{eqnarray}\nonumber
\phi^{BPS}_m (r) = - \frac{m^2 r}{6} + \frac{m^4 r^3}{90} + ...\  , \ a^{BPS}_m (r) = 1- \frac{m^2r^2}{6} + \frac{7m^4 r^4}{360} - ... \nonumber
\end{eqnarray}
while for large $r$
\begin{eqnarray}\nonumber
\phi^{BPS}_m (r) = - \frac{1}{2} \left( m - \frac{1}{r} \right) + O(e^{-mr}) \ , \ a^{BPS}_m (r) = O (2r e^{-mr} ).
\end{eqnarray}

In the hyperbolic case $h(r)= \sinh(r)$ and there is also a one parameter family of monopoles parametrized their mass $m \in \mathbb{R^+}$ and given by
\begin{eqnarray}
\phi_m (r) = \frac{1}{2} \left( \frac{1}{\tanh(r)} - \frac{m+1}{\tanh((m+1)r)} \right)  \ , \ a_m (r) = \frac{(m+1) \sinh(r) }{\sinh((m+1)r)}.
\end{eqnarray}
In both cases the parameter $m$ is the asymptotic value of the Higgs field at $\infty$, i.e. the mass of the monopole.

\subsection{PDE Analysis}\label{Analysis}

The metric $g_{\delta} = \delta^{-2}s_{\delta}^*g$ on its unit ball can be written as
\begin{eqnarray}\nonumber
g_{\delta} = dt^2 + h^2_{\delta}(t)g_{\mathbb{S}^2}
\end{eqnarray}
where $t \in (0, 1)$ is the geodesic coordinate of the new metric (i.e. $\delta t = r \circ s_{\delta}$) and $h_{\delta}^2(t) = t^2 + \delta^2 G_{\delta}(t)$, with $G_{\delta}$ an analytic function such that $\frac{G_{\delta}(t)}{t^4}$ can be bounded independently of $\delta$. This changes the problem of solving the equations in a small $\delta$ ball to that of solving the equations in a unit ball but with a varying metric $g_{\delta}$, which is a spherically symmetric perturbation in $\delta$ from the Euclidean one. So one needs to solve $\ast_{\delta} F_A - \nabla_A \Phi =0$, where $*_{\delta}$ is the $g_{\delta}$-Hodge operator. For each $m \in \mathbb{R}^+$ consider the mass $m$ Euclidean BPS monopoles \cite{BPS}, $(A_m^{BPS},\Phi_m^{BPS})$. Their error term
$$\epsilon_m^{\delta} = \ast_{\delta} F_{A_m^{BPS}} - \nabla_{A_{m}^{BPS}} \Phi_m^{BPS} = O( (\delta m)^2),$$
is small and vanishes for $\delta=0$, where the metric is Euclidean. The idea is to use these as approximate solutions and search for a solution of the form $(A^{\delta}_m , \Phi^{\delta}_m )= (A_m^{BPS}, \Phi_m^{BPS} ) + (b , \psi)$, with $v = (b,\psi)$ a section of $ \left( \Lambda^1 \oplus \Lambda^0 \right) \otimes \mathfrak{su}(2))$. The Bogomolny equation looks like a first order quasilinear PDE and
\begin{equation}\label{firstorderPDE}
P(v)= \epsilon^m_\delta + d_2(v) + Q(v,v)=0,
\end{equation}
where $d_2(a, \phi) = \ast d_{A_m^{BPS}} a - \nabla_{A_m^{BPS}} \phi - [a, \Phi_m^{BPS} ]$ is linearization of the Bogomolnyi equation and $Q(v,v) = \ast [ b \wedge b ] - [b , \psi ] $ is a quadratic $0$ order term. We shall search for a solution of the form $v=d_2^* u$, then the new problem is to solve $P(d_2^* u)=0$, and a first step to do this is to find an inverse for $d_2 d_2^*$. This can be achieved by further requiring a boundary condition giving rise to an elliptic problem,
\begin{eqnarray}\label{PDE}
\epsilon^{\delta}_C + d_2 d_2^* (u) + Q(d_2^*u , d_2^* u) & = & 0, \\
u \vert_{\partial B_1(0)} & = & 0.
\end{eqnarray}
The claim is that the Dirichlet boundary allows inverting $d_2 d_2^*$. This follows from a Weitzenb\"ock formula, which at $\delta=0$ is
$$d_2 d_2^* u=\nabla_{A_m^{BPS}}^* \nabla_{A_m^{BPS}} u - \left[  [u , \Phi^{BPS}_m ] \Phi^{BPS}_m \right],$$
acting on $\mathfrak{su}(2)$ valued $1$-forms. Then $d_2d_2^*$ at $\delta=0$, together with the boundary condition $u\vert_{\partial B_1}=0$ is an elliptic, positive and self adjoint operator. As it is self adjoint it has index $0$ and the boundary condition and positivity show it has zero kernel. So at $\delta=0$, the unique solution is $u=0$ and the linearisation of $P(d_2^*u)$ is $d_2 d_2^*$ which has a bounded inverse
$$L: C^{k, \alpha} \rightarrow C^{k+2, \alpha}.$$
The Implicit Function Theorem applies and for each $m \in \mathbb{R}^+$ there is $\Delta(m)$, such that for all $\delta < \Delta(m)$, there is a small solution $u_m^{\delta}$ of \ref{PDE}. Since $\epsilon^{\delta}_m$ and the metric are analytic, elliptic regularity guarantees that $u_m^{\delta}$ is itself analytic, see sections $5.8$ and $6.7$ of \cite{Morrey08}. This result can be improved to come together with useful estimates which are stated in the following

\begin{proposition}\label{existenceprop}
Let $m >0$, then for all positive $\epsilon$, there is $\Delta(m, \epsilon)>0$, such that for $\delta \leq \Delta(m, \epsilon)$, the solution $u_{m}^{\delta}$ is the unique one satisfying
\begin{equation}\label{estimate}
\Vert d_2^* u_{m}^{\delta} \Vert_{C^{\infty}} \leq \epsilon.
\end{equation}
Moreover,  $u_{m}^{\delta}$ is real analytic and for a bound in the $C^1$ norm it is sufficient to take $\Delta(m,\epsilon) =  \frac{1}{m} \min \Big\lbrace \sqrt{\frac{\epsilon}{\Vert d_2^* \Vert \Vert L \Vert }}\, \frac{1}{ \Vert d_2^* \Vert \Vert L \Vert } \Big\rbrace$, where $\Vert d_2^* \Vert ,\Vert L \Vert $ denote the norms of the operators $d_2^*: C^{1, \alpha} \rightarrow C^{0,\alpha}$ and $L: C^{0,\alpha} \rightarrow C^{2,\alpha}$.
\end{proposition}

To prove proposition \ref{existenceprop} one uses an alternative formulation to the Implicit Function Theorem via interpreting \ref{PDE} as a fixed point equation and making use of the following lemma. It is proved by using the contraction mapping principle and keeping track of the norms in the iterations converging to the solution, see lemma $7.2.23$ in \cite{Donaldson1990}.

\begin{lemma}\label{quadratic}
Let $B$ be a Banach space and $q : B \rightarrow B$ a smooth map such that for all $u,v \in B$
$$\Vert q(u)-q(v) \Vert \leq k \left( \Vert u \Vert + \Vert v \Vert \right)  \Vert u - v \Vert,$$
for some fixed constant $k$ (i.e. independent of $u$ and $v$). Then, if $ \Vert v \Vert \leq \frac{1}{10k}$ there is a unique solution $u$ to the equation
\begin{equation}\label{quadraticeq}
u + q(u) =v,
\end{equation}
which satisfies the bound $\Vert u \Vert \leq 2 \Vert v \Vert$. 
\end{lemma}

This is applied to prove proposition \ref{existenceprop} as follows. Let $B$ be the space of $C^{2, \alpha}$ sections of $\Lambda^1 B_1(0)$ vanishing at the boundary and apply $L$ to the left of $P(d_2^*u)=0$, this equation is now in the form of \ref{quadraticeq}
$$u + LQ(d_2^*u,d_2^*u) = - L\epsilon^{\delta}_m,$$
and $q(u)=LQ(d_2^*u,d_2^*u)$ does satisfy the hypothesis of lemma \ref{quadratic} as shown below
\begin{eqnarray}\nonumber
\Vert q(u) - q(v) \Vert_{C^{2,\alpha}} & = &  \Vert  LQ(d_2^* u , d_2^* u) -  LQ(d_2^* v , d_2^* v) \Vert_{C^{2,\alpha}} =\Vert  LQ(d_2^* (u+ v) , d_2^* (u-v) \Vert_{C^{2,\alpha}} \\ \nonumber
& \leq & cst. \Vert L \Vert  \Vert d_2^* (u+ v) \Vert_{C^{0,\alpha}} \Vert d_2^* (u-v) \Vert_{C^{0,\alpha}} \\ \nonumber
& \leq & cst. \Vert L \Vert \Vert d_2^* \Vert^2 \left( \Vert u \Vert_{C^{2,\alpha}} + \Vert v \Vert_{C^{2,\alpha}} \right) \Vert u-v \Vert_{C^{2,\alpha}}
\end{eqnarray}
So that $k = cst. \Vert L \Vert \Vert d_2^* \Vert^2$. Then, the lemma applies for $\Vert L \epsilon^{\delta}_m \Vert_{C^{2, \alpha}} \leq cst. k^{-1}$,  since $ \Vert L \epsilon^{\delta}_m \Vert_{C^{2, \alpha}} \leq \Vert L \Vert  \Vert \epsilon^{\delta}_m \Vert_{C^{0, \alpha}}$ it is enough to guarantee that 
\begin{equation}\label{esti}
\Vert \epsilon^{\delta}_m \Vert_{C^0} \leq cst. (\Vert L \Vert \Vert d_2^* \Vert)^{-2},
\end{equation}
and in this case there is a unique solution $u_m^{\delta}$ satisfying the estimate $\Vert u_m^{\delta} \Vert_{C^{2, \alpha}} \leq cst. \Vert L\epsilon^{\delta}_m \Vert_{C^{2, \alpha}}$. Proposition \ref{existenceprop} is proven by showing that given $\epsilon >0$ it is possible to make $\Vert d_2^* u_m^{\delta} \Vert_{C^1} \leq \epsilon$. Since
$$\Vert d_2^* u_m^{\delta} \Vert_{C^{1, \alpha}} \leq \Vert d_2^* \Vert \Vert u_m^{\delta} \Vert_{C^{2, \alpha}} \leq cst. \Vert d_2^* \Vert \Vert L \Vert \Vert \epsilon^{\delta}_m \Vert_{C^{0, \alpha}}, $$
it is enough to make $\delta \leq \Delta(m, \epsilon)$ small enough so that $\Vert \epsilon^{\delta}_m \Vert_{C^{0, \alpha}} \leq \epsilon \Vert d_2^* \Vert^{-1} \Vert L \Vert^{-1}$. Having in mind that one still needs to guarantee the estimate \ref{esti} holds, one concludes that $\Vert \epsilon^{\delta}_m \Vert_{C^{0, \alpha}}$ needs to be small enough so that
\begin{equation}\label{esti2}
\Vert \epsilon^{\delta}_m \Vert_{C^{0, \alpha}} \leq cst. \min \lbrace \Vert L \Vert^{-1} \Vert d_2^* \Vert^{-1}\epsilon , \Vert d_2^* \Vert^{-2} \Vert L \Vert^{-2} \rbrace.
\end{equation}

\begin{lemma}
The estimate $\Vert \epsilon^{\delta}_m \Vert_{C^{2, \alpha}} \leq cst. m^2 \delta^2$ holds.
\end{lemma}
\begin{proof}
For $\delta \neq 0$, the error term does not vanish and is given by
\begin{equation}\label{epsilon_0}
\epsilon_0 = \ast F_{A_0^m} - \nabla_{A_0^m} \Phi_0^C = \frac{a_m^2 -1}{2t^2} \left(  \frac{t^2}{h_{\delta}^2} -1 \right) T_1 \otimes dt.
\end{equation}
Moreover, the point-wise norm of the above quantity is
\begin{eqnarray} \nonumber
\vert \epsilon_0 \vert & \leq &  \frac{1- a_m^2(t) }{2t^2} \delta^2  \frac{\vert G_{\delta}(t) \vert }{t^2}  + o(\delta^4) \leq  \delta^2 \sup_{t \in [0,1]} \left( \vert \dot{\phi_m} \vert \frac{\vert G_{\delta}(t) \vert }{t^2} \right).
\end{eqnarray}
Since as remarked at the beginning of this subsection $\frac{\vert G(t) \vert }{t^4}$ can be bounded independently of $\delta$ on can just use the explicit formula for $\phi_m$ and compute $\sup_{t \in [0,1]} \vert \dot{\phi_m} \vert = \frac{m^2}{6}$. Differentiating once more we can prove a similar bound for $\dot{\epsilon_0}$ and so the bound in the statement holds for the $C^{0, \alpha}$ norm, for all $\alpha<1$.
\end{proof}

Putting this together with equation \ref{esti2} finally gives that is is enough to set $\delta \leq \Delta(m, \epsilon)$, with
\begin{equation}\label{deltamepsilon}
\Delta (m,\epsilon) =  \frac{1}{m} \min \Big\lbrace \sqrt{\frac{\epsilon}{\Vert d_2^* \Vert \Vert L \Vert }} ,  \frac{1}{ \Vert d_2^* \Vert \Vert L \Vert } \Big\rbrace,
\end{equation}
in order to obtain
$$\Vert d_2^* u_m^{\delta} \Vert_{C^{1, \alpha}} \leq \epsilon.$$
Improving this to a $C^{\infty}$ bound can be made by standard bootstrapping arguments in elliptic PDE theory. Notice that all the coefficients of the PDE are real analytic as the BPS monopole is real analytic and so is the metric by assumption. Then it follows by the regularity theory for elliptic PDE's, sections $5.8$ and $6.7$ of \cite{Morrey08}, that the solution $u_m^{\delta}$ is real analytic. This finishes the proof of proposition \ref{existenceprop}.\\

The solution to the monopole equations on $B_{1}(0)$ for the metric $g_{\delta}$ obtained is $\left(  A_{m}^{BPS} , \Phi_{m}^{BPS} \right) + d_2^* u^{\delta}_m$. Denote by $(d_2^* u^{\delta}_m)_i$ the component of $d_2^* u^{\delta}_C$ in $\Lambda^i$. Then proposition \ref{scaling} gives the monopole on $B_{\delta}(0)$ for the metric $g$, given by
\begin{eqnarray}\label{solafterscaling}\nonumber
\left( A^{\delta}_{m} , \Phi^{\delta}_{m} \right) & = & \left( s_{\delta^{-1}}^* ( A_{m}^{BPS} + (d_2^* u^{\delta}_m)_1 ) , \delta^{-1}s_{\delta^{-1}}^* (\Phi_{m}^{BPS} +  (d_2^* u^{\delta}_m)_0 ) \right) \\ 
& = &  \left(  A_{\delta^{-1}m}^{BPS} + s_{\delta^{-1}}^*(d_2^* u^{\delta}_m)_1  , \Phi_{\delta^{-1}m}^{BPS} +   \delta^{-1} s_{\delta^{-1}}^* (d_2^* u^{\delta}_m)_0 ) \right) 
\end{eqnarray}
Rescaling the estimate \ref{estimate} gives

\begin{lemma}\label{estimatesclosetozero}
Let $m$ and $\epsilon$ be positive, then for $\delta \leq \Delta(m, \epsilon)$, the monopole $\left( A^{\delta}_{m} , \Phi^{\delta}_{m} \right)$ for $g$ in $B_{\delta}$ is such that
\begin{eqnarray}
\Vert A^{\delta}_m - A_{\delta^{-1}m}^{BPS} \Vert_{C^{\infty}(B_{\delta})} +
\Vert   \Phi^{\delta}_m - \Phi_{\delta^{-1}m}^{BPS}\Vert_{C^{\infty}(B_{\delta})} & \leq & \delta^{-1} \epsilon,
\end{eqnarray}
where the norms are measured in the metric $g$. In particular, there is $\epsilon_0(m)= \frac{1}{\Vert d_2^* \Vert \Vert L \Vert} >0$, such that for all $\epsilon \leq \epsilon_0(m)$ and $\delta = \Delta(m, \epsilon)$
\begin{eqnarray}
\Vert A^{\delta}_m - A_{\delta^{-1}m}^{BPS} \Vert_{C^{\infty}(B_{\delta})} +
\Vert   \Phi^{\delta}_m - \Phi_{\delta^{-1}m}^{BPS}\Vert_{C^{\infty}(B_{\delta})}  & \leq & m \sqrt{ \frac{\epsilon}{\epsilon_0 }},
\end{eqnarray}
and once again the norms are measured using the metric $g$.
\end{lemma}
\begin{proof}
Denote by $(B_{r},g)$ the radius $r$ ball centred at zero where the distance $r$ is measured with respect to the metric $g$. Then, as sets $(B_{\delta}, g)=(B_{1}, g_{\delta})$, moreover the norm of a $1$-form $\omega$ gets scaled according to $\vert \omega \vert_{g} = \delta^{-1} \vert \omega \vert_{\delta^{-2}g}$
\begin{eqnarray}\nonumber
\Vert A^{\delta}_m - A_{\delta^{-1}m}^{BPS} \Vert_{C^{\infty}(B_{\delta},g)}= \delta^{-1} \Vert s_{\delta^{-1}}^*(d_2^* u^{\delta}_m)_1 \Vert_{C^{\infty}(B_1, s_{\delta^{-1}}g_{\delta})} \leq \delta^{-1} \Vert (d_2^* u^{\delta}_m)_1 \Vert_{C^{\infty}(B_{1}, g_{\delta})} \leq  \delta^{-1} \epsilon.
\end{eqnarray}
In the same way for $\Phi_m^{\delta}$ one computes
\begin{eqnarray}\nonumber
\Vert  \Phi_m^{\delta} - \Phi_{\delta^{-1}m}^{BPS}\Vert_{C^{\infty}(B_{\delta}, g)} =& \delta^{-1} \Vert (d_2^* u^{\delta}_m)_0 \nonumber \Vert_{C^{\infty}(B_{\delta}, g)} \leq  \delta^{-1}  \Vert (d_2^*u^{\delta}_m)_0 \Vert_{C^{\infty}(B_{1}, g_{\delta})}\leq  \delta^{-1}  \epsilon.
\end{eqnarray}
the second statement follows directly from inserting the formula \ref{deltamepsilon} and $\epsilon_0$ is determined by $\epsilon_0(m)= \frac{1}{\Vert d_2^* \Vert \Vert L \Vert}$ in order to make the first term in \ref{deltamepsilon} smaller than the second.
\end{proof}

\subsection{ODE Analysis}\label{Faraway}

Recall the monopole ODE's \ref{monopoleode} and \ref{monopoleode2} and define $v=2\log(a)$ (note that this implies $\dot{v} = 4\phi$) and write the equations \ref{monopoleode2} as a second order ODE for $v$
\begin{eqnarray}\label{second}
\ddot{v} = \frac{2}{h^2} \left( e^v -1 \right).
\end{eqnarray}
The first result in this section gives conditions on the existence of a formal power series solution to equation \ref{second}. Before the statement, recall that one is interested in solutions of \ref{monopoleode2} satisfying $a(0)=1, \phi(0)=0$ and $\lim_{r \rightarrow \infty} r^{-k}a(r)=0$, for some $k \in \mathbb{Z}$. Translated into $v$, these are the conditions that $v(0)=\dot{v}(0)=0$ and $\lim_{r \rightarrow \infty} r^{-k} e^{v(r)}=0$, for some $k \in \mathbb{Z}$.

\begin{lemma}\label{power}
Let $h$ be analytic and $b \in \mathbb{R}$. Write $h^2(r)= r^2 \varphi(r) $ with $\varphi(r)$ analytic such that its expansion can be written as $\varphi(r)= \sum_{i \geq 0} \varphi_i r^i$, with $\varphi_0 =1$. Then, there is a unique formal power series solution $v= \sum_{i \geq 0} v_i r^i$ to the equation \ref{second} such that $v(0)=\dot{v}(0)=0$ and $\ddot{v}(0)=b \in \mathbb{R}$. It is determined by $v_0=v_1=0$, $v_2=b$ and
\begin{equation}\label{eq:RecurrenceRelation}
v_{i+2} = \frac{2}{i(i+3)}   \left(   \left( \sum_{k \geq 2}  \frac{1}{k!} \sum_{l_1 + ... + l_k = i+2} v_{l_1} ... v_{l_k} \right)  + \sum_{j < i} \varphi_{i-j} \left( \sum_{k \geq 1}  \frac{1}{k !} \sum_{l_1 + ... + l_k = j+2} v_{l_1} ... v_{l_k} \right)  \right) ,
\end{equation}
for all $i +2 \geq 3$.
\end{lemma}
\begin{proof}
Substituting into the equation shows that the recurrence relation formally satisfies equation \ref{second}. It remains to check that the recurrence relation is completely determined by setting $v_0=v_1=0$ and $v_2=b \in \mathbb{R}$. This, as well, can be directly checked from equation \ref{eq:RecurrenceRelation}. To do this notice that the first term 
$$\sum_{k \geq 2}  \frac{1}{k!} \sum_{l_1 + ... + l_k = i+2} v_{l_1} ... v_{l_k},$$
contains no terms in $v_{i+2}$, since $k \geq 2$ and so one must have at least two $v_l$'s. Since $v_0=0$, each $l \geq 1$, which is the same as saying that each $l \leq i+1$. As for the second term
$$\sum_{j < i} \varphi_{i-j} \left( \sum_{k \geq 1}  \frac{1}{k !} \sum_{l_1 + ... + l_k = j+2} v_{l_1} ... v_{l_k} \right),$$
it just contains terms in $j+2 < i+2$.
\end{proof}

The monopoles from the last section give a family of solutions $(A_m^{\delta}, \Phi_m^{\delta})$ on $r \leq \delta$ depending on two parameters $m \in \mathbb{R}^+$ and $\delta \leq \Delta(m)$. These can be used to give initial conditions for the ODE's at $r=\delta$. The estimates from lemma \ref{estimatesclosetozero}, can be used to obtain estimates to these initial conditions as follows.

\begin{lemma}\label{BoundsLemma}
Let $m \in \mathbb{R}^+$ and $\epsilon >0$, then for all $\delta \leq \Delta(m , \epsilon)$ the monopole $(A_m^{\delta}, \Phi_m^{\delta})$ constructed in the previous section has its fields satisfying
\begin{eqnarray}
\vert \phi^{\delta}_m (\delta) - \frac{1}{2 \delta} \left( 1- m \coth( m ) \right) \vert & \leq & \delta^{-1} \epsilon ,
\end{eqnarray}
and $\vert a^{\delta}_m(\delta) - \frac{m}{\sinh(m)} \vert \leq \delta^{-1} \epsilon$. Moreover, the following estimate also holds
\begin{eqnarray} \label{epsilonbounds}
\dot{\phi}^{\delta}_m (0) & \in & \left[ I(m, \delta), J(m, \delta) \right],
\end{eqnarray}
with $ I(m, \delta)= - \frac{1}{6}\frac{m^2 }{\delta^{2}} - \epsilon \delta^{-1}$ and $J(m, \delta)=- \frac{1}{2}\frac{m^2}{\delta^{2}}\left( m^{-2} - \sinh^{-2}(m) \right)+ \epsilon \delta^{-1}$.
\end{lemma}
\begin{proof}
The estimates from lemma \ref{estimatesclosetozero} guarantee that
$$ \sup_{r \leq \delta} \left( \vert \phi^{\delta}_m  - \phi_{\delta^{-1}m}^{BPS} \vert + \vert h^{-1} \left( a^{\delta}_m - a_{\delta^{-1}m}^{BPS} \right) \vert \right) \leq \delta^{-1} \epsilon.$$
Using the explicit formulas $\phi_{\delta^{-1}m}^{BPS}(\delta) = \frac{1}{2 \delta} \left( 1- m \coth( m ) \right)$ and $a_{\delta^{-1}m}^{BPS}(\delta)= \frac{m}{\sinh(m)}$, one obtains the desired bounds on the values of the fields at $\delta$. Since lemma \ref{estimatesclosetozero} actually gives $C^1$ estimates one also has $
\sup_{r \leq \delta} \vert \dot{\phi^{\delta}_m}  - \dot{\phi}_{\delta^{-1}m}^{BPS} \vert \leq  \delta^{-1} \epsilon$ and once again the explicit formula for $\dot{\phi}_{\delta^{-1}m}^{BPS}$ gives the result in the statement. In order to obtain the bounds stated one must notice that $\dot{\phi}_{\delta^{-1}m}^{BPS}$ is increasing, so one bounds below by $\dot{\phi}_{\delta^{-1}m}^{BPS}(0)$ and above by $\dot{\phi}_{\delta^{-1}m}^{BPS}(\delta)$.
\end{proof}

The following lemma contains two sequences of values $(m_n,\epsilon_n)$ inducing sequences of values $(m_n, \delta_n)$ 
which can be used to show that the PDE constructed monopoles are actually all monopoles as done in proposition \ref{prop:AllMonopoles} and that there are monopoles with all values of mass $m \in \mathbb{R}^+$ as done in proposition \ref{MassBijection}.

\begin{lemma}\label{seqs}
Let $I, J$ be the quantities provided by the previous lemma, then:
\begin{enumerate}
\item There are sequences $(m_n, \epsilon_n)$ and $\delta_n \leq \Delta(m_n, \epsilon_n)$, such that $ I_n = I(m_n,\delta_n) \rightarrow 0$. Moreover, for this sequence of $(m_n, \epsilon_n)$ and $\delta_n$, the quantity 
$$\Phi_+ (n)=\frac{1}{\delta_n} \left( m_n\coth(m_n) -1 + 2 \epsilon_n  \right) + 2G(\delta_n),$$
also converges to zero.
\item There are other sequences $(m_n, \epsilon_n)$ and $\delta_n \leq \Delta(m_n, \epsilon_n)$, such that $J_n = J(m_n,\delta_n)= \rightarrow - \infty$. For these sequences of $(m_n, \epsilon_n)$ and $\delta_n$, the quantity 
$$\Phi_- (n)=\frac{1}{\delta_n} \left( m_n\coth(m_n) -1 - 2 \epsilon_n  \right),$$
converges to $+ \infty$.
\end{enumerate}
\end{lemma}
\begin{proof}
\begin{enumerate}
\item We shall first fix a sequence $m_n \rightarrow 0$. Then, $m_n \coth (m_n) -1 = O(m_n^2)$ and notice that to prove the statement it is enough to show that one can take the sequences to be such that both $\frac{m_n}{\delta_n}$ and $\frac{\epsilon_n}{\delta_n}$ converge to $0$, while $\delta_n$ can be taken arbitrarily large, so that $G(\delta_n) \rightarrow 0$. To achieve this we shall first take $m_n \rightarrow 0$ as remarked before, and $\epsilon_n = m_n^{a}$, for some positive $a < 1$, then $\epsilon_n \leq \sqrt{\epsilon_n}$ and the formula for $\Delta(m_n , \epsilon_n)$ in proposition \ref{existenceprop} is 
\begin{equation}\label{eq:DeltaLowerBound}
\Delta(m_n ,\epsilon_n) \geq \frac{\epsilon_n}{m_n} \min \Big\lbrace \frac{1}{\sqrt{\Vert (d_2)_n^* \Vert \Vert L_n \Vert }} ,  \frac{1}{ \Vert (d_2)_n^* \Vert \Vert L_n \Vert } \Big\rbrace.
\end{equation}
As $\Vert (d_2)_n^* \Vert \Vert L_n \Vert$ is uniformly bounded above and below for any sequence $m_n \rightarrow 0$, we can take $\delta_n= C \frac{\epsilon_n}{m_n} = C m_n^{a-1}$, for some $C>0$. In this way we do have $\delta_n$ getting arbitrarily large and
\begin{eqnarray}\nonumber
\frac{m_n}{\delta_n} = C^{-1}m_n^{2-a} \ , \  \frac{\epsilon_n}{\delta_n} = C^{-1} m_n,
\end{eqnarray}
which do converge to zero as $m_n$ does.

\item One can take $m_n = m>0$ constant and $\epsilon_n$ to be a sequence converging to zero, in this way the inequality \ref{eq:DeltaLowerBound} still holds and it is enough to set $\delta_n = Cm^{-1} \epsilon_n$, where $C>0$ is constant. By substitution in $J_n$ one obtains $J_n = -k_1 \epsilon_n^{-1} + k_2 \sqrt{\epsilon_n}$,
for some positive real constants $k_1,k_2$ and this converges to $- \infty$ as $\epsilon_n \rightarrow 0$.\\
To check that $\Phi_-(n) \rightarrow + \infty$, notice that by increasing $n$, $\epsilon_n$ can be taken arbitrarily small and so $m\coth(m) -1 -2\epsilon_n$ is greater than a positive constant $C'$. Since $\delta_n =Cm^{-1} \epsilon_n$ is converging to zero we see that
$$\Phi_-(n) \geq \frac{Cm}{C'} \frac{1}{\epsilon_n} \rightarrow + \infty.$$
\end{enumerate}
\end{proof}

\begin{lemma}\label{maxmin}
Let $v$ be a solution of \ref{second}. Suppose $v$ has a minimum at $m$, or a maximum at $M$, then $v(m) \geq 0$ and $v(M) \leq 0$. Moreover, if $v$ satisfies initial conditions $v(\delta) <0$, $\dot{v}(\delta)<0$ (resp. $v(\delta) >0$, $\dot{v}(\delta)>0$), then $v <0$ (resp. $v>0$) in $(\delta, \infty)$.
\end{lemma}
\begin{proof}
Let $m$ be the point at which the minimum is achieved, then $\ddot{v}(m) \geq 0$ and so
$$\frac{2}{h^2} \left( e^v -1 \right) \geq 0 \ \  \implies \ \ v \geq 0.$$
In the same way at a maximum $M$, $\ddot{v}(M) \leq 0$ and this gives $\frac{2}{h^2} \left( e^v -1 \right) \leq 0$, which implies $v \leq 0$. For the second part assume that $v (\delta) , \dot{v}(\delta)<0$, then one needs to prove that $v < 0$, for all $t \geq \delta$. Suppose not, then let $x > \delta$ be the smallest possible such that $v=0$. Since $v(\delta), \dot{v}(\delta) <0$ there must be a minimum $m \in (\delta , x)$. Applying the maximum principles just proved to conclude that $v(m) \geq 0$ and this contradicts the minimality of $x$.
\end{proof}

\begin{corollary}\label{NoSolutions}
There are no solutions to the ODE \ref{second} with $v(0)= \dot{v}(0)=0$ and $\lim_{r \rightarrow \infty} r^{-k} e^{v}=0$ for some $k \in \mathbb{Z}$, such that $\ddot{v}(0) = b >0$.
\end{corollary}
\begin{proof}
We argue by contradiction and suppose there is $v$ with $v(0)= \dot{v}(0)=0$ and $\ddot{v}(0) = b >0$. Then, there is $\delta>0$ such that $v(\delta), \dot{v}(\delta)$ are both positive, and by lemma \ref{maxmin}, $v>0$ in $(\delta, + \infty)$. Using the equation $\ddot{v}=\frac{2}{h^2} \left( e^v -1 \right) $ we see that $\ddot{v} >0$ in $(\delta, + \infty)$. Integrating this gives that
$$v(r) \geq v(\delta) + \dot{v}(\delta)(r- \delta),$$
for all $r \geq \delta$. Then $ r^{-k} e^{v} \geq r^{-k}e^{ v(\delta) + \dot{v}(\delta)(r- \delta)}$, and since $\dot{v}(\delta)$ is positive, for all $k \in \mathbb{Z}$ this diverges as $r \rightarrow + \infty$. This contradicts the hypothesis of the theorem.
\end{proof}

\begin{lemma}\label{vbounds}
Let $u,v, a : (\delta , \infty) \rightarrow \mathbb{R}$ be differentiable $u< 0$, such that
\begin{eqnarray}\nonumber
\ddot{v} - a v \geq 0 \ , \ \ddot{u} - a u =0.
\end{eqnarray}
If $u(\delta )=v(\delta)$ and $\dot{u}(\delta) = \dot{v}(\delta)$, then $v (r) \geq u(r)$ for all $r \geq \delta$.
\end{lemma}
\begin{proof}
Define $f= \frac{v}{u}$, since by assumption $u <0$ it is enough to prove that $f \leq 1$, for $r \geq \delta$ and that $f \geq 1$ for $r \leq \delta$. Moreover, since $f(\delta)=1$ it is enough to prove that $\dot{f} \leq 0$, i.e. that $\dot{v}u - v \dot{u} \leq 0$. Once again, our hypothesis dictate that at $r= \delta$ this expression vanishes and so it is enough to show that its derivative $\ddot{v}u - v \ddot{u}$ is nonpositive. Substituting $\ddot{u}=au$ and $\ddot{v} \geq av$ gives that indeed $\ddot{v} u - v \ddot{u} \leq 0$.
\end{proof}

\begin{proposition}\label{vbounds2}
Let $v$ be a solution of \ref{second} on $(\delta , \infty)$, with the initial conditions $v(\delta)=-k_2<0$ and $\dot{v}(\delta)=-k_1 <0$, for some positive constants $k_1,k_2$. Then, for $t \geq \delta$
$$v_b (r) \leq v(t) \leq v_u (r),$$
where $v_b(r)= -k_2 -k_1 (r - \delta) - 2\int_{\delta}^r \int_{\delta}^s h^{-2}(s') ds'  ds$, and $v_{u}(t)$ solves $\ddot{v_{u}}-\frac{2}{h^2}v_{u} =0$ with the initial conditions $v_u(\delta)= -k_2$, $\dot{v}_{u}(\delta)=-k_1$.
\end{proposition}
\begin{proof}
Since the function $F(v)= e^{v}$ is convex it lies above all its tangents, then $\ddot{v} = \frac{2}{h^2} (e^v - 1 ) \geq  \frac{2}{h^2} v$. The second step is using lemma \ref{vbounds} with $a=\frac{2}{h^2}$ and $u=v_b$ to obtain the lower bound. The upper bound comes from integrating $\ddot{v} \geq -\frac{2}{h^2}$, which holds since $e^v$ is positive.
\end{proof}

Insert $a^2 = e^{v}$ into the first monopole ODE in \ref{monopoleode2}, then
$$\dot{\phi} = \dfrac{1}{2 h^2} (e^v -1 ).$$
The above bounds on $v$ can be used to estimate the values of the Higgs field. However, in the following application a crude approach to these bounds will be given. Since $\ddot{v}(0) <0$, the maximum principle from lemma \ref{maxmin} guarantees $v \leq 0$ for all $r$. Moreover, the standard existence and uniqueness theorem applies locally at $r= \delta$ and the estimates in proposition \ref{vbounds2} guarantee the solution does not explode at any finite value of $r> \delta$ and so does extend up to $r=\infty$. One other application of the previous proposition is to compute

\begin{proposition}\label{ola1}
Let $(a,\phi)$ be a solution to the monopole ODE's \ref{monopoleode2}, then for all $t \in (\delta , \infty)$
$$\phi(\delta) \geq \phi(r) \geq \phi (\delta) - \int_{\delta}^r \frac{1}{2h^2(t)} dt .$$
So, if the Green's function $G (r) = - \int \frac{1}{2h^2(r)} dr$ is bounded at $\infty$ (which is the case for nonparabolic $g$), then so is the Higgs field.
\end{proposition}

This together with the fact that $\dot{\phi}(r) \rightarrow 0$ as $r \rightarrow \infty$ allows the conclusion that the limit $\phi(\infty) = \lim_{r \rightarrow \infty} \phi (r)$, exists and is finite. As an application one obtains

\begin{corollary}\label{ola}
Let $g$ be a spherically symmetric metric and $(A, \Phi) \in \mathcal{M}_{inv}$ an invariant monopole on $(\mathbb{R}^3, g)$. The norm of the Higgs field is dominated by the Green's function $G$. Moreover, if $G$ is bounded at infinity then the mass $m (A, \Phi)$ exists and is finite. Let $m \in \mathbb{R}^+$ and $\epsilon>0$, then for $\delta \leq \Delta(m, \epsilon)$, the monopole $(A^{\delta}_m, \Phi^{\delta}_m)$ satisfies
$$m(A^{\delta}_m, \Phi^{\delta}_m) \in \left[ \frac{1}{\delta} \left( m \coth(m) -1 - 2 \epsilon \right)  , \frac{1}{\delta} \left( m\coth(m) + 2 \epsilon  \right) + 2G(\delta) \right]. $$
\end{corollary}

\end{document}